\pdfoutput=1
\RequirePackage{ifpdf}
\ifpdf 
\documentclass[pdftex]{sigma}
\else
\documentclass{sigma}
\fi

\numberwithin{equation}{section}

\newtheorem{Theorem}{Theorem}[section]
\newtheorem*{Theorem*}{Theorem}
\newtheorem{Corollary}[Theorem]{Corollary}
\newtheorem{Lemma}[Theorem]{Lemma}
\newtheorem{Proposition}[Theorem]{Proposition}
 { \theoremstyle{definition}
\newtheorem{Definition}[Theorem]{Definition}

\newtheorem{Remark}[Theorem]{Remark} }

\usepackage{subcaption}
\usepackage[labelsep=period,labelfont=bf,font=small]{caption}
\usepackage{tikz-cd}
\usetikzlibrary{calc}

\newcommand{\brane}{\mathcal D}
\newcommand{\CC}{\mathbb C}
\newcommand{\CCP}{\mathbb C \mathrm{P}}
\newcommand{\RR}{\mathbb R}
\newcommand{\torus}{\mathbb T}
\newcommand{\MM}{\mathbb M}
\newcommand{\NN}{\mathbb N}
\newcommand{\G}{\mathcal G}
\newcommand{\hb}{\mathbf h}
\newcommand{\ub}{\mathbf u}
\newcommand{\V}{\text{\bf\textcolor{red}{/}}}
\newcommand{\U}{\text{\bf\textcolor{blue}{\textbackslash}}}
\newcommand{\W}{\mathbf W}
\newcommand{\prequot}{\widetilde \MM^s}
\newcommand{\cdloop}[1]{\ar[loop, distance = 2em, out=110, in=70, "#1"]}
\newcommand{\pt}{\mathrm{pt}}
\newcommand{\but}{\beta}
\newcommand{\site}{\mathfrak s}
\newcommand{\surgery}{\mathfrak S}
\newcommand{\bx}{\mathrm{bx}}

\def\ttt#1{\text{\tt #1}}
\newcommand\bs{{\color{blue} \char`\\}}
\newcommand\fs{{\color{red}/}}

\DeclareMathOperator{\bowvar}{\mathcal C}
\DeclareMathOperator{\Hom}{Hom}
\DeclareMathOperator{\End}{End}

\DeclareMathOperator{\im}{im}
\DeclareMathOperator{\charge}{charge}

\DeclareMathOperator{\pen}{\Gamma}

\begin{document}

\newcommand{\arXivNumber}{2310.04973}

\renewcommand{\PaperNumber}{016}

\FirstPageHeading

\ShortArticleName{Tangent Weights and Invariant Curves in Type A Bow Varieties}

\ArticleName{Tangent Weights and Invariant Curves\\ in Type A Bow Varieties}

\Author{Alexander FOSTER~$^{\rm a}$ and Yiyan SHOU~$^{\rm b}$}

\AuthorNameForHeading{A.~Foster and Y.~Shou}

\Address{$^{\rm a)}$~Department of Mathematics, University of North Carolina at Chapel Hill, NC, USA}
\EmailD{\href{mailto:aofoster@live.unc.edu}{aofoster@live.unc.edu}}

\Address{$^{\rm b)}$~Independent Researcher, Alexandria, VA, USA}
\EmailD{\href{mailto:yiyanshou@gmail.com}{yiyanshou@gmail.com}}

\ArticleDates{Received December 08, 2023, in final form February 20, 2025; Published online March 09, 2025}

\Abstract{This paper provides a complete classification of torus-invariant curves in Cherkis bow varieties of type~A. We develop combinatorial codes for compact and noncompact invariant curves involving the butterfly diagrams, Young diagrams, and binary contingency tables. As a key intermediate step, we also develop a novel tangent weight formula. Finally, we apply this new machinery to example bow varieties to demonstrate how to obtain their~1-skeletons (union of fixed points and invariant curves).}


\Classification{14H10; 05E14}

\section{Introduction}
For a broad class of complex algebraic varieties with a torus action $T\, \rotatebox[origin=c]{-90}{$\circlearrowright$}\, Y$, the fixed points and invariant curves under the torus action play an important role in the study of the equivariant cohomology of $Y$. Here, we define an invariant curve to be the closure of a 1-dimensional $T$-orbit. This paper extends the work of \cite{RS} on the case where $Y$ is a Cherkis bow variety of type A by providing combinatorial codes for invariant curves.

For simplicity, assume that $Y$ has finitely many fixed points. Under certain conditions, the map $\iota^*\colon H^*_T (Y) \to H^*_T \bigl(Y^T\bigr)$ induced by the inclusion of the fixed point locus is injective. In other words, each equivariant cohomology class is uniquely determined by its fixed point restrictions. This is the case for partial flag varieties, and more generally, Nakajima quiver varieties of type~A. As $H^*_T(\pt) = \CC[u_1, \dots , u_m]$, where $m = \dim(T)$, this allows us to describe equivariant cohomology classes as tuples of polynomials. This approach is taken in \cite{R} to study $\hbar$-deformed Schubert classes. In the setting where the variety is equivariantly formal with respect to the~$T$-action, we can describe the image of $\iota^*$ in terms of the fixed points and invariant curves. For instance, when~$Y$ possesses finitely many invariant curves, the result of~\cite{GKM} describes the image of $\iota^*$ in terms of simple matching conditions on the polynomials (see also~\cite{Ty}). These matching conditions are determined by the invariant curves and their tangent data. When there are infinitely many invariant curves, the Chang--Skjelbred lemma \cite{CS} provides a more complicated description (see also \cite{equiv_coh}). Unfortunately, not every Cherkis bow variety is equivariantly formal nor has injective restriction map \cite{yibo}. The appendix of \cite{BR} shows, however, that there is a natural subalgebra (generated by the stable envelope) of $H^*_T(Y)$ on which $\iota^*$ is injective. Determining whether this subalgebra can be described in terms of fixed points and invariant curves is an open problem.\looseness=-1

Besides providing a concrete way to describe equivariant cohomology classes, the fixed points and invariant curves play an important role in the study of stable envelopes \cite{AO, MO}. The stable envelope is an axiomatically defined $T$-equivariant characteristic class with important connections to quantum integrable systems, quantum groups, and quantum cohomology. In certain settings, the axioms can be formulated as conditions on fixed point restrictions that are entirely determined by the invariant curves and their tangent data. This formulation is, in a sense, local to the fixed points, in contrast to the original global axioms of \cite{AO, MO}. See \cite{RTV1, RTV2, RTV3} for a~discussion of local stable envelope axioms in the Schubert calculus setting and \cite[Section~7]{RS} for the Cherkis bow variety setting. In the Schubert calculus setting, the stable envelope agrees with the Chern--Schwartz--MacPherson classes of Schubert cells \cite{motivic_overview}. While we are focusing on equivariant cohomology here, this discussion generalizes to K-theory and elliptic cohomology. In the elliptic cohomology setting, the stable envelope plays an important role in the study of~3d~($N = 4$) mirror symmetry \cite{BR, fullflag_mirror, grassman_mirror, SZ}.

The work of \cite{BR} shows that Cherkis bow varieties \cite{C1, C2, C3, NT} form a natural pool of varieties where the phenomenon of mirror symmetry manifests itself in the combinatorics and elliptic stable envelope. Bow varieties are holomorphic symplectic manifolds that come equipped with a torus action and generalize Nakajima quiver varieties \cite{N_quiver}. In \cite{RS}, a combinatorial framework for studying the fixed points of bow varieties of type A is established. This paper completes the picture by providing a complete combinatorial description of the invariant curves. A construction for compact invariant curves involving so called ``butterfly surgeries'' was previously known to the second author (see~\cite{S}). However, whether that construction captured all compact invariant curves remained a conjecture. This work resolves that conjecture in the affirmative and extends the result by also capturing the noncompact invariant curves using modified ``butterfly surgery'' constructions.

While the main accomplishment of this paper is the complete classification of invariant curves, several intermediate developments are required, some of which are interesting in their own right. In Section~\ref{sec:background}, we review the combinatorial framework of \cite{RS} for the study of type A Cherkis bow varieties. In Section~\ref{sec:bct_weights}, we prove a novel and noncancellative formula for the tangent weights at a fixed point of a bow variety. This formula can be used to efficiently calculate tangent weights by hand or by computer and plays a pivotal role in the proof of our main result. In Section~\ref{sec:curves}, we review the essentials of torus invariant curves and give a self-contained presentation of the butterfly surgery construction of \cite{S}. This section also describes two novel constructions for noncompact invariant curves. Finally, we reformulate these constructions in terms of surgery operations on Young diagrams, resulting in a description of invariant curves in terms of familiar combinatorics. Section~\ref{sec:block_swap} further reformulates the combinatorics of Section~\ref{sec:curves} in terms of binary contingency tables (BCTs). The results of this section provide a simple algorithm for constructing all invariant curves in a bow variety. Section~\ref{sec:classification} contains the main result, a classification of invariant curves containing a given fixed point. Curves are characterized as one of three types, each having its own unique properties. In the final section, Section~\ref{sec:examples}, we~apply the machinery developed in the previous sections to the example bow varieties of \cite{RS}.

\section{Background} \label{sec:background}
We begin by recalling the essential pieces of the combinatorial framework of \cite{RS} for the study of bow varieties. We will require brane diagrams, tie diagrams, table-with-margins, butterfly diagrams, the formula for the tangent bundle of a bow variety, formulas for fixed point restrictions, and the Hanany--Witten transition.

\subsection{Brane diagrams and bow varieties}
The \cite{RS} framework is based on the \cite{NT} construction of bow varieties. Here, we provide a~simplified presentation. A \textit{brane diagram} $\brane$ consists of a horizontal line divided into segments by~$n$ lines of slope 1 and $m$ lines of slope $-1$. The lines must be arranged so as not to create a~triple (or greater) intersection. Additionally, each segment is decorated by a nonnegative integer multiplicity where the infinite segments on the far left and right are automatically decorated with~0. An example of a brane diagram is shown in Figure~\ref{fig:brane_dgm}. The lines of slope 1 are called ``NS5 branes'', and we label them from left to right as $V_1, \dots , V_n$. The lines of slope $-1$ are called ``D5 branes'', and we label them from left to right as $U_1, \dots , U_m$. The NS5 and D5 branes are referred to collectively as ``5-branes''. The segments along with their multiplicities are referred to as ``D3 branes''. The segments are labeled as $X_0, \dots , X_{n + m}$ and the multiplicity of a D3 brane $X$ is denoted by $d_X$. In particular, $d_{X_0} = d_{X_{n + m}} = 0$.

\begin{figure}[t]\vspace*{-11mm}
\centering
\begin{tikzpicture}[scale=.5]
\draw [thick,red] (0.5,0) --(1.5,2);
\node[red] at (.5,-.6) {$V_1$}; 
\draw[thick] (1.5,1)--(2.5,1) node [above] {$2$} -- (3.5,1);
\draw [thick,blue](4.5,0) --(3.5,2);
\draw [thick](4.5,1)--(5.5,1) node [above] {$2$} -- (6.5,1);
\draw [thick,red](6.5,0) -- (7.5,2); 
\node[red] at (6.5,-.6) {$V_2$};
\draw [thick](7.5,1) --(8.5,1) node [above] {$2$} -- (9.5,1);
\draw[thick,blue] (10.5,0) -- (9.5,2);
\draw[thick] (10.5,1) --(11.5,1) node [above] {$4$} -- (12.5,1);
\draw [thick,red](12.5,0) -- (13.5,2); 
\draw [thick](13.5,1) --(14.5,1) node [above] {$3$} -- (15.5,1);
\draw[thick,red] (15.5,0) -- (16.5,2); 
\draw [thick](16.5,1) --(17.5,1) node [above] {$3$} -- (18.5,1);
\draw [thick,red](18.5,0) -- (19.5,2); 
\draw [thick](19.5,1) --(20.5,1) node [above] {$4$} -- (21.5,1);
\draw [thick,blue](22.5,0) -- (21.5,2);
\node[blue] at (21.5,2.4) {$U_3$};
\draw [thick](22.5,1) --(23.5,1) node [above] {$3$} -- (24.5,1);
\draw[thick,red] (24.5,0) -- (25.5,2);
\node[blue] at (27.5,2.4) {$U_4$};
\node[blue] at (30.5,2.4) {$U_5$};
\draw[thick] (25.5,1) --(26.5,1) node [above] {$2$} -- (27.5,1);
\draw [thick,blue](28.5,0) -- (27.5,2); 
\draw [thick](28.5,1) --(29.5,1) node [above] {$2$} -- (30.5,1);
\draw [thick,blue](31.5,0) -- (30.5,2); 

\draw [black,dashed,->](7,-5) to [out=170,in=270] (0.3,-1);
\draw [black,dashed,->](7,-5) to [out=110,in=270] (6.3,-1);
\draw [black,dashed,->](7,-5) to [out=80,in=210] (12.3,-.3);
\draw [black,dashed,->](7,-5) to [out=70,in=210] (15.3,-.3);
\draw [black,dashed,->](7,-5) to [out=40,in=210] (18.3,-.3);
\draw [black,dashed,->](7,-5) to [out=20,in=220] (24.3,-.3);
\node at (7,-5.5) {NS5 branes};

\draw [black,dashed,->](23,6) to [out=190,in=60] (3.7,2.25);
\draw [black,dashed,->](23,6) to [out=230,in=80] (9.6,2.25);
\draw [black,dashed,->](23,6) to [out=250,in=90] (21.7,3.25);
\draw [black,dashed,->](23,6) to [out=270,in=100] (27.25,3.25);
\draw [black,dashed,->](23,6) to [out=280,in=110] (30,3.25);
\node at (25.3, 6) {D5 branes};

\node at (3,6) {D3 branes};
\draw [black,dashed,->](3,5) to [out=-90,in=90] (2,1.3);
\draw [black,dashed,->](3,5) to [out=-80,in=90] (5,1.3);
\draw [black,dashed,->](3,5) to [out=-70,in=100] (8,1.3);
\draw [black,dashed,->](3,5) to [out=-65,in=110] (11,1.3);
\node [color=black] at (3.7,4.7) {$\ldots$};
\end{tikzpicture}

\caption{The first example of a brane diagram given in \cite{RS}.}
\label{fig:brane_dgm}
\end{figure}
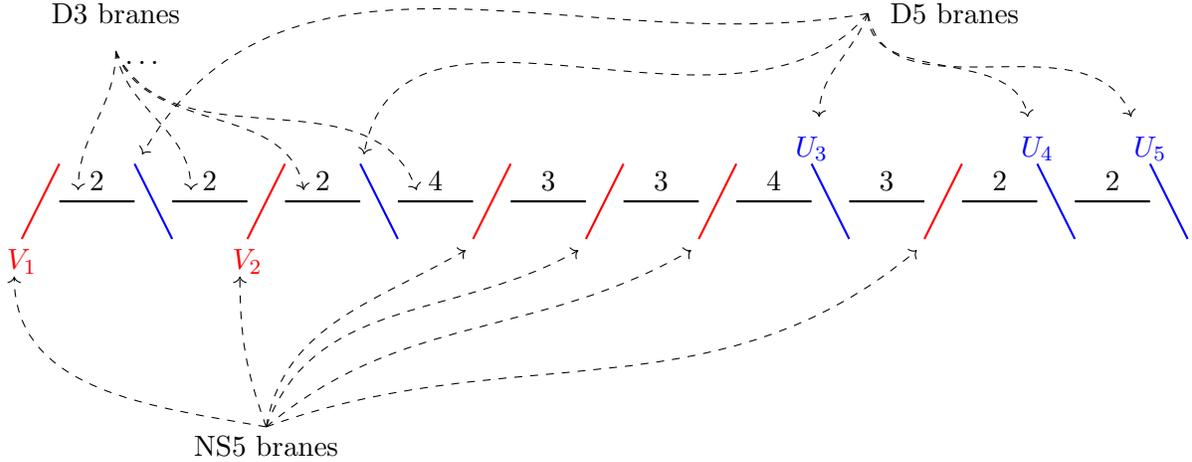

We do not bother displaying $X_0$ and $X_{n + m}$ in graphical representations of $\brane$. We also color the NS5 branes red and the D5 branes blue. These colors are purely for visual clarity and have no mathematical meaning.

To each brane diagram $\brane$, we associate a smooth, torus equivariant, symplectic holomorphic variety $\torus\, \rotatebox[origin=c]{-90}{$\circlearrowright$}\,\bowvar(\brane)$. This variety will be constructed as a GIT quotient. Given any brane $\beta$ (D5, NS5, or D3), denote the brane immediately to the left by $\beta^-$ and right by $\beta^+$.

First, a vector space will be associated to each 5-brane. For each D3 brane $X$, let $W_X = \CC^{d_X}$, and for each D5 brane $U$, let $\CC_U = \CC$.
\begin{itemize}\itemsep=0pt
\item To each D5 brane $U$, associate a ``triangle part''
\begin{align*}
\MM_U={}&\Hom(W_{U^+},W_{U^-}) \oplus
\Hom(W_{U^+},\CC_U) \oplus \Hom(\CC_U,W_{U^-}) \\
& \oplus\End(W_{U^-}) \oplus \End(W_{U^+}),
\end{align*}
whose elements will be denoted by $(A_U, b_U, a_U, B_U, B'_U)$, as shown in the diagram
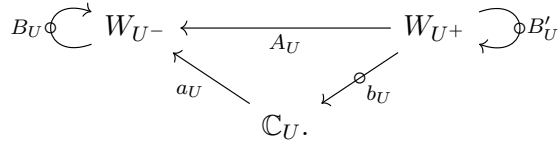
\[
\begin{tikzcd}
W_{U^-} \arrow[loop,out=200,in=160,distance=2em, "B_U", "\circ" marking] & & W_{U^+} \arrow[ld, "\circ" marking, "b_U" ] \arrow[ll, "A_U"] \arrow[loop,out=20,in=-20,distance=2em, "B'_U", "\circ" marking] \\
& \CC_U.\arrow[ul, "a_U"] &
\end{tikzcd}
\]

\item To each NS5 brane $V$, associate the ``two-way part''
\[
\MM_V= \Hom(W_{V^+},W_{V^-}) \oplus
\Hom(W_{V^-},W_{V^+}),
\]
whose elements will be denoted by $(C_V, D_V)$, as shown in the diagram
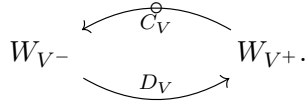
\[
\begin{tikzcd}
W_{V^-} \ar[rr,bend right, "D_V"] & & W_{V^+}. \ar[ll, "C_V", bend right, "\circ" marking]
\end{tikzcd}
\]
\end{itemize}
Finally, define
\[
\MM=\bigoplus_{U\text{ D5}}\MM_U \oplus \bigoplus_{V\text{ NS5}}\MM_V,
\]
the sum of all triangle and two-way parts.

Next, define the groups
\[
\torus = \CC_\hbar^\times \times \prod_{U \text{ D5}}\CC_U^\times \qquad\text{and}\qquad \G = \prod_{X \text{ D3}} \mathrm{GL}(W_X),
\]
where $\CC_\hbar^\times$ and $\CC^\times_U$ are copies of $\CC^\times$. The group $\G$ acts on each $W_X$ and hence on $\MM$ in a natural way. The torus $\torus$ acts on $\MM$ by scaling $\CC_U$ by the $\CC_U^\times$ factor and scaling each circled map in the above diagrams by the $\CC^\times_\hbar$ factor.

Let $\MM_0$ be the subset of $\MM$ satisfying the following conditions:
\begin{enumerate}\itemsep=0pt
\item[(1)] For each D5 brane $U$, we have $B_UA_U - A_UB'_U + a_Ub_U = 0$.
\item[(2)] For each D3 brane $X$, we have
\begin{itemize}\itemsep=0pt
\item $B'_{X^-} - B_{X^+} = 0$ if $X$ is in between two D5 branes ($\U-X-\U$),
\item $C_{X^+}D_{X^+} - D_{X^-}C_{X^-} = 0$ if $X$ is in between two NS5 branes ($\V-X-\V$),
\item $-D_{X^-}C_{X^-} - B_{X^+} = 0$ if $X^-$ is an NS5 brane and $X^+$ is a D5 brane ($\V-X-\U$),
\item $C_{X^+}D_{X^+} + B'_{X^-} = 0$ if $X^-$ is a D5 brane and $X^+$ is an NS5 brane ($\U-X-\V$).
\end{itemize}
\end{enumerate}
These conditions arise from setting a certain moment map equal to 0 (see \cite{NT} or \cite{RS} for details). Hence, we will refer to them as ``0-momentum'' conditions. Let $\W = \bigoplus_{X\text{ D3}} W_X$, and define three additional ``stability'' conditions:
\begin{enumerate}\itemsep=0pt
\item[(1)] For all D5 branes $U$, the only $B_{U^+}$-invariant subspace $S\subset W_{U^+}$ with $A_U(S) = 0$, $b_U(S) = 0$ is $S=0$.
\item[(2)] For all D5 branes $U$, the only $B_{U^-}$-invariant subspace $S\subset W_{U^-}$ with $\im(A_U)+\im(a_U) \subset S$ is $S=W_{U^-}$.
\item[(3)] The only subspace $S=\bigoplus_X S_X\subset \W$ invariant under all $A$, $B$, $C$, |$D$ maps such that $\im(a_U) \subset S$ and $A_U$ induces an isomorphism $W_{U^+}/S_{U^+}\to W_{U^-}/S_{U^-}$ for all D5 branes $U$ is $S=\W$.
\end{enumerate}
Stability conditions 1 and 2 are the S1 and S2 conditions of \cite[Section~2.2]{NT}, and stability condition 3 is the $\nu$2 condition of \cite[Section~2.4]{NT}. Let $\prequot$ be the subset of $\MM_0$ satisfying the three stability conditions. $\prequot$ is $\G$- and $\torus$-invariant.

Define the ``Cherkis bow variety'' corresponding to brane diagram $\brane$ as the quotient
$
\bowvar(\brane) = \prequot / \G$.
There is a residual action of $\torus$ on $\bowvar(\brane)$, along with a collection of $\torus$-equivariant ``tautological'' vector bundles $\xi_X \to \bowvar(\brane)$, associated with the $W_X$ vector spaces. The rank of $\xi_X$ is $d_X$. A formula for the $\torus$-equivariant K-theory class of the tangent bundle $T\bowvar(\brane)$ is given in \cite[Section~3.2]{RS}:
\begin{itemize}\itemsep=0pt
\item Let $\hb$ be the trivial line bundle whose fibres are scaled by the $\CC_\hbar^\times$ factor of $\torus$.
\item For a D5 brane $U$ define
\begin{align*}
T_U ={}&
\Hom(\xi_{U^+},\xi_{U^-})\oplus
\hb\Hom(\xi_{U^+},\CC_U) \oplus
\Hom(\CC_U,\xi_{U^-})
\\
& \oplus \hb \End(\xi_{U^-})\oplus \hb\End(\xi_{U^+}).
\end{align*}
\item
For an NS5 brane $V$ define
$
T_V =\hb\Hom(\xi_{V^+},\xi_{V^-}) \oplus \Hom(\xi_{V^-},\xi_{V^+})
$.
\end{itemize}
Then, as elements of $K^0_\torus (\bowvar(\brane))$, we have
\begin{align}
T\bowvar(\brane)={}&
\biggl(\bigoplus_{U \text{ D5}} T_U \biggr)\oplus
\biggl( \bigoplus_{V \text{ NS5}} T_V \biggr)\nonumber \\
&
\ominus \biggl( \bigoplus_{U \text{ D5}} \hb \Hom(\xi_{U^+},\xi_{U^-}) \biggr)
\ominus \biggl(\bigoplus_{X \text{ D3}} (1+\hb)\End(\xi_{X}) \biggr).\label{eqn:tangent_bundle}
\end{align}

\subsection{Torus fixed points} \label{sec:fixed_pts}
Paper \cite{RS} provides three combinatorial codes for the $\torus$-fixed points of a bow variety $\bowvar(\brane)$: tie diagrams, butterfly diagrams, and table-with-margins (BCTs). We will be utilizing all three in what follows.

Let us begin with tie diagrams. A ``tie diagram'' is obtained from a brane diagram by adding dashed lines, called ``ties'', between the 5-branes such that
\begin{itemize}\itemsep=0pt
\item each tie joins 5-branes of different type (one NS5 and one D5 brane),
\item for each segment $X$, there are $d_X$ distinct ties covering $X$.
\end{itemize}
We allow at most one tie between any two 5-branes. In Figure~\ref{fig:tie_dgm}, we added ties to the brane diagram in Figure~\ref{fig:brane_dgm} to create a tie diagram. This tie diagram corresponds to one of the $\torus$-fixed points.
\begin{figure}[t]\vspace{-3mm}
\centering
\begin{tikzpicture}[baseline=0,scale=.44]
\draw [thick,red] (0.5,0) --(1.5,2);
\draw[thick] (1,1)--(2.5,1) node [above] {$2$} -- (31,1);
\draw [thick,blue](4.5,0) --(3.5,2);
\draw [thick](4.5,1)--(5.5,1) node [above] {$2$} -- (6.5,1);
\draw [thick,red](6.5,0) -- (7.5,2); 
\draw [thick](7.5,1) --(8.5,1) node [above] {$2$} -- (9.5,1);
\draw[thick,blue] (10.5,0) -- (9.5,2);
\draw[thick] (10.5,1) --(11.5,1) node [above] {$4$} -- (12.5,1);
\draw [thick,red](12.5,0) -- (13.5,2); 
\draw [thick](13.5,1) --(14.5,1) node [above] {$3$} -- (15.5,1);
\draw[thick,red] (15.5,0) -- (16.5,2); 
\draw [thick](16.5,1) --(17.5,1) node [above] {$3$} -- (18.5,1);
\draw [thick,red](18.5,0) -- (19.5,2); 
\draw [thick](19.5,1) --(20.5,1) node [above] {$4$} -- (21.5,1);
\draw [thick,blue](22.5,0) -- (21.5,2);
\draw [thick](22.5,1) --(23.5,1) node [above] {$3$} -- (24.5,1);
\draw[thick,red] (24.5,0) -- (25.5,2);
\draw[thick] (25.5,1) --(26.5,1) node [above] {$2$} -- (27.5,1);
\draw [thick,blue](28.5,0) -- (27.5,2); 
\draw [thick](28.5,1) --(29.5,1) node [above] {$2$} -- (30.5,1);
\draw [thick,blue](31.5,0) -- (30.5,2); 

\draw [dashed, black](4.5,-.25) to [out=-45,in=225] (12.5,-.25);
\draw [dashed, black](10.5,-.25) to [out=-45,in=225] (12.5,-.25);
\draw [dashed, black](10.5,-.25) to [out=-45,in=225] (15.5,-.25);
\draw [dashed, black](10.5,-.25) to [out=-45,in=225] (24.5,-.25);
\draw [dashed, black](22.5,-.25) to [out=-45,in=225] (24.5,-.25);

\draw [dashed, black](1.5,2.25) to [out=45,in=-225] (3.5,2.25);
\draw [dashed, black](1.5,2.25) to [out=45,in=-225] (9.5,2.25);
\draw [dashed, black](13.5,2.25) to [out=45,in=-225] (21.5,2.25);
\draw [dashed, black](16.5,2.25) to [out=45,in=-225] (21.5,2.25);
\draw [dashed, black](19.5,2.25) to [out=45,in=-225] (30.5,2.25);
\draw [dashed, black](25.5,2.25) to [out=45,in=-225] (30.5,2.25);
\end{tikzpicture}
\vspace{-4mm}

\caption{An example of a tie diagram whose underlying brane diagram is the one displayed in Figure~\ref{fig:brane_dgm}.}
\label{fig:tie_dgm}
\vspace{-2mm}

\end{figure}
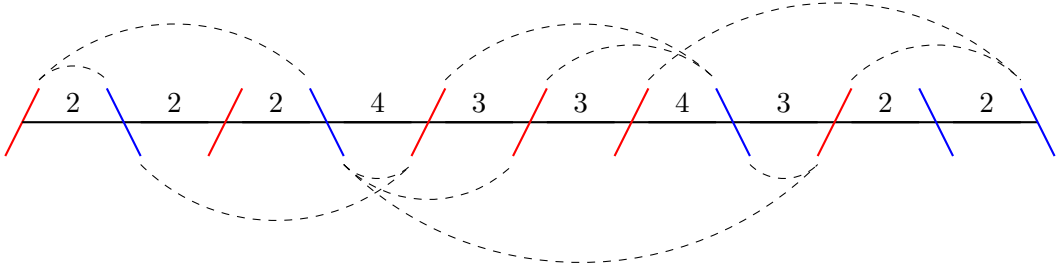
Next, we will create a correspondence between tie diagrams and binary contingency tables (BCTs). An intermediate notion will be required. An integer called the ``charge'' will be associated to each 5-brane.
For an NS5 brane $V$, let
\[
\charge(V)=
(d_{V^+}-d_{V^-})+\#\{\text{D5 branes left of $V$}\}.
\]
For a D5 brane $U$, let
\[
\charge(U)=
(d_{U^-}-d_{U^+})+\#\{\text{NS5 branes right of $U$}\}.
\]
Define the ``row margin'' and ``column margin'', respectively, by
\begin{align*}
r = (\charge(V_1), \dots , \charge(V_n)), \qquad
c = (\charge(U_1), \dots , \charge(U_m)).
\end{align*}
A ``binary contingency table'' (BCT) with row and column margins $r$, $c$ is a binary matrix with row sums equal to $r$ and columns sums equal to $c$. Note that the existence of such a matrix implies that $\sum_i r_i = \sum_j c_j$. Given a tie diagram, construct a BCT $M$ with margins $r$, $c$ as follows:
\begin{itemize}\itemsep=0pt
\item if $V_i$ is to the left of $U_j$, set $M_{ij} = 1$ if there is a tie between $V_i$ and $U_j$ and $M_{ij} = 0$ otherwise,
\item if $V_i$ is to the right of $U_j$, set $M_{ij} = 1$ if there is no tie between $V_i$ and $U_j$ and $M_{ij} = 0$ otherwise.
\end{itemize}
Assuming the underlying brane diagram $\brane$ is fixed, this gives a bijection between tie diagrams and BCTs with margins $r$, $c$. However, the map sending all tie diagrams (with arbitrary underlying brane diagram) to their corresponding BCT is not injective. An additional piece of data, called the ``separating line'', is required to make this map bijective. A separating line is a path running from the top left to the bottom right of an $(m + 1) \times (n + 1)$ lattice making only downward and rightward moves. We draw the separating line over the BCT, as depicted in Figure~\ref{fig:bct}. The separating line plays the role of encoding the order of the 5-branes. Downward moves correspond to NS5 branes and rightward moves to D5 branes. A BCT along with a separating line is called a ``table-with-margins''. Table-with-margins are in bijection with tie diagrams.\looseness=-1

\begin{figure}
\centering
\begin{tikzpicture}[baseline=1.4 cm, scale=.57]
\draw[ultra thin] (0,0) -- (5,0);
\draw[ultra thin] (0,1) -- (5,1);
\draw[ultra thin] (0,2) -- (5,2);
\draw[ultra thin] (0,3) -- (5,3);
\draw[ultra thin] (0,4) -- (5,4);
\draw[ultra thin] (0,5) -- (5,5);
\draw[ultra thin] (0,6) -- (5,6);
\draw[ultra thin] (0,0) -- (0,6);
\draw[ultra thin] (1,0) -- (1,6);
\draw[ultra thin] (2,0) -- (2,6);
\draw[ultra thin] (3,0) -- (3,6);
\draw[ultra thin] (4,0) -- (4,6);
\draw[ultra thin] (5,0) -- (5,6);
\draw[ultra thick] (0,6) -- (0,5) -- (1,5) -- (1,4) -- (2,4) -- (2,1) -- (3,1) -- (3,0) -- (5,0);
\node at (.5,6.4) {$5$}; \node at (1.5,6.4) {$2$}; \node at (2.5,6.4) {$2$}; \node at (3.5,6.4) {$0$}; \node at (4.5,6.4) {$2$};
\node at (.5,7.4) {$U_1$}; \node at (1.5,7.4) {$U_2$}; \node at (2.5,7.4) {$U_3$}; \node at (3.5,7.4) {$U_4$}; \node at (4.5,7.4) {$U_5$};
\node at (-.5,0.5) {$2$}; \node at (-.5,1.5) {$3$}; \node at (-.5,2.5) {$2$}; \node at (-.5,3.5) {$1$}; \node at (-.5,4.5) {$1$}; \node at (-.5,5.5) {$2$};
\node at (-1.8,0.5) {$V_6$}; \node at (-1.8,1.5) {$V_5$}; \node at (-1.8,2.5) {$V_4$}; \node at (-1.8,3.5) {$V_3$}; \node at (-1.8,4.5) {$V_2$}; \node at (-1.8,5.5) {$V_1$};
\node[violet] at (0.5,0.5) {$1$};\node[violet] at (1.5,0.5) {$0$};\node[violet] at (2.5,0.5) {$0$};\node[violet] at (3.5,0.5) {$0$};\node[violet] at (4.5,0.5) {$1$};
\node[violet] at (0.5,1.5) {$1$};\node[violet] at (1.5,1.5) {$1$};\node[violet] at (2.5,1.5) {$0$};\node[violet] at (3.5,1.5) {$0$};\node[violet] at (4.5,1.5) {$1$};
\node[violet] at (0.5,2.5) {$1$};\node[violet] at (1.5,2.5) {$0$};\node[violet] at (2.5,2.5) {$1$};\node[violet] at (3.5,2.5) {$0$};\node[violet] at (4.5,2.5) {$0$};
\node[violet] at (0.5,3.5) {$0$};\node[violet] at (1.5,3.5) {$0$};\node[violet] at (2.5,3.5) {$1$};\node[violet] at (3.5,3.5) {$0$};\node[violet] at (4.5,3.5) {$0$};
\node[violet] at (0.5,4.5) {$1$};\node[violet] at (1.5,4.5) {$0$};\node[violet] at (2.5,4.5) {$0$};\node[violet] at (3.5,4.5) {$0$};\node[violet] at (4.5,4.5) {$0$};
\node[violet] at (0.5,5.5) {$1$};\node[violet] at (1.5,5.5) {$1$};\node[violet] at (2.5,5.5) {$0$};\node[violet] at (3.5,5.5) {$0$};\node[violet] at (4.5,5.5) {$0$};
\end{tikzpicture}

\caption{The table-with-margins corresponding to the tie diagram in Figure~\ref{fig:tie_dgm}.}
\label{fig:bct}
\vspace{-2mm}
\end{figure}

\begin{Remark}
Paper \cite{NT} provides a more general construction of bow varieties of \emph{affine} type~A. In this more general setting, BCTs are replaced by the ``Maya diagrams'' of~\cite{N_satake}. The precise relationship between BCTs and Maya diagrams is discussed in an appendix of~\cite{RS}.
\end{Remark}

Finally, we will describe butterfly diagrams. These provide a depiction of a representative in the prequotient $\prequot$ of a $\torus$-fixed point of $\bowvar(\brane)$. Fix a tie diagram and a D5 brane $U$. For each segment $X$, let $d_X^U$ be the number of distinct ties connected to $U$ covering $X$. Place a column of $d_X^U$ vertices below $X$ with fixed spacing, and align them in the following way:
\begin{itemize}\itemsep=0pt
\item If the 5-brane between two consecutive columns is D5, align the columns at the bottom.
\item For columns to the right of $U$, if the 5-brane between two consecutive columns is NS5, align the columns at the top.
\item For columns to the left of $U$, if the 5-brane between two consecutive columns is NS5, align the columns so that the top vertex of the left column is one position lower than that of the right.
\end{itemize}
Also, place a special ``framing'' vertex (indicated by an open circle in our diagrams) below $U$. We will then join these vertices with directed edges:
\begin{itemize}\itemsep=0pt
\item If a column is adjacent to a D5 brane, create downward (black) edges between each consecutive pair of vertices in the column.
\item If the 5-brane between two adjacent columns is D5, create leftward (blue) edges between horizontally adjacent pairs of vertices.
\item If the 5-brane between two adjacent columns is NS5, create (magenta dotted) edges pointing left one position and down one position, wherever possible, and rightward (red) edges between horizontally adjacent pairs of vertices.
\item If $d_{U^-}^U > 0$, create an (green) edge from the framing vertex to the top vertex under $U^-$.
\item If $d_{U^+}^U > d_{U^-}^U$, create an (green) edge from vertex $d_{U^+}^U - d_{U^-}^U$ (counted from top to bottom) to the framing vertex.
\end{itemize}
The resulting directed graph is called a ``butterfly''. Taking the disjoint union of the butterflies of all the D5 branes gives us a ``butterfly diagram''. Figure~\ref{fig:butterfly_dgm} displays the butterfly diagram corresponding to the tie diagram in Figure~\ref{fig:tie_dgm}.

\begin{figure}[t]\vspace{-3mm}
\centering
\begin{tikzpicture}[scale=.45]
\draw [thick, red] (0.5,0) --(1.5,2);
\draw[thick] (1,1)--(2.5,1) node [above] {$2$} -- (31,1);
\draw [thick,blue](4.5,0) --(3.5,2);
\draw [thick](4.5,1)--(5.5,1) node [above] {$2$} -- (6.5,1);
\draw [thick,red](6.5,0) -- (7.5,2);
\draw [thick](7.5,1) --(8.5,1) node [above] {$2$} -- (9.5,1);
\draw[thick,blue] (10.5,0) -- (9.5,2);
\draw[thick] (10.5,1) --(11.5,1) node [above] {$4$} -- (12.5,1);
\draw [thick,red](12.5,0) -- (13.5,2);
\draw [thick](13.5,1) --(14.5,1) node [above] {$3$} -- (15.5,1);
\draw[thick,red] (15.5,0) -- (16.5,2);
\draw [thick](16.5,1) --(17.5,1) node [above] {$3$} -- (18.5,1);
\draw [thick,red](18.5,0) -- (19.5,2);
\draw [thick](19.5,1) --(20.5,1) node [above] {$4$} -- (21.5,1);
\draw [thick,blue](22.5,0) -- (21.5,2);
\draw [thick](22.5,1) --(23.5,1) node [above] {$3$} -- (24.5,1);
\draw[thick,red] (24.5,0) -- (25.5,2);
\draw[thick] (25.5,1) --(26.5,1) node [above] {$2$} -- (27.5,1);
\draw [thick,blue](28.5,0) -- (27.5,2);
\draw [thick](28.5,1) --(29.5,1) node [above] {$2$} -- (30.5,1);
\draw [thick,blue](31.5,0) -- (30.5,2);

\draw [dashed, black](4.5,-.25) to [out=-45,in=225] (12.5,-.25);
\draw [dashed, black](10.5,-.25) to [out=-45,in=225] (12.5,-.25);
\draw [dashed, black](10.5,-.25) to [out=-45,in=225] (15.5,-.25);
\draw [dashed, black](10.5,-.25) to [out=-45,in=225] (24.5,-.25);
\draw [dashed, black](22.5,-.25) to [out=-45,in=225] (24.5,-.25);

\draw [dashed, black](1.5,2.25) to [out=45,in=-225] (3.5,2.25);
\draw [dashed, black](1.5,2.25) to [out=45,in=-225] (9.5,2.25);
\draw [dashed, black](13.5,2.25) to [out=45,in=-225] (21.5,2.25);
\draw [dashed, black](16.5,2.25) to [out=45,in=-225] (21.5,2.25);
\draw [dashed, black](19.5,2.25) to [out=45,in=-225] (30.5,2.25);
\draw [dashed, black](25.5,2.25) to [out=45,in=-225] (30.5,2.25);

\draw[fill] (2.5,-4) circle [radius=.04]; \draw[fill] (5.5,-4) circle [radius=.04]; \draw[fill] (8.5,-4) circle [radius=.04]; \draw[fill] (11.5,-4) circle [radius=.04];
\draw[blue, <-] (2.7,-4)--(5.3,-4); \draw[red, ->] (5.7,-4)--(8.3,-4); \draw[blue, <-] (8.7,-4)--(11.3,-4);
\draw[<-,green] (2.5,-3.8) to [out=70,in=-120] (4,-3.1); \draw[] (4,-3) circle [radius=.1];

\draw[fill] (11.5,-5) circle [radius=.04]; \draw[fill] (14.5,-5) circle [radius=.04]; \draw[fill] (17.5,-5) circle [radius=.04]; \draw[fill] (20.5,-5) circle [radius=.04]; \draw[fill] (23.5,-5) circle [radius=.04];
\draw[red, ->] (11.7,-5)--(14.3,-5);\draw[red, ->] (14.7,-5)--(17.3,-5);\draw[red, ->] (17.7,-5)--(20.3,-5);\draw[blue,<-] (20.7,-5)--(23.3,-5);
\draw[fill] (11.5,-6) circle [radius=.04]; \draw[fill] (14.5,-6) circle [radius=.04];
\draw[red, ->] (11.7,-6)--(14.3,-6);
\draw[fill] (11.5,-7) circle [radius=.04];
\draw [magenta, dotted,->] (14.25,-5.1) -- (11.8,-5.9);\draw [magenta, dotted,->] (17.25,-5.1) -- (14.8,-5.9);\draw [magenta, dotted,->] (14.25,-6.1) -- (11.8,-6.9);
\draw[blue,<-] (8.7,-7)--(11.3,-7);
\draw[fill] (8.5,-7) circle [radius=.04];\draw[fill] (5.5,-8) circle [radius=.04];\draw[fill] (2.5,-8) circle [radius=.04];
\draw [magenta, dotted,->] (8.25,-7.1) -- (5.8,-7.9);\draw [blue, ->] (5.25,-8) -- (2.8,-8);
\draw [->] (11.5,-6.1) -- (11.5,-6.9);\draw [->] (11.5,-5.1) -- (11.5,-5.9);
\draw[<-,green] (8.5,-6.8) to [out=70,in=-120] (9.9,-5.5); \draw[<-,green] (10.1,-5.5) to [out=-50,in=160] (11.5,-6); \draw[] (10,-5.5) circle [radius=.1];

\draw[fill] (17.5,-8) circle [radius=.04]; \draw[fill] (20.5,-8) circle [radius=.04]; \draw[fill] (23.5,-8) circle [radius=.04];\draw[fill] (20.5,-7) circle [radius=.04];\draw[fill] (17.5,-9) circle [radius=.04];\draw[fill] (14.5,-9) circle [radius=.04];
\draw[red, ->] (14.7,-9)--(17.3,-9);\draw[red, ->] (17.7,-8)--(20.3,-8);\draw [blue, ->] (23.25,-8) -- (20.8,-8);
\draw [magenta, dotted,->] (20.25,-7.1) -- (17.8,-7.9);\draw [magenta, dotted,->] (20.25,-8.1) -- (17.8,-8.9);\draw [magenta, dotted,->] (17.25,-8.1) -- (14.8,-8.9);
\draw [->] (20.5,-7.1) -- (20.5,-7.9);
\draw[<-,green] (20.5,-6.8) to [out=70,in=-120] (21.9,-5.9);\draw[] (21.9,-5.8) circle [radius=.1];

\draw[fill] (20.5,-10) circle [radius=.04];\draw[fill] (23.5,-10) circle [radius=.04];\draw[fill] (26.5,-10) circle [radius=.04];\draw[fill] (29.5,-10) circle [radius=.04];\draw[fill] (26.5,-9) circle [radius=.04];\draw[fill] (29.5,-9) circle [radius=.04];
\draw [->] (29.5,-9.1) -- (29.5,-9.9);\draw [->] (26.5,-9.1) -- (26.5,-9.9);
\draw [blue, ->] (29.25,-10) -- (26.8,-10);\draw [blue, ->] (29.25,-9) -- (26.8,-9);\draw [blue, ->] (23.25,-10) -- (20.8,-10);
\draw [magenta, dotted,->] (26.25,-9.1) -- (23.8,-9.9);
\draw[red, ->] (23.7,-10)--(26.3,-10);
\draw[<-,green] (29.5,-8.8) to [out=70,in=-120] (30.9,-7.9);\draw[] (30.9,-7.8) circle [radius=.1];
\end{tikzpicture}

\caption{The tie diagram from Figure~\ref{fig:tie_dgm} along with its corresponding butterfly diagram.}
\label{fig:butterfly_dgm}
\end{figure}
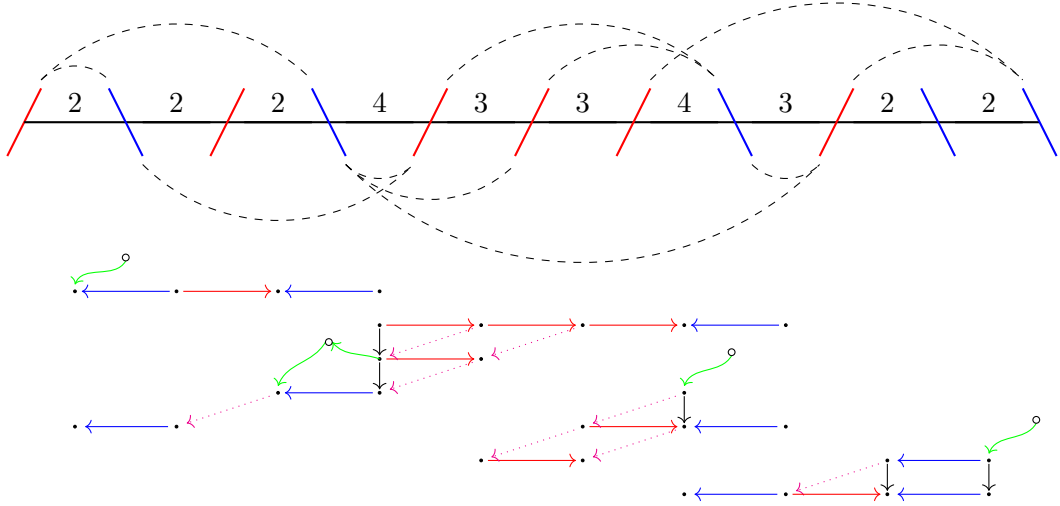

For each segment $X$, pick a basis for $W_X$ and identify the basis vectors with the vertices of the butterfly diagram below $X$. For each D5 brane $U$, identify $\CC_U$ with the framing vertex of the butterfly of $U$. Then, interpret the edges as $A$, $-B$, $C$, $D$, $a$, $-b$ maps. Note that when $X^-$ and $X^+$ are both D5 branes, there are two $B$ maps acting on $W_X$, one from the triangle part of $X^-$ and one from the triangle part of $X^+$. Considering the $0$-momentum condition, we take these maps to be equal, both represented by the black edges below $X$. The maps determined by the butterfly diagram constitute an element $\tilde{p}\in\prequot$. Moreover, this element descends to a~$\torus$-fixed point of $\bowvar(\brane)$. The converse is also true, that is, every $\torus$-fixed point of $\bowvar(\brane)$ can be represented by an element of $\prequot$ determined by some butterfly diagram. This is the content of~\cite[Theorem~4.8]{RS}.

\subsection{Fixed point restrictions} \label{sec:fixed_res}
Recall that a bow variety $\bowvar(\brane)$ comes equipped with tautological bundles $\xi_X$, one for each segment. Let $p\in \bowvar(\brane)^\torus$, and consider the corresponding butterfly diagram. Each non-framing vertex will be assigned a ``height'' relative to the butterfly it belongs to.

\begin{Definition} \label{def:height}
Fix a butterfly and define the ``height'' of each of its vertices as follows:
\begin{itemize}\itemsep=0pt
\item If the tail of a green edge is the framing vertex, the head is a vertex of height $0$.
\item If the head of a green edge is the framing vertex, the tail is a vertex of height $1$.
\item The height decreases by $1$ when we go from one vertex to another one immediately below it, and is constant across horizontal rows of vertices.
\end{itemize}
Given any vertex $v$ in a butterfly diagram, denote its height by $y(v)$. Note that the height of $v$ is measured with respect to the butterfly containing $v$.
\end{Definition}

Let $\ub$ be the trivial line bundle over $p$ whose $\torus$-action is given by scaling by the $\CC_U$ factor. Similarly, let $\hb$ be the trivial line bundle scaled by $\CC_\hbar$. Decorate each vertex $v$ of the $U$ butterfly by $\ub \hb^{y(v)}\in K^0_\torus(p)$. Figure~\ref{fig:butterfly_mons} shows these decorations for the butterfly diagram of Figure~\ref{fig:butterfly_dgm}. Then, we have that $\xi_X|_p$ is equal the sum of the line bundles decorating the vertices of the butterfly diagram below $X$ as elements of $K^0_\torus(p)$. For example, in Figure~\ref{fig:butterfly_mons}, we have
\[
\xi_{X_7} = \ub_2\hb^2 + \ub_3 + \frac{\ub_3}{\hb} + \frac{\ub_4}{\hb}.
\]

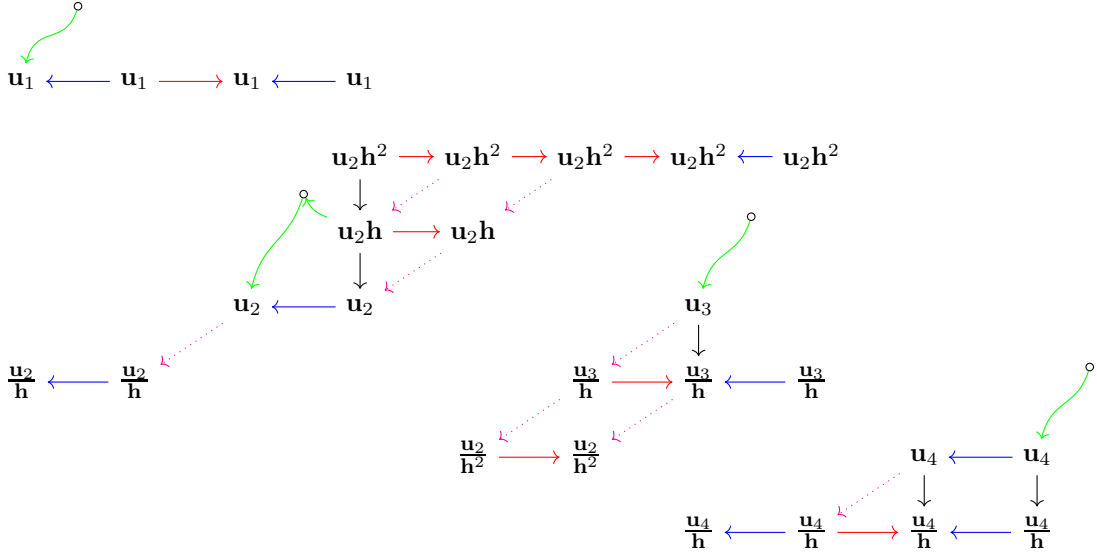
\begin{figure}[t]
\centering
\begin{tikzpicture}[scale=.5, yscale=2]
\draw (2.5,-4) node(1a1){\small $\ub_1$}; \draw[fill] (5.5,-4) node(2a1){\small $\ub_1$} (2a1); \draw[fill] (8.5,-4) node(3a1){\small $\ub_1$};
\draw[fill] (11.5,-4) node(4a1){\small $\ub_1$};
\draw[blue, <-] (1a1)--(2a1); \draw[red, ->] (2a1)--(3a1); \draw[blue, <-] (3a1)--(4a1);
\draw (4, -3) node [draw, circle, inner sep = 0, minimum size = 0.1cm] (u1) {}; \draw[<-,green] (1a1) to [out=70,in=-120] (u1);

\draw[fill] (11.5,-5) node(4b1){\small $\ub_2\hb^{2}$}; \draw[fill] (14.5,-5) node(5b1){\small $\ub_2\hb^{2}$};
\draw[fill] (17.5,-5) node(6b1){\small $\ub_2\hb^2$}; \draw[fill] (20.5,-5) node(7b1){\small $\ub_2\hb^2$};
\draw[fill] (23.5,-5) node(8b1){\small $\ub_2\hb^2$};
\draw[red, ->] (4b1)--(5b1);\draw[red, ->] (5b1)--(6b1);\draw[red, ->] (6b1)--(7b1);\draw[blue,<-] (7b1)--(8b1);
\draw[fill] (11.5,-6) node(4b2){\small $\ub_2\hb$}; \draw[fill] (14.5,-6) node(5b2){\small $\ub_2\hb$};
\draw[red, ->] (4b2)--(5b2);
\draw[fill] (11.5,-7) node(4b3){\small $\ub_2$};
\draw [magenta, dotted,->] (5b1) -- (4b2);\draw [magenta, dotted,->] (6b1) -- (5b2);\draw [magenta, dotted,->] (5b2) -- (4b3);
\draw[fill] (8.5,-7) node(3b1){\small $\ub_2$};\draw[fill] (5.5,-8) node(2b1){$\frac{\ub_2}{\hb}$};
\draw[fill] (2.5,-8) node(1b1){$\frac{\ub_2}{\hb}$};
\draw[blue,<-] (3b1)--(4b3);\draw [magenta, dotted,->] (3b1) -- (2b1);\draw [blue, ->] (2b1) -- (1b1);
\draw [->] (4b1) -- (4b2);\draw [->] (4b2) -- (4b3);
\draw[] (10,-5.5) node [draw, circle, inner sep = 0, minimum size = 0.1cm] (u2) {};
\draw[<-,green] (3b1) to [out=70,in=-120] (u2); \draw[<-,green] (u2) to [out=-50,in=160] (4b2);

\draw[fill] (17.5,-8) node(2c1){$\frac{\ub_3}{\hb}$}; \draw[fill] (20.5,-8) node(3c2){$\frac{\ub_3}{\hb}$};
\draw[fill] (23.5,-8) node(4c1){$\frac{\ub_3}{\hb}$};
\draw[fill] (20.5,-7) node(3c1){\small $\ub_3$};
\draw[fill] (17.5,-9) node(2c2){$\frac{\ub_2}{\hb^2}$};\draw[fill] (14.5,-9) node(1c1){$\frac{\ub_2}{\hb^2}$};
\draw[red, ->] (1c1)--(2c2);\draw[red, ->] (2c1)--(3c2);\draw [blue, <-] (3c2) -- (4c1);
\draw [magenta, dotted,->] (2c1) -- (1c1);\draw [magenta, dotted,->] (3c1) -- (2c1);\draw [magenta, dotted,->] (3c2) -- (2c2);
\draw [->] (3c1) -- (3c2);
\draw[] (21.9,-5.8) node [draw, circle, inner sep = 0, minimum size = 0.1cm] (u3) {};\draw[<-,green] (3c1) to [out=70,in=-120] (u3);

\draw[fill] (20.5,-10) node(1d1){$\frac{\ub_4}{\hb}$};\draw[fill] (23.5,-10) node(2d1){$\frac{\ub_4}{\hb}$};
\draw[fill] (26.5,-10) node(3d2){$\frac{\ub_4}{\hb}$};\draw[fill] (29.5,-10) node(4d2){$\frac{\ub_4}{\hb}$};
\draw[fill] (26.5,-9) node(3d1){\small $\ub_4$};\draw[fill] (29.5,-9) node(4d1){\small $\ub_4$};
\draw [->] (3d1) -- (3d2);\draw [->] (4d1) -- (4d2);
\draw [blue, ->] (4d1) -- (3d1);\draw [blue, ->] (4d2) -- (3d2);\draw [blue, ->] (2d1) -- (1d1);
\draw [magenta, dotted,->] (3d1) -- (2d1);
\draw[red, ->] (2d1)--(3d2);
\draw[] (30.9,-7.8) node [draw, circle, inner sep = 0, minimum size = 0.1cm] (u4) {};\draw[<-,green] (4d1) to [out=70,in=-120] (u4);
\end{tikzpicture}

\caption{The butterfly diagram of Figure~\ref{fig:butterfly_dgm} decorated by elements of $K^0_\torus(p)$.}
\label{fig:butterfly_mons}
\end{figure}

Of central importance in this work is the $\torus$-equivariant K-theory class of $T_p\bowvar(\brane)$. To express this class in terms of $\ub$ and $\hb$ bundles, one simply takes the tangent bundle formula \eqref{eqn:tangent_bundle} and substitutes each tautological bundle with the formula for its fixed point restriction calculated as above. Then, one can apply bilinearity of $\Hom$ and the isomorphism $\Hom(\alpha, \beta)\cong \beta\otimes\alpha^*$ to expand this expression into a sum of line bundles. In our running example, after a great deal of cancellations, we obtain
\begin{align}
T_p\bowvar(\brane) ={}& \frac{\ub_2\hb^4}{\ub_5} + \frac{\ub_2\hb^4}{\ub_3} + \frac{\ub_1\hb^3}{\ub_3} + \frac{\ub_2\hb^3}{\ub_3}
+ \frac{\ub_4\hb^2}{\ub_5} + \frac{\ub_3\hb}{\ub_5} + \frac{\ub_4\hb}{\ub_5} + \frac{\ub_5\hb}{\ub_3}\nonumber \\
& + \frac{\ub_3}{\ub_5}
+ \frac{\ub_5}{\ub_4} + \frac{\ub_5}{\ub_4\hb} + \frac{\ub_3}{\ub_2\hb^2} + \frac{\ub_3}{\ub_1\hb^2} + \frac{\ub_3}{\ub_2\hb^3}
+ \frac{\ub_5}{\ub_2\hb^3}.\label{eqn:cancel}
\end{align}
Each line bundle in this expansion is called a ``tangent weight at $p$''. In other words, the tangent weights at $p$ are the K-theoretic Chern roots of $T_p\bowvar(\brane)$.

\subsection{Hanany--Witten transition and separated bow varieties} \label{sec:HW}

Let $\brane$ contain a D5 brane $U$ followed immediately by an NS5 brane $V$. Let $\widetilde\brane$ agree with $\brane$ except at $U$ and $V$, where they differ as in the picture
\[
\begin{tikzpicture}[baseline=(current bounding box.center), scale=.5]
\draw[thick] (0,1)--(1,1) node [above] {$d_1$} -- (2,1);
\draw[thick,blue] (3,0)--(2,2);
\draw[thick] (3,1)--(4,1) node [above] {$d_2$}--(5,1);
\draw[thick,red] (5,0)--(6,2);
\draw[thick] (6,1)--(7,1) node [above] {$d_3$}--(8,1);
\node at (2.7,-.45) {$U$}; \node at (5.3,-.45) {$V$};
\node at (4,.3) {$X$};
\draw[thick] (15,1)--(16,1) node [above] {$d_1$} -- (17,1);
\draw[thick,red] (17,0)--(18,2);
\draw[thick] (18,1)--(19,1) node [above] {$\widetilde{d}_2$}--(20,1);
\draw[thick,blue] (21,0)--(20,2);
\draw[thick] (21,1)--(22,1) node [above] {$d_3$}--(23,1);
\node at (20.7,-.4) {$\widetilde{U}$}; \node at (17.3,-.45) {$\widetilde{V}$};
\node at (19,.3) {$\widetilde{X}$};
\end{tikzpicture},
\]
where $d_2 + \widetilde d_2 = d_1 + d_3 + 1$. Then, there is a natural holomorphic symplectic isomorphism \smash{$\bowvar(\brane) \overset{\cong}{\to} \bowvar(\widetilde\brane)$} called a ``Hanany--Witten transition''. In other words, Hanany--Witten transitions allow one to exchange two consecutive 5-branes of different type (cf.\ \cite[Section~8]{RS}) at the expense of changing the multiplicity of the D3 brane in between. The Hanany--Witten transition is equivariant with respect to the reparametrization $\widetilde \ub = \ub\hb$ (the other factors of the torus are not affected).

The Hanany--Witten transition induces a bijection of $\torus$-fixed points. This bijection can best be understood through transformations of tie diagrams and tables-with-margins. The action of the Hanany--Witten transition on tie diagrams is described by Figure~\ref{fig:HW_tie}. In words, switch $U$ and $V$ while keeping all ties attached. If there was a tie between $U$ and $V$, remove it, and if there was no such tie, create one.
\begin{figure}[t]
\centering
\begin{tikzpicture}[scale=.5]
\draw [thick] (0,1) --(2,1);
\draw [thick, blue] (3,0) -- (2,2);
\draw[thick] (3,1)--(5,1);
\draw[thick,red] (5,0)--(6,2);
\draw[thick] (6,1)--(8,1);
\draw [dashed](2,2) to [out=120,in=0] (0,3) ;
\draw [dashed](2,2) to [out=120,in=0] (0,3.25) node [left] {$A$} ;
\draw [dashed](2,2) to [out=120,in=0] (0,3.4) ;
\draw [dashed](5,0) to [out=240,in=0] (0,-2) ;
\draw [dashed](5,0) to [out=240,in=0] (0,-2.25) node [left] {$B$} ;
\draw [dashed](5,0) to [out=240,in=0] (0,-2.4);
\draw [dashed](3,0) to [out=-60,in=180] (8,-2);
\draw [dashed](3,0) to [out=-60,in=180] (8,-2.25) node [right]{$C$};
\draw [dashed](3,0) to [out=-60,in=180] (8,-2.4) ;
\draw [dashed](6,2) to [out=60,in=180] (8,3) ;
\draw [dashed](6,2) to [out=60,in=180] (8,3.25) node [right]{$D$};
\draw [dashed](6,2) to [out=60,in=180] (8,3.4);
\draw[dashed] (3,0) to [out=-40,in=-140](5,0);
\node at (4,0.25) {$E$};
\end{tikzpicture}
\qquad
\begin{tikzpicture}[scale=.5, baseline=-30pt]
\draw[ultra thick, <->] (0,1)--(2,1) node[above]{HW transition} -- (4,1);
\draw[ultra thick, <->] (0,1)--(2,1) node[below]{on fixpoints} -- (4,1);
\end{tikzpicture}
\qquad
\begin{tikzpicture}[scale=.5]
\draw [thick] (0,1) --(2,1);
\draw [thick, red] (2,0) -- (3,2);
\draw[thick] (3,1)--(5,1);
\draw[thick, blue] (6,0)--(5,2);
\draw[thick] (6,1)--(8,1);
\draw [dashed](2,0) to [out=240,in=0] (0,-.7);
\draw [dashed](2,0) to [out=240,in=0] (0,-.9) node [left] {$B$};
\draw [dashed](2,0) to [out=240,in=0] (0,-1.1) ;
\draw [dashed](5,2) to [out=120,in=0] (0,4) ;
\draw [dashed](5,2) to [out=120,in=0] (0,4.25) node [left] {$A$};
\draw [dashed](5,2) to [out=120,in=0] (0,4.4) ;
\draw [dashed](3,2) to [out=60,in=180] (8,4);
\draw [dashed](3,2) to [out=60,in=180] (8,4.25) node [right]{$D$};
\draw [dashed](3,2) to [out=60,in=180] (8,4.4) ;
\draw [dashed](6,0) to [out=-60,in=180] (8,-.7) ;
\draw [dashed](6,0) to [out=-60,in=180] (8,-.9)node [right]{$C$};
\draw [dashed](6,0) to [out=-60,in=180] (8,-1.1);
\draw[dashed] (3,2) to [out=40,in=140](5,2);
\node at (4,1.8) {$\lnot E$};
\end{tikzpicture}

\caption{The action of Hanany--Witten transition on tie diagrams. The $E$ and $\neg E$ symbols mean that the tie labeled $E$ is present if and only if the tie labeled $\neg E$ is not.}\label{fig:HW_tie}
\end{figure}
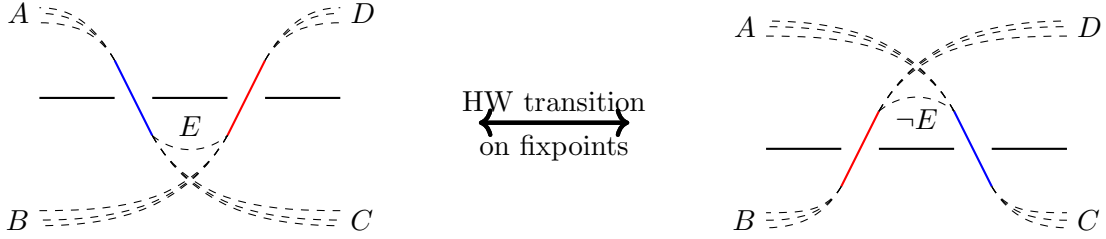
The Hanany--Witten transition acts on tables-with-margins by perturbing the separating line while leaving the BCT fixed.

By a sequence of Hanany--Witten transitions, one may transform any brane diagram into a~brane diagram with all NS5 branes on the left and all D5 branes on the right. Such a brane diagram will be called ``separated'', as will the associated bow variety. Figure~\ref{fig:sep_tie} depicts the result of applying this procedure to separate the tie diagram of Figure~\ref{fig:tie_dgm}.
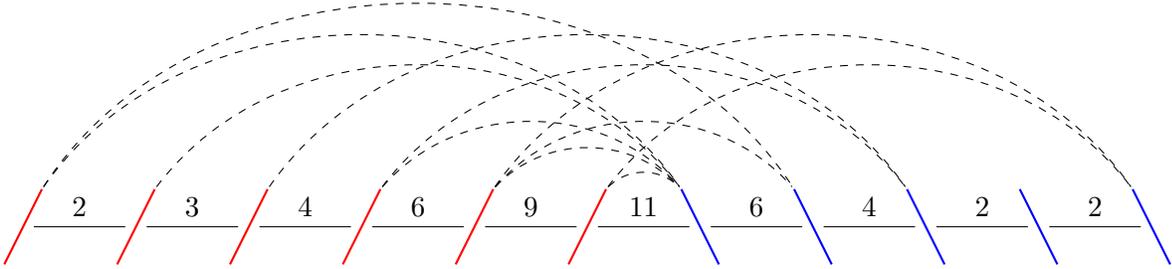
\begin{figure}\centering\vspace{-3mm}
\begin{tikzpicture}[scale = 0.5]
\draw[thick,red] (-0.5,-1)--(0.5,1) node[pos=0.5](V1){} node[pos=1,inner sep=0](V1t){} node[pos=0,inner sep=0](V1b){};
\draw[thick,red] (2.5,-1)--(3.5,1) node[pos=0.5](V2){} node[pos=1,inner sep=0](V2t){} node[pos=0,inner sep=0](V2b){};
\draw[thick,red] (5.5,-1)--(6.5,1) node[pos=0.5](V3){} node[pos=1,inner sep=0](V3t){} node[pos=0,inner sep=0](V3b){};
\draw[thick,red] (8.5,-1)--(9.5,1) node[pos=0.5](V4){} node[pos=1,inner sep=0](V4t){} node[pos=0,inner sep=0](V4b){};
\draw[thick,red] (11.5,-1)--(12.5,1) node[pos=0.5](V5){} node[pos=1,inner sep=0](V5t){} node[pos=0,inner sep=0](V5b){};
\draw[thick,red] (14.5,-1)--(15.5,1) node[pos=0.5](V6){} node[pos=1,inner sep=0](V6t){} node[pos=0,inner sep=0](V6b){};
\draw[thick,blue] (18.5,-1)--(17.5,1) node[pos=0.5](U1){} node[pos=1,inner sep=0](U1t){} node[pos=0,inner sep=0](U1b){};
\draw[thick,blue] (21.5,-1)--(20.5,1) node[pos=0.5](U2){} node[pos=1,inner sep=0](U2t){} node[pos=0,inner sep=0](U2b){};
\draw[thick,blue] (24.5,-1)--(23.5,1) node[pos=0.5](U3){} node[pos=1,inner sep=0](U3t){} node[pos=0,inner sep=0](U3b){};
\draw[thick,blue] (27.5,-1)--(26.5,1) node[pos=0.5](U4){} node[pos=1,inner sep=0](U4t){} node[pos=0,inner sep=0](U4b){};
\draw[thick,blue] (30.5,-1)--(29.5,1) node[pos=0.5](U5){} node[pos=1,inner sep=0](U5t){} node[pos=0,inner sep=0](U5b){};

\draw (V1)--(V2) node[pos=0.5, above](X1){2};
\draw (V2)--(V3) node[pos=0.5, above](X2){3};
\draw (V3)--(V4) node[pos=0.5, above](X3){4};
\draw (V4)--(V5) node[pos=0.5, above](X4){6};
\draw (V5)--(V6) node[pos=0.5, above](X5){9};
\draw (V6)--(U1) node[pos=0.5, above](X6){11};
\draw (U1)--(U2) node[pos=0.5, above](X7){6};
\draw (U2)--(U3) node[pos=0.5, above](X8){4};
\draw (U3)--(U4) node[pos=0.5, above](X9){2};
\draw (U4)--(U5) node[pos=0.5, above](X10){2};

\draw[dashed] (V1t) to [out=55,in=125] (U1t);
\draw[dashed] (V1t) to [out=57,in=123] (U2t);
\draw[dashed] (V3t) to [out=55,in=125] (U3t);
\draw[dashed] (V4t) to [out=53,in=127] (U3t);
\draw[dashed] (V5t) to [out=55,in=125] (U5t);
\draw[dashed] (V6t) to [out=53,in=127] (U5t);
\draw[dashed] (V2t) to [out=53,in=127] (U1t);
\draw[dashed] (V4t) to [out=49,in=131] (U1t);
\draw[dashed] (V5t) to [out=49,in=131] (U2t);
\draw[dashed] (V5t) to [out=47,in=133] (U1t);
\draw[dashed] (V6t) to [out=45,in=135] (U1t);
\end{tikzpicture}

\caption{The tie diagram of Figure~\ref{fig:tie_dgm} put into separated form using Hanany--Witten transitions.}
\label{fig:sep_tie}
\end{figure}
Separated bow varieties will be the focus of the remainder of this work. Any results derived can be translated to non-separated bow varieties by making a suitable reparametrization of $\torus$. In practical terms, this amounts to adjusting the exponents of $\hb$ in the formulas.

\section{A combinatorial formula for tangent weights} \label{sec:bct_weights}
Fix a separated brane diagram $\brane$ and torus fixed point $p\in\bowvar(\brane)^\torus$. We will give a combinatorial formula for the tangent weights at $p$ (i.e., K-theoretic Chern roots of $T_p\bowvar(\brane)$) in terms of the associated table-with-margins. Unlike the approach used to derive \eqref{eqn:cancel}, this formula will return the tangent weights without the need for any further cancellations or simplifications. Note that since the bow variety is separated, the separating line always runs along the left and bottom sides. We ignore it and consider only the BCT $M$. Again, since the bow variety is separated, we have
\[
M_{ij} = \begin{cases}
1 & \text{if there is a tie between }V_i\text{ and }U_j, \\
0 & \text{otherwise}. \end{cases}
\]

\begin{Definition} \label{def:pairs}
For $0\leq i \leq n$, $1 \leq j \leq m$, define
\[
s_{ij} = \sum_{i' = 1}^i M_{i' j}.
\]
Furthermore, define a ``$kl$-pair'' to be a tuple of indices $(i, j, j')$ such that $M_{i j} = k$, $M_{i j'} = l$, and $j < j'$.
\end{Definition}

\begin{Theorem} \label{thm:weights}
Given a separated brane diagram $\brane$ and a fixed point $p\in\bowvar(\brane)^\torus$, we have
\begin{equation*}
T_p\bowvar(\brane) = \sum_{(i, j_0, j_1)\text{$01$-pair}} \left(
 \frac{\ub_{j_0}}{\ub_{j_1}}\hb^{s_{i j_1}-s_{i j_0}} + \frac{\ub_{j_1}}{\ub_{j_0}}\hb^{s_{i j_0}-s_{i j_1}+1} \right)
\end{equation*}
as elements of $K^0_\torus(p)$.
\end{Theorem}

\begin{proof}

To aid in our proof, let us establish some notation. For $1\leq j \leq m$ and $r\in\NN$, let
\[
r*\ub_j = \ub_j + \frac{\ub_j}{\hb} + \frac{\ub_j}{\hb^2} + \cdots + \frac{\ub_j}{\hb^{r-1}} \in K^0_\torus(p).
\]
Using this notation, we can rephrase the formula for the fixed point restriction of tautological bundles of Section~\ref{sec:fixed_res} as
\begin{equation}
\label{eqn:bct_res}
\xi_k|_p = \begin{cases}
\displaystyle\sum_{j=1}^m \hb^{k-n}(s_{kj}*\ub_j) & \text{if }k\leq n, \\
\displaystyle\sum_{j=k-n+1}^{m} s_{nj}*\ub_j & \text{if }n\leq k \leq n+m-1, \\
0 & \text{otherwise}.
\end{cases}
\end{equation}
Let \smash{$\xi_k^{(j)}$} be the subspace of $\xi_k|_p$ corresponding to the $U_j$ butterfly. Then, each of the summands in \eqref{eqn:bct_res} corresponds to a \smash{$\xi_k^{(j)}$}. Namely, we have
\begin{equation}
\label{eqn:bct_res_bfly}
\xi_k^{(j)} = \begin{cases}
\hb^{k-n}(s_{kj}*\ub_j) & \text{if }k\leq n, \\
s_{nj}*\ub_j & \text{if }n\leq k \leq n+j-1, \\
0 & \text{otherwise}.
\end{cases}
\end{equation}
It will also be helpful to summarize the tangent bundle formula (2.1) using the diagram
\begin{equation}
\label{eqn:tangent_sep}
\begin{tikzcd}[column sep = 0.8em]
\xi_1 \ar[rr, bend left, "1"] \cdloop{-1-\hb}
&& \cdots \ar[ll, bend left, "\hb"] \ar[rr, bend left, "1"]
&& \xi_{n - 1} \ar[ll, bend left, "\hb"] \ar[rr, bend left, "1"] \cdloop{-1-\hb}
&& \xi_n \ar[ll, bend left, "\hb"] \cdloop{-1}
&& \xi_{n + 1} \ar[ll, swap, "1-\hb"] \cdloop{-1+\hb} \ar[ld, "\hb"]
&& \cdots \ar[ll, swap, "1-\hb"] \ar[ld, "\hb"]
&& \xi_{n+m-1} \ar[ll, swap, "1-\hb"] \cdloop{-1+\hb} \ar[ld, "\hb"]
\\
&&&&&&&
\ub_{1} \ar[lu, "1"]
&& \ub_{2} \ar[lu, "1"]
&& \ub_{{m-1}} \ar[lu, "1"]
&& \ub_{m} \ar[lu, "1"]
\end{tikzcd}\!\!\!\!\!
\end{equation}
where $\xi_k = \xi_{X_k}$. Finally, divide the diagram \eqref{eqn:tangent_sep} into four types of pieces

\begin{tabular}{cp{0.59\textwidth}}
\begin{tikzcd}
\xi_{i-1} \cdloop{-1} \ar[r, bend left, "1"]
& \xi_{i} \ar[l, bend left, "\hb"] \cdloop{-\hb}
\end{tikzcd}
&
$P_{1i} = \Hom(\xi_{i-1}, \xi_{i}) + \hb\Hom(\xi_{i}, \xi_{i-1}) - \End(\xi_{i-1}) - \hb\End(\xi_i)$ for $1\leq i \leq n$. \\

\begin{tikzcd}
\xi_{n+j-1}
& \xi_{n+j} \ar[l, "1-\hb"] \cdloop{-1+\hb}
\end{tikzcd}
&
$P_{2j} = (1-\hb)\Hom(\xi_{n+j}, \xi_{n+j-1}) + (-1+\hb)\End(\xi_{n+j})$ for $1 \leq j\leq m$. \\

\begin{tikzcd}
\xi_{n+j-1} && \xi_{n+j} \ar[ld, "\hb"] \\
& \ub_j \ar[lu, "1"]
\end{tikzcd}
&
$P_{3j} = \Hom(\ub_j, \xi_{n+j-1}) + \hb\Hom(\xi_{n+j}, \ub_j)$ for $1\leq j\leq m$, where $\xi_{n+m}=0$ by convention. \\

\begin{tikzcd}
\xi_n \cdloop{-1+\hb}
\end{tikzcd}
&
$P_4 = (-1+\hb)\End(\xi_n)$.\vspace{2mm}
\end{tabular}

Inspecting the tangent bundle formula~\eqref{eqn:tangent_bundle} and the fixed point restriction formula~\eqref{eqn:bct_res}, we see that
\[
T_p\bowvar(\brane) = f(\hb) + \sum_{j_0 \neq j_1} \frac{\ub_{j_0}}{\ub_{j_1}} g_{j_0j_1}(\hb),
\]
where $f, g_{j_0j_1} \in \CC\big[\hb^{\pm 1}\big]$ for all $1\leq j_0, j_1 \leq m$. Fix $j_0 < j_1$ and define
\[
\delta_i^{(j)} = \xi_i^{(j)} - \xi_{i-1}^{(j)} \qquad \text{and} \qquad
\zeta_i^{(j)} = \xi_i^{(j)} - \hb\xi_{i-1}^{(j)}.
\] Then, the contribution of $P_{1i}$ to $g_{j_0j_1}$ is given by \smash{$\hat{P}_{1i} = -\hb\Hom\bigl(\zeta_i^{(j_1)}, \delta_i^{(j_0)}\bigr)$}.
The only nonzero contribution of the $P_{2j}$ is
\[
\hat{P}_{2j_0} = (1-\hb)\Hom\bigl(\xi_{n+j_0}^{j_1}, \xi_{n+j_0 - 1}^{(j_0)}\bigr).
\]
 From \eqref{eqn:bct_res_bfly}, we have \smash{$\xi_{n+j_0}^{(j_1)} = \xi_n^{(j_1)}$} and \smash{$ \xi_{n+j_0 - 1}^{(j_0)} = \xi_n^{(j_0)}$}. Thus, $\hat{P}_{2j_0}$ cancels with the contribution~$\hat{P}_4$ of $P_4$. Finally, the only nonzero contribution of the $P_{3j}$ is
 \[
 \hat{P}_{3j_0} = \hb\Hom\bigl(\xi_{n+j_0}^{(j_1)}, \ub_{j_0}\bigr).
 \] Summing everything together yields
\begin{align*}
\frac{\ub_{j_0}}{\ub_{j_1}}g_{j_0j_1} &= \hat{P}_{3j_0} + \sum_{i=1}^n \hat{P}_{1i}
= \sum_{i=1}^n \hat{P}_{1i} + \sum_{j=1}^m \bigl( \hat{P}_{2j} + \hat{P}_{3j} \bigr) + \hat{P}_4 \\
&= \hb\Hom\bigl(\xi_{n+j_0}^{(j_1)}, \ub_{j_0}\bigr) - \sum_{i=1}^n \hb\Hom(\zeta_i^{(j_1)}, \delta_i^{(j_0)}) \\
&= \hb\Hom\bigl(\xi_n^{(j_1)}, \ub_{j_0}\bigr) -\sum_{i=1}^n \hb\Hom(\zeta_i^{(j_1)}, \delta_i^{(j_0)}).
\end{align*}
Note that \smash{$\xi_n^{(j)} = \sum_{i=1}^n \hb^{n-i}\zeta_i^{(j)}$}. Substituting into the first term above results in
\[
\frac{\ub_{j_0}}{\ub_{j_1}}g_{j_0j_1} = \sum_{i=1}^n\hb\Hom\bigl(\zeta_i^{(j_1)}, \hb^{i-n}\ub_0 - \delta_i^{(j_0)}\bigr).
\]
The following formulas can be derived easily from \eqref{eqn:bct_res_bfly}:
\[
\delta_i^{(j)} = \hb^{i-n}(\ub_j + (M_{ij} - 1)\ub_j\hb^{-s_{ij}}),\qquad \zeta_i^{(j)} = \hb^{i-n} \bigl(M_{ij}\ub_j\hb^{-s_{ij}+1}\bigr).
\]
From these formulas, we can deduce
\[
g_{j_0j_1} = \sum_{i=1}^n (1-M_{ij_0})M_{ij_1}\hb^{s_{ij_1}-s_{ij_0}} = \sum_{\substack{i \\ (i, j_0, j_1) \text{ 01-pair}}} \hb^{s_{ij_1}-s_{ij_0}}.
\]

From a symmetric calculation, one can obtain
\[
g_{j_1j_0} = \sum_{\substack{i \\ (i, j_0, j_1) \mathrm{01-pair}}} \hb^{s_{ij_0}-s_{ij_1}+1}.
\]
Alternatively, once can appeal to the selfduality of the tangent bundle: $T\bowvar(\brane) = \hb T\bowvar(\brane)^\vee$. This selfduality property follows from general considerations but also from the explicit construction of a polarization bundle in \cite[Section~4.4.2]{S}.

Finally, a similar calculation can be done for $f$
\begin{gather*}
\hat{P}_{1i} = -\sum_{j=1}^m (M_{ij}\hb^{s_{ij}} + M_{ij}(M_{ij} - 1)), \qquad
\hat{P}_{2j} = 0, \\
\hat{P}_{3j} = \sum_{k=0}^{s_{nj}-1} \hb^{-k}, \qquad
\hat{P}_4 = (-1+\hb)\sum_{j=1}^m \End\bigl(\xi_n^{(j)}\bigr).
\end{gather*}
Since $M_{ij}(M_{ij} - 1) = 0$, we have
\[
\sum_{i=1}^n \hat{P}_{1i} = -\sum_{j=1}^m \sum_{i=1}^n M_{ij}\hb^{s_{ij}} = -\sum_{j=1}^m \sum_{k=1}^{s_{nj}} \hb^k,
\]
from which is easy to see that $\sum_i \hat{P}_{1i} + \sum_j \hat{P}_{3j} = -\hat{P}_4$. It follows that $f=0$, completing the proof.
\end{proof}

Suppose now that $\brane$ is not separated and $p\in\bowvar(\brane)^\torus$. By applying a sequence of Hanany--Witten transitions, the tie diagram for $p$ can be transformed into a separated diagram $\widetilde\brane$, representing a fixed point \smash{$\widetilde p\in\bowvar(\widetilde\brane)^{\widetilde\torus}$} on a separated bow variety. This changes the separating line, but not the BCT. Hence, we can immediately apply Theorem~\ref{thm:weights} to obtain the tangent weights at $\widetilde p$. As discussed in Section~\ref{sec:HW}, these differ from the tangent weights at $p$ only in the exponents of $\hb$. Starting from $\widetilde\brane$ and going to $\brane$, each time we move a D5 brane $U$ to the left, we must multiply $\ub$ by $\hb$.

\begin{Corollary} \label{cor:weights}
Let $\brane$ be a brane diagram and $p\in\bowvar(\brane)^\torus$. Let $M$ be the corresponding BCT and $\sigma_j$ be the number of entries in the $j$-th column of $M$ lying below the separating line. Then, the tangent weights at $p$ are given by
\[
T_p\bowvar(\brane) = \sum_{(i, j_0, j_1)\text{\rm 01-pair}} \left(
 \frac{\ub_{j_0}}{\ub_{j_1}}\hb^{s_{ij_1}-s_{ij_0} + \sigma_{j_0}-\sigma_{j_1}} + \frac{\ub_{j_1}}{\ub_{j_0}}\hb^{s_{ij_0}-s_{ij_1} + \sigma_{j_1}-\sigma_{j_0} +1} \right).
\]
\end{Corollary}

Observe that the tangent weights come in pairs of the form $w+\hb w^{-1}$. This is a manifestation of the fact that bow varieties admit polarization bundles \cite[Section~4.4.2]{S}. Theorem~\ref{thm:weights} combined with the dimension formula \cite[equation~(4)]{RS} gives us a roundabout proof of the following purely combinatorial result.
\begin{Corollary}
Let $M$ be a BCT with margins $r$, $c$. Then, the number of $01$-pairs in $M$ is
\[
\frac{1}{2}\Biggl(
\sum_{j=1}^m ( \underline c_j(\underline c_j + 1) + \underline c_{j-1}(\underline c_{j-1} + 1) )
+ 2\sum_{i=1}^n \underline r_{i-1}\underline r_i
-2\sum_{j=1}^m \underline c_j^2 -2\sum_{i=1}^{n-1} \underline r_i^2
\Biggr),
\]
where $\underline c_j=\sum_{j'=1}^j c_{j'}$, and $\underline r_i=\sum_{i'=1}^i r_{i'}$. In particular, this number depends only on the margins.
\end{Corollary}

\section{Butterfly surgeries and invariant curves} \label{sec:curves}
In this section, we will develop a combinatorial representation of $\torus$-invariant curves involving surgery operations on pairs of butterflies (see Section~\ref{sec:fixed_pts}). Our goal is to identify all $\torus$-invariant curves. We will accomplish this using an exhaustive approach. For each fixed point $p$ and tangent weight $w$ that appears with multiplicity $k$ at $p$, we will produce a $k$-dimensional pencil of invariant curves with tangent weight $w$ at $p$.

\subsection{Torus invariant curves} \label{sec:invariant_curves}
Let us begin by collecting some standard facts regarding invariant curves that will be used throughout what follows. These facts can be found in \cite[Section~7]{GKM} and \cite[Section~7.2]{equiv_coh}. We address specifically the properties of invariant curves in bow varieties, though these properties are quite general and hold for a much larger family of spaces. Fix a bow variety $\bowvar(\brane)$ with the corresponding action by the torus $\torus$. By a ``$\torus$-invariant curve'' $\gamma$, we mean the closure of a~1-dimensional $\torus$-orbit in $\bowvar(\brane)$. The following are true of $\torus$-invariant curves:
\begin{itemize}\itemsep=0pt
\item A 1-dimensional $\torus$-orbit is isomorphic to $\CC^\times$, and its closure is obtained by adding up to two $\torus$-fixed points.
\item An invariant curve is compact if it contains two fixed points and noncompact otherwise. Compact invariant curves can be parametrized by equivariant embeddings $\CCP^1 \to \bowvar(\brane)$, and noncompact invariant curves with one fixed point can be parametrized by equivariant embeddings $\CC \to \bowvar(\brane)$.
\item If $\gamma$ is an invariant curve, and $p \in \gamma$ is a fixed point, then $T_p \gamma$ is a $\torus$-weight space of~$T_p \bowvar(\brane)$. If $p_1, p_2 \in \gamma$ are two distinct fixed points, then the $\torus$-weights of $T_{p_1}\gamma$ and $T_{p_2}\gamma$ are reciprocols.
\item In contrast to \cite{GKM}, bow varieties may contain infinitely many invariant curves. In this case, the curves can be arranged into parametrized families or ``pencils''. Pencils of invariant curves will be described in Section~\ref{sec:butterfly_surgery}. An invariant curve is a special case of a pencil, where the dimension of the pencil is 1.
\end{itemize}

The pencils of invariant curves have important implications for the study of the equivariant cohomology of $\bowvar(\brane)$. When $\bowvar(\brane)$ is equivariantly formal and has only finitely many fixed points and invariant curves (no multidimensional pencils), \cite{GKM} shows that the map $H^*_\torus (\bowvar(\brane)) \to H^*_\torus \bigl(\bowvar(\brane)^\torus\bigr)$ induced by the inclusion of the fixed point locus is an injective ring homomorphism whose image is characterized by simple matching conditions on pairs of fixed points joined by an invariant curve. When multidimensional pencils of invariant curves are present, however, these matching conditions are not sufficient, and must be supplemented by more general matching conditions. See \cite{CS} and \cite[Section~7.4]{equiv_coh}. It also appears that not all bow varieties are equivariantly formal \cite{yibo}.

From now on, we will only be interested in invariant curves containing at least one fixed point. The following result will be pivotal in classifying all such curves.

\begin{Proposition} \label{prop:equiv_geo}
Let $\bowvar(\brane)$ be a bow variety and $p \in \bowvar(\brane)^\torus$. Then, we have
\begin{enumerate}\itemsep=0pt
\item[$(1)$] All $\torus$-invariant curves containing $p$ are smooth.
\item[$(2)$] If $\gamma_1$ and $\gamma_2$ are $\torus$-invariant curves containing $p$, and $T_p\gamma_1=T_p\gamma_2$ as vector subspaces of~$T_p\bowvar(\brane)$, then $\gamma_1 = \gamma_2$.
\end{enumerate}
\end{Proposition}

\begin{proof}
Let $p\in U \subset \bowvar(\brane)$ be an affine $\torus$-invariant neighborhood. By the Luna slice theorem~\cite[Appendix of Chapter~1]{MFK}, there exists a $\torus$-equivariant morphism $\phi\colon V \to T_p \bowvar(\brane)$ such that~${p \in V \subset U}$ is a $\torus$-invariant neighborhood, $\phi(p) = 0$, and $\phi$ is \'etale. Due to equivariance, $\phi$~maps invariant curves containing $p$ to invariant curves containing $0$. It can be easily seen from the structure of the tangent weights provided by Corollary~\ref{cor:weights} that the only invariant curves in~$T_p\bowvar(\brane)$ containing $0$ are lines. \'Etale morphisms induce isomorphisms on tangent spaces. This proves (1). Moreover, since our spaces are smooth, $\phi$ restricts to a diffeomorphism on some analytic neighborhood of $p$. This proves (2).
\end{proof}

 Our approach will be to construct a collection of invariant curves containing $p$ such that every $\torus$-invariant tangent line at $p$ appears as the tangent line to one of the curves. Then, Proposition~\ref{prop:equiv_geo} will guarantee that all curves containing $p$ are accounted for.

\subsection{Butterfly surgery} \label{sec:butterfly_surgery}
\begin{figure}[t]
\centering
\begin{tikzpicture}
\begin{scope}[shift = {(0, 21)}]
\draw[fill] (-1,-1) circle [radius = 0.025] node(0a1){};

\draw[fill] (0,0) circle [radius = 0.025] node(1a1){}; \draw[fill] (0,-1) circle [radius = 0.025] node(1a2){};

\draw[fill] (1,1) circle [radius = 0.025] node(2a1){}; \draw[fill] (1,0) circle [radius = 0.025] node(2a2){}; \draw[fill] (1,-1) circle [radius = 0.025] node(2a3){};

\draw[fill] (2,2) circle [radius = 0.025] node(3a1){}; \draw[fill] (2,1) circle [radius = 0.025] node(3a2){}; \draw[fill] (2,0) circle [radius = 0.025] node(3a3){};

\draw[fill] (3,3) circle [radius = 0.025] node(4a1){}; \draw[fill] (3,2) circle [radius = 0.025] node(4a2){}; \draw[fill] (3,1) circle [radius = 0.025] node(4a3){};
\draw[fill] (3,0) circle [radius = 0.025] node(4a4){};

\draw[fill] (4,4) circle [radius = 0.025] node(5a1){}; \draw[fill] (4,3) circle [radius = 0.025] node(5a2){}; \draw[fill] (4,2) circle [radius = 0.025] node(5a3){};
\draw[fill] (4,1) circle [radius = 0.025] node(5a4){};

\draw[fill] (5,4) circle [radius = 0.025] node(6a1){}; \draw[fill] (5,3) circle [radius = 0.025] node(6a2){}; \draw[fill] (5,2) circle [radius = 0.025] node(6a3){};
\draw[fill] (5,1) circle [radius = 0.025] node(6a4){};

\draw[fill] (6,4) circle [radius = 0.025] node(7a1){}; \draw[fill] (6,3) circle [radius = 0.025] node(7a2){}; \draw[fill] (6,2) circle [radius = 0.025] node(7a3){};
\draw[fill] (6,1) circle [radius = 0.025] node(7a4){};

\draw[fill] (7,4) circle [radius = 0.025] node(8a1){}; \draw[fill] (7,3) circle [radius = 0.025] node(8a2){}; \draw[fill] (7,2) circle [radius = 0.025] node(8a3){};
\draw[fill] (7,1) circle [radius = 0.025] node(8a4){};

\draw[fill] (8,4) circle [radius = 0.025] node(9a1){}; \draw[fill] (8,3) circle [radius = 0.025] node(9a2){}; \draw[fill] (8,2) circle [radius = 0.025] node(9a3){};
\draw[fill] (8,1) circle [radius = 0.025] node(9a4){};

\draw[fill] (9,4) circle [radius = 0.025] node(10a1){}; \draw[fill] (9,3) circle [radius = 0.025] node(10a2){};
\draw[fill] (9,2) circle [radius = 0.025] node(10a3){};

\draw[fill] (10,4) circle [radius = 0.025] node(11a1){}; \draw[fill] (10,3) circle [radius = 0.025] node(11a2){};

\draw[fill] (11,4) circle [radius = 0.025] node(12a1){}; \draw[fill] (11,3) circle [radius = 0.025] node(12a2){};

\draw[fill] (12,4) circle [radius = 0.025] node(13a1){};

\draw (4.5, 5) circle [radius = 0.05] node(a){};

\draw[->, red] (0a1)--(1a2);
\draw[->, dotted, magenta] (1a1)--(0a1);

\draw[->, red] (1a1)--(2a2); \draw[->, red] (1a2)--(2a3);
\draw[->, dotted, magenta] (2a1)--(1a1); \draw[->, dotted, magenta] (2a2)--(1a2);

\draw[->, red] (2a1)--(3a2); \draw[->, red] (2a2)--(3a3);
\draw[->, dotted, magenta] (3a2)--(2a2); \draw[->, dotted, magenta] (3a1)--(2a1); \draw[->, dotted, magenta] (3a3)--(2a3);

\draw[->, red] (3a1)--(4a2); \draw[->, red] (3a2)--(4a3); \draw[->, red] (3a3)--(4a4);
\draw[->, dotted, magenta] (4a1)--(3a1); \draw[->, dotted, magenta] (4a2)--(3a2); \draw[->, dotted, magenta] (4a3)--(3a3);

\draw[->, red] (4a1)--(5a2); \draw[->, red] (4a2)--(5a3); \draw[->, red] (4a3)--(5a4);
\draw[->, dotted, magenta] (5a1)--(4a1); \draw[->, dotted, magenta] (5a2)--(4a2); \draw[->, dotted, magenta] (5a3)--(4a3);
\draw[->, dotted, magenta] (5a4)--(4a4);

\draw[->, blue] (6a1)--(5a1); \draw[->, blue] (6a2)--(5a2); \draw[->, blue] (6a3)--(5a3); \draw[->, blue] (6a4)--(5a4);
\draw[->] (5a1)--(5a2); \draw[->] (5a2)--(5a3); \draw[->] (5a3)--(5a4); \draw[->] (6a1)--(6a2); \draw[->] (6a2)--(6a3); \draw[->] (6a3)--(6a4);

\draw[->, red] (6a1)--(7a1); \draw[->, red] (6a2)--(7a2); \draw[->, red] (6a3)--(7a3); \draw[->, red] (6a4)--(7a4);
\draw[->, dotted, magenta] (7a1)--(6a2); \draw[->, dotted, magenta] (7a2)--(6a3); \draw[->, dotted, magenta] (7a3)--(6a4);

\draw[->, red] (7a1)--(8a1); \draw[->, red] (7a2)--(8a2); \draw[->, red] (7a3)--(8a3); \draw[->, red] (7a4)--(8a4);
\draw[->, dotted, magenta] (8a1)--(7a2); \draw[->, dotted, magenta] (8a2)--(7a3); \draw[->, dotted, magenta] (8a3)--(7a4);

\draw[->, blue] (9a1)--(8a1); \draw[->, blue] (9a2)--(8a2); \draw[->, blue] (9a3)--(8a3); \draw[->, blue] (9a4)--(8a4);
\draw[->] (8a1)--(8a2); \draw[->] (8a2)--(8a3); \draw[->] (8a3)--(8a4); \draw[->] (9a1)--(9a2); \draw[->] (9a2)--(9a3); \draw[->] (9a3)--(9a4);

\draw[->, red] (9a1)--(10a1); \draw[->, red] (9a2)--(10a2); \draw[->, red] (9a3)--(10a3);
\draw[->, dotted, magenta] (10a1)--(9a2); \draw[->, dotted, magenta] (10a2)--(9a3); \draw[->, dotted, magenta] (10a3)--(9a4);

\draw[->, red] (10a1)--(11a1); \draw[->, red] (10a2)--(11a2);
\draw[->, dotted, magenta] (11a1)--(10a2); \draw[->, dotted, magenta] (11a2)--(10a3);

\draw[->, red] (11a1)--(12a1); \draw[->, red] (11a2)--(12a2);
\draw[->, dotted, magenta] (12a1)--(11a2);

\draw[->, red] (12a1)--(13a1);
\draw[->, dotted, magenta] (13a1)--(12a2);

\draw[->, green] (a) to [in = 15, out = 245] (5a1);
\end{scope}

\begin{scope}[shift = {(0, 14)}]
\draw[fill] (1,1) circle [radius = 0.025] node(2b1){};

\draw[fill] (2,2) circle [radius = 0.025] node(3b1){};

\draw[fill] (3,3) circle [radius = 0.025] node(4b1){};

\draw[fill] (4,4) circle [radius = 0.025] node(5b1){}; \draw[fill] (4,3) circle [radius = 0.025] node(5b2){};

\draw[fill] (5,4) circle [radius = 0.025] node(6b1){}; \draw[fill] (5,3) circle [radius = 0.025] node(6b2){};

\draw[fill] (6,5) circle [radius = 0.025] node(7b1){}; \draw[fill] (6,4) circle [radius = 0.025] node(7b2){};

\draw[fill] (7,6) circle [radius = 0.025] node(8b1){}; \draw[fill] (7,5) circle [radius = 0.025] node(8b2){};

\draw[fill] (8,5) circle [radius = 0.025] node(9b1){};

\draw[fill] (9,5) circle [radius = 0.025] node(10b1){};

\draw[fill] (10,5) circle [radius = 0.025] node(11b1){};

\draw[fill] (11,5) circle [radius = 0.025] node(12b1){};

\draw[fill] (12,5) circle [radius = 0.025] node(13b1){};

\draw (7.5, 7) circle [radius = 0.05] node(b){};

\draw[->, dotted, magenta] (3b1)--(2b1);

\draw[->, dotted, magenta] (4b1)--(3b1);

\draw[->, red] (4b1)--(5b2);
\draw[->, dotted, magenta] (5b1)--(4b1);

\draw[->, blue] (6b1)--(5b1); \draw[->, blue] (6b2)--(5b2);
\draw[->] (5b1)--(5b2); \draw[->] (6b1)--(6b2);

\draw[->, red] (6b1)--(7b2);
\draw[->, dotted, magenta] (7b1)--(6b1); \draw[->, dotted, magenta] (7b2)--(6b2);

\draw[->, red] (7b1)--(8b2);
\draw[->, dotted, magenta] (8b1)--(7b1); \draw[->, dotted, magenta] (8b2)--(7b2);

\draw[->, blue] (9b1)--(8b2);
\draw[->] (8b1)--(8b2);

\draw[->, red] (9b1)--(10b1);

\draw[->, red] (10b1)--(11b1);

\draw[->, red] (11b1)--(12b1);

\draw[->, red] (12b1)--(13b1);

\draw[->, green] (b) to [in = 15, out = 245] (8b1);
\end{scope}

\begin{scope}[shift = {(0, 20)}]
\draw[gray] (-0.1, -0.1)--(1.05, -0.1)--(2.05, 0.9)--(3.05, 0.9)--(4.05, 1.9)--(8.05, 1.9)--(10.05, 3.9)--(11.1, 3.9)--(11.1, 4.1)--(9.95, 4.1)--(6.95, 4.1)--(4.95, 2.1)--(2.95, 2.1)--(0.95, 0.1)--(-0.1, 0.1)--(-0.1, -0.1);
\end{scope}

\begin{scope}[shift = {(0, 14)}]
\draw[gray] (-0.1, -0.1)--(1.05, -0.1)--(2.05, 0.9)--(3.05, 0.9)--(4.05, 1.9)--(8.05, 1.9)--(10.05, 3.9)--(11.1, 3.9)--(11.1, 4.1)--(9.95, 4.1)--(6.95, 4.1)--(4.95, 2.1)--(2.95, 2.1)--(0.95, 0.1)--(-0.1, 0.1)--(-0.1, -0.1);
\end{scope}
\end{tikzpicture}

\caption{Fixed point with the site of a butterfly surgery outlined.} \label{fig:preop}
\end{figure}

A fixed point $p_1\in \bowvar(\brane)^\torus$ corresponds to a butterfly diagram, that is, a collection of butterflies, one centered on each D5 brane. We stack the butterflies on top of each other, so that the centers of the butterflies listed from top to bottom go from left to right in $\brane$. Fix two butterflies $\but_1$ and~$\but'_1$ with centers $U$ and $U'$, respectively. Suppose that there is a (possibly disconnected) subgraph $\site$ of~$\but_1$, such that stacking $\site$ below $\but'_1$ and creating new edges between $\but'_1$ and $\site$ according to the rules of Section~\ref{sec:fixed_pts} results in a butterfly $\but'_2$. The subgraph $\site$, called the \emph{site} of the surgery, must be translated vertically, in a rigid fashion, without any lateral movement, rotation, or deformation, and we do not create new edges within $\site$. Let $\but_2$ be obtained by deleting $\site$ and all adjacent edges from $\but_1$, and assume that $\but_2$ is also a butterfly. The operation of replacing $\but_1$ with $\but_2$ and $\but'_1$ with $\but'_2$ is called a ``butterfly surgery''. It transforms the butterfly diagram for $p_1$ into the butterfly diagram for another fixed point $p_2\in \bowvar(\brane)^\torus$. We will use $\surgery$ to denote butterfly surgeries. Clearly, in a butterfly surgery, $\site$ must contain all edges of $\but_1$ with both endpoints within $\site$. Moreover, the 0-momentum condition forces $\site$ to be $A$, $A^{-1}$, $B$, $C$, $D$-invariant. By this, we mean that all $B$, $C$, $D$ edges in $\but_1$ incident to $\site$ are directed into $\site$, and $\site$ contains all incident $A$ edges.

An invariant curve arises from a butterfly surgery as follows. Consider the newly created edges in the butterfly diagram for $p_2$. Each edge starts in $\but'_1$ and ends in the translated copy of~$\site$, which we will call $\site'$. Create the corresponding edges starting in $\but'_1$ and ending in $\site\subset\but_1$ in the butterfly diagram for $p_1$, as in Figure \ref{fig:invariant_curve}. Identifying vertices with basis vectors of the $W_X$ spaces and interpreting the edges as linear maps, the resulting graph gives an element~${\tilde{p}_\gamma \in \MM}$. The 0-momentum condition holds by construction, as the rules in Section~\ref{sec:fixed_pts} governing the vertices and edges of a butterfly directly reflect the 0-momentum conditions (see the proof of~\cite[Theorem~4.8]{RS}). The S1, S2, and $\nu$ stability conditions hold as well.

\begin{Lemma} \label{lemma:type1_curve}
The point $\tilde{p}_\gamma \in \MM$ satisfies the stability conditions S$1$, S$2$, and $\nu$.
\end{Lemma}

\begin{proof}
For each D5 brane $U_0$, S1 and S2, place a condition on the $A_{U_0}$, $B_{U_0}$, $a_{U_0}$, $b_{U_0}$ maps. Let~${s \in \ker(A_{U_0}) \cap \ker(b_{U_0})}$. Inspecting the structure of $\tilde{p}_\gamma$, $s$ must be in the span of the vertices of the $U_0$ butterfly above the source of the~$b_{U_0}$ edge. However, sufficiently many applications of~$B_{{U_0}^+}$ will result in a vector no longer in the kernel of $b_{U_0}$. This proves S1. For S2, observe that~$\im(A_{U_0})$ contains all vertices under ${U_0}^-$ except possibly for some in the $U_0$ butterfly. Any vertex of the $U_0$ butterfly can be reached from $\im(a_{U_0})$ by sufficiently many applications of $B_{{U_0}^-}$.

To verify the $\nu$ stability condition, observe that for any D5 brane $U_0$, every vertex of the $U_0$ butterfly can be reached from the target vertex of $a_{U_0}$ or from the highest vertex under $U_0^+$ by following a sequence of $A$, $A^{-1}$, $B$, $C$, $D$ edges. By following an $A^{-1}$ edge, we simply mean following an $A$ edge in the direction opposite to its orientation.
\end{proof}

It follows that \smash{$\tilde{p}_\gamma \in \prequot$} descends to a point $p_\gamma \in \bowvar(\brane)$. Denote the $\torus$-orbit closure of $p_\gamma$ by~$\gamma$.

\begin{figure}[t]
\centering
\begin{tikzpicture}
\begin{scope}[shift = {(0, 7)}]
\draw[fill] (-1,-1) circle [radius = 0.025] node(0a1){};

\draw[fill] (0,0) circle [radius = 0.025] node(1a1){};

\draw[fill] (1,1) circle [radius = 0.025] node(2a1){}; \draw[fill] (1,0) circle [radius = 0.025] node(2a2){};

\draw[fill] (2,2) circle [radius = 0.025] node(3a1){}; \draw[fill] (2,1) circle [radius = 0.025] node(3a2){};

\draw[fill] (3,3) circle [radius = 0.025] node(4a1){}; \draw[fill] (3,2) circle [radius = 0.025] node(4a2){};

\draw[fill] (4,4) circle [radius = 0.025] node(5a1){}; \draw[fill] (4,3) circle [radius = 0.025] node(5a2){}; \draw[fill] (4,2) circle [radius = 0.025] node(5a3){};

\draw[fill] (5,4) circle [radius = 0.025] node(6a1){}; \draw[fill] (5,3) circle [radius = 0.025] node(6a2){}; \draw[fill] (5,2) circle [radius = 0.025] node(6a3){};

\draw[fill] (6,4) circle [radius = 0.025] node(7a1){}; \draw[fill] (6,3) circle [radius = 0.025] node(7a2){};

\draw[fill] (7,4) circle [radius = 0.025] node(8a1){};

\draw[fill] (8,4) circle [radius = 0.025] node(9a1){};

\draw[fill] (9,4) circle [radius = 0.025] node(10a1){};

\draw[fill] (10,4) circle [radius = 0.025] node(11a1){};

\draw[fill] (11,4) circle [radius = 0.025] node(12a1){};

\draw[fill] (12,4) circle [radius = 0.025] node(13a1){};

\draw (4.5, 5) circle [radius = 0.05] node(a){};

\draw[->, dotted, magenta] (1a1)--(0a1);

\draw[->, red] (1a1)--(2a2); \draw[->, dotted, magenta] (2a1)--(1a1);

\draw[->, red] (2a1)--(3a2);
\draw[->, dotted, magenta] (3a2)--(2a2); \draw[->, dotted, magenta] (3a1)--(2a1);

\draw[->, red] (3a1)--(4a2);
\draw[->, dotted, magenta] (4a1)--(3a1); \draw[->, dotted, magenta] (4a2)--(3a2);

\draw[->, red] (4a1)--(5a2); \draw[->, red] (4a2)--(5a3);
\draw[->, dotted, magenta] (5a1)--(4a1); \draw[->, dotted, magenta] (5a2)--(4a2);

\draw[->, blue] (6a1)--(5a1); \draw[->, blue] (6a2)--(5a2); \draw[->, blue] (6a3)--(5a3);
\draw[->] (5a1)--(5a2); \draw[->] (5a2)--(5a3); \draw[->] (6a1)--(6a2); \draw[->] (6a2)--(6a3);

\draw[->, red] (6a1)--(7a1); \draw[->, red] (6a2)--(7a2);
\draw[->, dotted, magenta] (7a1)--(6a2); \draw[->, dotted, magenta] (7a2)--(6a3);

\draw[->, red] (7a1)--(8a1);
\draw[->, dotted, magenta] (8a1)--(7a2);

\draw[->, blue] (9a1)--(8a1);

\draw[->, red] (9a1)--(10a1);

\draw[->, red] (10a1)--(11a1);

\draw[->, red] (11a1)--(12a1);

\draw[->, red] (12a1)--(13a1);

\draw[->, green] (a) to [in = 15, out = 245] (5a1);
\end{scope}

\begin{scope}
\draw[fill] (0,0) circle [radius = 0.025] node(1a1){};

\draw[fill] (1,1) circle [radius = 0.025] node(2a1){}; \draw[fill] (1,0) circle [radius = 0.025] node(2a2){};

\draw[fill] (2,2) circle [radius = 0.025] node(3a1){}; \draw[fill] (2,1) circle [radius = 0.025] node(3a2){};

\draw[fill] (3,3) circle [radius = 0.025] node(4a1){}; \draw[fill] (3,2) circle [radius = 0.025] node(4a2){}; \draw[fill] (3,1) circle [radius = 0.025] node(4a3){};

\draw[fill] (4,4) circle [radius = 0.025] node(5a1){}; \draw[fill] (4,3) circle [radius = 0.025] node(5a2){}; \draw[fill] (4,2) circle [radius = 0.025] node(5a3){};

\draw[fill] (5,4) circle [radius = 0.025] node(6a1){}; \draw[fill] (5,3) circle [radius = 0.025] node(6a2){}; \draw[fill] (5,2) circle [radius = 0.025] node(6a3){};

\draw[fill] (6,5) circle [radius = 0.025] node(7a1){}; \draw[fill] (6,4) circle [radius = 0.025] node(7a2){}; \draw[fill] (6,3) circle [radius = 0.025] node(7a3){};
\draw[fill] (6,2) circle [radius = 0.025] node(7a4){};

\draw[fill] (7,6) circle [radius = 0.025] node(8a1){}; \draw[fill] (7,5) circle [radius = 0.025] node(8a2){}; \draw[fill] (7,4) circle [radius = 0.025] node(8a3){};
\draw[fill] (7,3) circle [radius = 0.025] node(8a4){}; \draw[fill] (7,2) circle [radius = 0.025] node(8a5){};

\draw[fill] (8,5) circle [radius = 0.025] node(9a1){}; \draw[fill] (8,4) circle [radius = 0.025] node(9a2){}; \draw[fill] (8,3) circle [radius = 0.025] node(9a3){};
\draw[fill] (8,2) circle [radius = 0.025] node(9a4){};

\draw[fill] (9,5) circle [radius = 0.025] node(10a1){}; \draw[fill] (9,4) circle [radius = 0.025] node(10a2){};
\draw[fill] (9,3) circle [radius = 0.025] node(10a3){};

\draw[fill] (10,5) circle [radius = 0.025] node(11a1){}; \draw[fill] (10,4) circle [radius = 0.025] node(11a2){};

\draw[fill] (11,5) circle [radius = 0.025] node(12a1){}; \draw[fill] (11,4) circle [radius = 0.025] node(12a2){};

\draw[fill] (12,5) circle [radius = 0.025] node(13a1){};

\draw (7.5, 7) circle [radius = 0.05] node(a){};

\draw[->, red] (1a1)--(2a2); \draw[->, dotted, magenta] (2a1)--(1a1);

\draw[->, red] (2a1)--(3a2);
\draw[->, dotted, magenta] (3a2)--(2a2); \draw[->, dotted, magenta] (3a1)--(2a1);

\draw[->, red] (3a1)--(4a2); \draw[->, red] (3a2)--(4a3);
\draw[->, dotted, magenta] (4a1)--(3a1); \draw[->, dotted, magenta] (4a2)--(3a2);

\draw[->, red] (4a1)--(5a2); \draw[->, red] (4a2)--(5a3);
\draw[->, dotted, magenta] (5a1)--(4a1); \draw[->, dotted, magenta] (5a2)--(4a2); \draw[->, dotted, magenta] (5a3)--(4a3);

\draw[->, blue] (6a1)--(5a1); \draw[->, blue] (6a2)--(5a2); \draw[->, blue] (6a3)--(5a3);
\draw[->] (5a1)--(5a2); \draw[->] (5a2)--(5a3); \draw[->] (6a1)--(6a2); \draw[->] (6a2)--(6a3);

\draw[->, red] (6a1)--(7a2); \draw[->, red] (6a2)--(7a3); \draw[->, red] (6a3)--(7a4);
\draw[->, dotted, magenta] (7a1)--(6a1); \draw[->, dotted, magenta] (7a2)--(6a2); \draw[->, dotted, magenta] (7a3)--(6a3);

\draw[->, red] (7a1)--(8a2); \draw[->, red] (7a2)--(8a3); \draw[->, red] (7a3)--(8a4); \draw[->, red] (7a4)--(8a5);
\draw[->, dotted, magenta] (8a1)--(7a1); \draw[->, dotted, magenta] (8a2)--(7a2); \draw[->, dotted, magenta] (8a3)--(7a3);
\draw[->, dotted, magenta] (8a4)--(7a4);

\draw[->, blue] (9a1)--(8a2); \draw[->, blue] (9a2)--(8a3); \draw[->, blue] (9a3)--(8a4); \draw[->, blue] (9a4)--(8a5);
\draw[->] (8a1)--(8a2); \draw[->] (8a2)--(8a3); \draw[->] (8a3)--(8a4); \draw[->] (8a4)--(8a5); \draw[->] (9a1)--(9a2); \draw[->] (9a2)--(9a3);
\draw[->] (9a3)--(9a4);

\draw[->, red] (9a1)--(10a1); \draw[->, red] (9a2)--(10a2); \draw[->, red] (9a3)--(10a3);
\draw[->, dotted, magenta] (10a1)--(9a2); \draw[->, dotted, magenta] (10a2)--(9a3); \draw[->, dotted, magenta] (10a3)--(9a4);

\draw[->, red] (10a1)--(11a1); \draw[->, red] (10a2)--(11a2);
\draw[->, dotted, magenta] (11a1)--(10a2); \draw[->, dotted, magenta] (11a2)--(10a3);

\draw[->, red] (11a1)--(12a1); \draw[->, red] (11a2)--(12a2);
\draw[->, dotted, magenta] (12a1)--(11a2);

\draw[->, red] (12a1)--(13a1);
\draw[->, dotted, magenta] (13a1)--(12a2);

\draw[->, green] (a) to [in = 15, out = 245] (8a1);
\end{scope}

\begin{scope}[shift = {(0, 6)}]
\draw[gray] (-0.1, -0.1)--(1.05, -0.1)--(2.05, 0.9)--(3.05, 0.9)--(4.05, 1.9)--(8.05, 1.9)--(10.05, 3.9)--(11.1, 3.9)--(11.1, 4.1)--(9.95, 4.1)--(6.95, 4.1)--(4.95, 2.1)--(2.95, 2.1)--(0.95, 0.1)--(-0.1, 0.1)--(-0.1, -0.1);
\end{scope}

\begin{scope}
\draw[gray] (-0.1, -0.1)--(1.05, -0.1)--(2.05, 0.9)--(3.05, 0.9)--(4.05, 1.9)--(8.05, 1.9)--(10.05, 3.9)--(11.1, 3.9)--(11.1, 4.1)--(9.95, 4.1)--(6.95, 4.1)--(4.95, 2.1)--(2.95, 2.1)--(0.95, 0.1)--(-0.1, 0.1)--(-0.1, -0.1);
\end{scope}
\end{tikzpicture}

\vspace{1mm}

\caption{New fixed point resulting from butterfly surgery.} \label{fig:postop}
\end{figure}
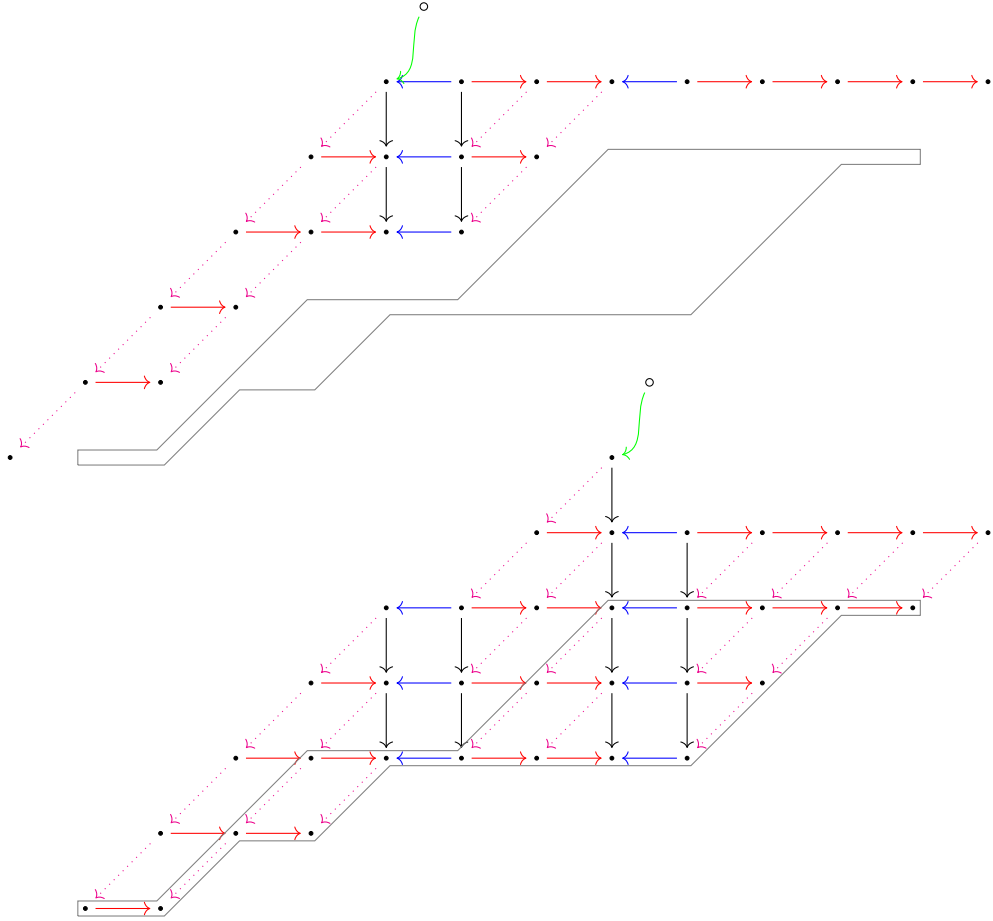

For a vertex $v$ of the butterfly diagram, recall that its height (see Definition~\ref{def:height}) is denoted by~$y(v)$. Now, let $v\in\site \subset \beta$. Applying the surgery yields the corresponding vertex $v'\in\site' \subset \beta'$. The change in height $\Delta_\surgery y = y(v') - y(v)$ is constant with respect to $v\in\site$. Consider the action of $\torus$ on $p_\gamma$. By a suitable change of basis in $W_X$ spaces (the $\G$-action), a representation of the resulting point can be obtained by multiplying the new edges by the weight $\hb^{-\Delta_\surgery y} \ub / \ub'$ and leaving the other edges alone. By fixing $\ub$, $\hb$ and taking $\ub' \to \infty$, we see that $p_1 \in \gamma$.

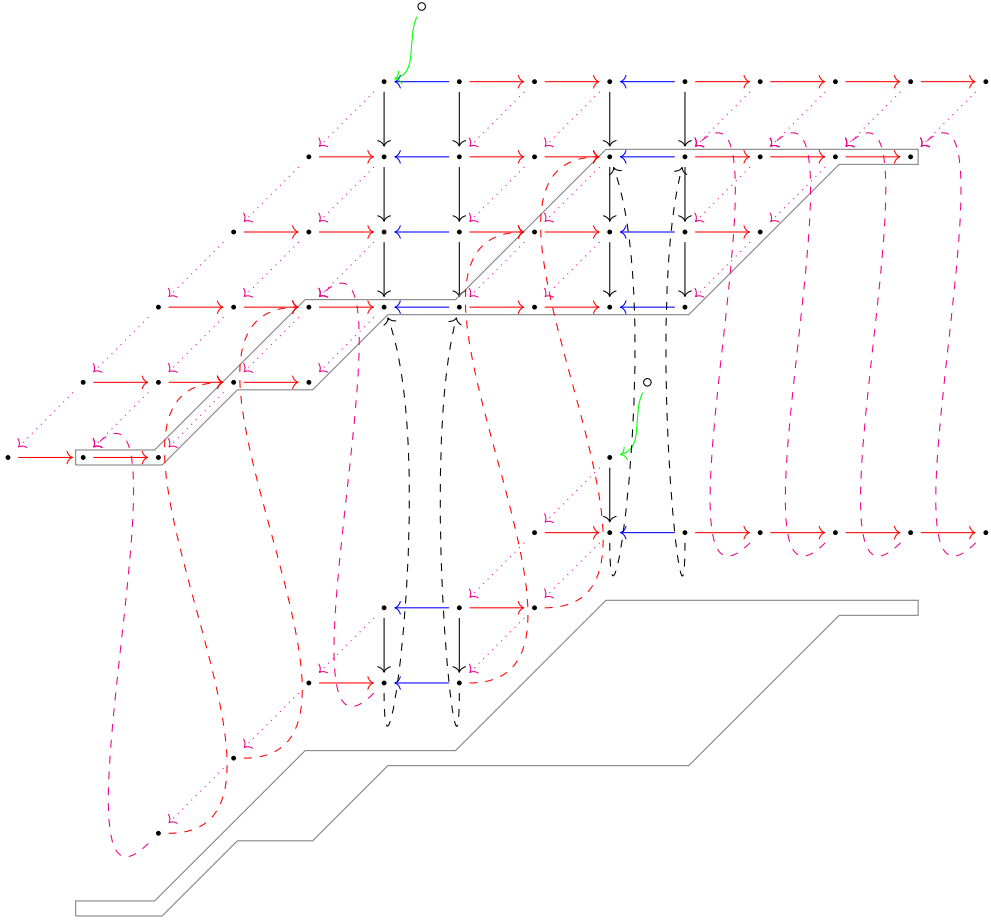
\begin{figure}[t]
\centering
\begin{tikzpicture}
\begin{scope}[shift = {(0, 21)}]
\draw[fill] (-1,-1) circle [radius = 0.025] node(0a1){};

\draw[fill] (0,0) circle [radius = 0.025] node(1a1){}; \draw[fill] (0,-1) circle [radius = 0.025] node(1a2){};

\draw[fill] (1,1) circle [radius = 0.025] node(2a1){}; \draw[fill] (1,0) circle [radius = 0.025] node(2a2){}; \draw[fill] (1,-1) circle [radius = 0.025] node(2a3){};

\draw[fill] (2,2) circle [radius = 0.025] node(3a1){}; \draw[fill] (2,1) circle [radius = 0.025] node(3a2){}; \draw[fill] (2,0) circle [radius = 0.025] node(3a3){};

\draw[fill] (3,3) circle [radius = 0.025] node(4a1){}; \draw[fill] (3,2) circle [radius = 0.025] node(4a2){}; \draw[fill] (3,1) circle [radius = 0.025] node(4a3){};
\draw[fill] (3,0) circle [radius = 0.025] node(4a4){};

\draw[fill] (4,4) circle [radius = 0.025] node(5a1){}; \draw[fill] (4,3) circle [radius = 0.025] node(5a2){}; \draw[fill] (4,2) circle [radius = 0.025] node(5a3){};
\draw[fill] (4,1) circle [radius = 0.025] node(5a4){};

\draw[fill] (5,4) circle [radius = 0.025] node(6a1){}; \draw[fill] (5,3) circle [radius = 0.025] node(6a2){}; \draw[fill] (5,2) circle [radius = 0.025] node(6a3){};
\draw[fill] (5,1) circle [radius = 0.025] node(6a4){};

\draw[fill] (6,4) circle [radius = 0.025] node(7a1){}; \draw[fill] (6,3) circle [radius = 0.025] node(7a2){}; \draw[fill] (6,2) circle [radius = 0.025] node(7a3){};
\draw[fill] (6,1) circle [radius = 0.025] node(7a4){};

\draw[fill] (7,4) circle [radius = 0.025] node(8a1){}; \draw[fill] (7,3) circle [radius = 0.025] node(8a2){}; \draw[fill] (7,2) circle [radius = 0.025] node(8a3){};
\draw[fill] (7,1) circle [radius = 0.025] node(8a4){};

\draw[fill] (8,4) circle [radius = 0.025] node(9a1){}; \draw[fill] (8,3) circle [radius = 0.025] node(9a2){}; \draw[fill] (8,2) circle [radius = 0.025] node(9a3){};
\draw[fill] (8,1) circle [radius = 0.025] node(9a4){};

\draw[fill] (9,4) circle [radius = 0.025] node(10a1){}; \draw[fill] (9,3) circle [radius = 0.025] node(10a2){};
\draw[fill] (9,2) circle [radius = 0.025] node(10a3){};

\draw[fill] (10,4) circle [radius = 0.025] node(11a1){}; \draw[fill] (10,3) circle [radius = 0.025] node(11a2){};

\draw[fill] (11,4) circle [radius = 0.025] node(12a1){}; \draw[fill] (11,3) circle [radius = 0.025] node(12a2){};

\draw[fill] (12,4) circle [radius = 0.025] node(13a1){};

\draw (4.5, 5) circle [radius = 0.05] node(a){};

\draw[->, red] (0a1)--(1a2);
\draw[->, dotted, magenta] (1a1)--(0a1);

\draw[->, red] (1a1)--(2a2); \draw[->, red] (1a2)--(2a3);
\draw[->, dotted, magenta] (2a1)--(1a1); \draw[->, dotted, magenta] (2a2)--(1a2);

\draw[->, red] (2a1)--(3a2); \draw[->, red] (2a2)--(3a3);
\draw[->, dotted, magenta] (3a2)--(2a2); \draw[->, dotted, magenta] (3a1)--(2a1); \draw[->, dotted, magenta] (3a3)--(2a3);

\draw[->, red] (3a1)--(4a2); \draw[->, red] (3a2)--(4a3); \draw[->, red] (3a3)--(4a4);
\draw[->, dotted, magenta] (4a1)--(3a1); \draw[->, dotted, magenta] (4a2)--(3a2); \draw[->, dotted, magenta] (4a3)--(3a3);

\draw[->, red] (4a1)--(5a2); \draw[->, red] (4a2)--(5a3); \draw[->, red] (4a3)--(5a4);
\draw[->, dotted, magenta] (5a1)--(4a1); \draw[->, dotted, magenta] (5a2)--(4a2); \draw[->, dotted, magenta] (5a3)--(4a3);
\draw[->, dotted, magenta] (5a4)--(4a4);

\draw[->, blue] (6a1)--(5a1); \draw[->, blue] (6a2)--(5a2); \draw[->, blue] (6a3)--(5a3); \draw[->, blue] (6a4)--(5a4);
\draw[->] (5a1)--(5a2); \draw[->] (5a2)--(5a3); \draw[->] (5a3)--(5a4); \draw[->] (6a1)--(6a2); \draw[->] (6a2)--(6a3); \draw[->] (6a3)--(6a4);

\draw[->, red] (6a1)--(7a1); \draw[->, red] (6a2)--(7a2); \draw[->, red] (6a3)--(7a3); \draw[->, red] (6a4)--(7a4);
\draw[->, dotted, magenta] (7a1)--(6a2); \draw[->, dotted, magenta] (7a2)--(6a3); \draw[->, dotted, magenta] (7a3)--(6a4);

\draw[->, red] (7a1)--(8a1); \draw[->, red] (7a2)--(8a2); \draw[->, red] (7a3)--(8a3); \draw[->, red] (7a4)--(8a4);
\draw[->, dotted, magenta] (8a1)--(7a2); \draw[->, dotted, magenta] (8a2)--(7a3); \draw[->, dotted, magenta] (8a3)--(7a4);

\draw[->, blue] (9a1)--(8a1); \draw[->, blue] (9a2)--(8a2); \draw[->, blue] (9a3)--(8a3); \draw[->, blue] (9a4)--(8a4);
\draw[->] (8a1)--(8a2); \draw[->] (8a2)--(8a3); \draw[->] (8a3)--(8a4); \draw[->] (9a1)--(9a2); \draw[->] (9a2)--(9a3); \draw[->] (9a3)--(9a4);

\draw[->, red] (9a1)--(10a1); \draw[->, red] (9a2)--(10a2); \draw[->, red] (9a3)--(10a3);
\draw[->, dotted, magenta] (10a1)--(9a2); \draw[->, dotted, magenta] (10a2)--(9a3); \draw[->, dotted, magenta] (10a3)--(9a4);

\draw[->, red] (10a1)--(11a1); \draw[->, red] (10a2)--(11a2);
\draw[->, dotted, magenta] (11a1)--(10a2); \draw[->, dotted, magenta] (11a2)--(10a3);

\draw[->, red] (11a1)--(12a1); \draw[->, red] (11a2)--(12a2);
\draw[->, dotted, magenta] (12a1)--(11a2);

\draw[->, red] (12a1)--(13a1);
\draw[->, dotted, magenta] (13a1)--(12a2);

\draw[->, green] (a) to [in = 15, out = 245] (5a1);
\end{scope}

\begin{scope}[shift = {(0, 14)}]
\draw[fill] (1,1) circle [radius = 0.025] node(2b1){};

\draw[fill] (2,2) circle [radius = 0.025] node(3b1){};

\draw[fill] (3,3) circle [radius = 0.025] node(4b1){};

\draw[fill] (4,4) circle [radius = 0.025] node(5b1){}; \draw[fill] (4,3) circle [radius = 0.025] node(5b2){};

\draw[fill] (5,4) circle [radius = 0.025] node(6b1){}; \draw[fill] (5,3) circle [radius = 0.025] node(6b2){};

\draw[fill] (6,5) circle [radius = 0.025] node(7b1){}; \draw[fill] (6,4) circle [radius = 0.025] node(7b2){};

\draw[fill] (7,6) circle [radius = 0.025] node(8b1){}; \draw[fill] (7,5) circle [radius = 0.025] node(8b2){};

\draw[fill] (8,5) circle [radius = 0.025] node(9b1){};

\draw[fill] (9,5) circle [radius = 0.025] node(10b1){};

\draw[fill] (10,5) circle [radius = 0.025] node(11b1){};

\draw[fill] (11,5) circle [radius = 0.025] node(12b1){};

\draw[fill] (12,5) circle [radius = 0.025] node(13b1){};

\draw (7.5, 7) circle [radius = 0.05] node(b){};

\draw[->, dotted, magenta] (3b1)--(2b1);

\draw[->, dotted, magenta] (4b1)--(3b1);

\draw[->, red] (4b1)--(5b2);
\draw[->, dotted, magenta] (5b1)--(4b1);

\draw[->, blue] (6b1)--(5b1); \draw[->, blue] (6b2)--(5b2);
\draw[->] (5b1)--(5b2); \draw[->] (6b1)--(6b2);

\draw[->, red] (6b1)--(7b2);
\draw[->, dotted, magenta] (7b1)--(6b1); \draw[->, dotted, magenta] (7b2)--(6b2);

\draw[->, red] (7b1)--(8b2);
\draw[->, dotted, magenta] (8b1)--(7b1); \draw[->, dotted, magenta] (8b2)--(7b2);

\draw[->, blue] (9b1)--(8b2);
\draw[->] (8b1)--(8b2);

\draw[->, red] (9b1)--(10b1);

\draw[->, red] (10b1)--(11b1);

\draw[->, red] (11b1)--(12b1);

\draw[->, red] (12b1)--(13b1);

\draw[->, green] (b) to [in = 15, out = 245] (8b1);
\end{scope}

\begin{scope}[shift = {(0, 20)}]
\draw[gray] (-0.1, -0.1)--(1.05, -0.1)--(2.05, 0.9)--(3.05, 0.9)--(4.05, 1.9)--(8.05, 1.9)--(10.05, 3.9)--(11.1, 3.9)--(11.1, 4.1)--(9.95, 4.1)--(6.95, 4.1)--(4.95, 2.1)--(2.95, 2.1)--(0.95, 0.1)--(-0.1, 0.1)--(-0.1, -0.1);
\end{scope}

\begin{scope}[shift = {(0, 14)}]
\draw[gray] (-0.1, -0.1)--(1.05, -0.1)--(2.05, 0.9)--(3.05, 0.9)--(4.05, 1.9)--(8.05, 1.9)--(10.05, 3.9)--(11.1, 3.9)--(11.1, 4.1)--(9.95, 4.1)--(6.95, 4.1)--(4.95, 2.1)--(2.95, 2.1)--(0.95, 0.1)--(-0.1, 0.1)--(-0.1, -0.1);
\end{scope}

\draw[->, red, dashed] (2b1) to[out = 0, in = 180] (3a3); \draw[->, magenta, dashed] (2b1) to[out = 225, in = 45] (1a2);

\draw[->, red, dashed] (3b1) to[out = 0, in = 180] (4a3);

\draw[->, magenta, dashed] (5b2) to[out = 225, in = 45] (4a3); \draw[->, dashed] (5b2) to[out = 270, in = 290] (5a4);

\draw[->, red, dashed] (6b2) to[out = 0, in = 180] (7a3); \draw[->, dashed] (6b2) to[out = 270, in = 255] (6a4);

\draw[->, red, dashed] (7b2) to[out = 0, in = 180] (8a2);

\draw[->, dashed] (8b2) to[out = 270, in = 290] (8a2);

\draw[->, dashed] (9b1) to[out = 270, in = 255] (9a2);

\draw[->, magenta, dashed] (10b1) to[out = 225, in = 45] (9a2);

\draw[->, magenta, dashed] (11b1) to[out = 225, in = 45] (10a2);

\draw[->, magenta, dashed] (12b1) to[out = 225, in = 45] (11a2);

\draw[->, magenta, dashed] (13b1) to[out = 225, in = 45] (12a2);
\end{tikzpicture}
\vspace{-3mm}

\caption{Explicit construction of an invariant curve corresponding to butterfly surgery. The dashed lines are the new edges added to the butterfly diagram. Taking the $\torus$-orbit closure of the element of $\prequot$ represented by this graph yields an invariant curve with tangent weight \smash{$\bigl(\frac{\ub_1}{\ub_2}\hb\bigr)^{\pm 1}$} at its fixed points.} \label{fig:invariant_curve}
\vspace{-2mm}
\end{figure}

Next, set $\ub' = \hb = 1$, so that the new edges are acted upon by $\ub$. Apply the $\G$-action that multiplies each vertex in $\site$ by $\ub^{-1}$. Since all edges in $\beta_1$ adjacent to $\site$ point into $\site$, those edges will be multiplied by $\ub^{-1}$. Meanwhile, the factor on the new edges will cancel to 1, and all other edges will be unchanged. Taking $\ub \to \infty$, the edges adjacent to $\site$ are killed, leaving behind $\beta_2$ and $\beta_2'$. Thus, $p_2\in\gamma$.

There is a surjective map $\CC^\times \to \torus.p_\gamma$ given by sending $t\in\CC^\times$ to the point obtained by multiplying the new edges by $t$. It follows that $\gamma$ is at most 1-dimensional. Since $\gamma$ contains two distinct fixed points $p_1$, $p_2$, it must be that $\dim(\gamma) = 1$. Therefore, $\gamma$ is a compact invariant curve. Its tangent weight at $p_1$ is precisely $\hb^{-\Delta_\surgery y} \ub / \ub'$, and its tangent weight at $p_2$ is $\hb^{\Delta_\surgery y} \ub' / \ub$.

\begin{figure}[t]
\centering
\begin{tikzpicture}
\begin{scope}[shift = {(0, 21)}]
\draw[fill] (-1,-1) circle [radius = 0.025] node(0a1){};

\draw[fill] (0,0) circle [radius = 0.025] node(1a1){}; \draw[fill] (0,-1) circle [radius = 0.025] node(1a2){};

\draw[fill] (1,1) circle [radius = 0.025] node(2a1){}; \draw[fill] (1,0) circle [radius = 0.025] node(2a2){}; \draw[fill] (1,-1) circle [radius = 0.025] node(2a3){};

\draw[fill] (2,2) circle [radius = 0.025] node(3a1){}; \draw[fill] (2,1) circle [radius = 0.025] node(3a2){}; \draw[fill] (2,0) circle [radius = 0.025] node(3a3){};

\draw[fill] (3,3) circle [radius = 0.025] node(4a1){}; \draw[fill] (3,2) circle [radius = 0.025] node(4a2){}; \draw[fill] (3,1) circle [radius = 0.025] node(4a3){};
\draw[fill] (3,0) circle [radius = 0.025] node(4a4){};

\draw[fill] (4,4) circle [radius = 0.025] node(5a1){}; \draw[fill] (4,3) circle [radius = 0.025] node(5a2){}; \draw[fill] (4,2) circle [radius = 0.025] node(5a3){};
\draw[fill] (4,1) circle [radius = 0.025] node(5a4){};

\draw[fill] (5,4) circle [radius = 0.025] node(6a1){}; \draw[fill] (5,3) circle [radius = 0.025] node(6a2){}; \draw[fill] (5,2) circle [radius = 0.025] node(6a3){};
\draw[fill] (5,1) circle [radius = 0.025] node(6a4){};

\draw[fill] (6,4) circle [radius = 0.025] node(7a1){}; \draw[fill] (6,3) circle [radius = 0.025] node(7a2){}; \draw[fill] (6,2) circle [radius = 0.025] node(7a3){};
\draw[fill] (6,1) circle [radius = 0.025] node(7a4){};

\draw[fill] (7,4) circle [radius = 0.025] node(8a1){}; \draw[fill] (7,3) circle [radius = 0.025] node(8a2){}; \draw[fill] (7,2) circle [radius = 0.025] node(8a3){};
\draw[fill] (7,1) circle [radius = 0.025] node(8a4){};

\draw[fill] (8,4) circle [radius = 0.025] node(9a1){}; \draw[fill] (8,3) circle [radius = 0.025] node(9a2){}; \draw[fill] (8,2) circle [radius = 0.025] node(9a3){};
\draw[fill] (8,1) circle [radius = 0.025] node(9a4){};

\draw[fill] (9,4) circle [radius = 0.025] node(10a1){}; \draw[fill] (9,3) circle [radius = 0.025] node(10a2){};
\draw[fill] (9,2) circle [radius = 0.025] node(10a3){};

\draw[fill] (10,4) circle [radius = 0.025] node(11a1){}; \draw[fill] (10,3) circle [radius = 0.025] node(11a2){};

\draw[fill] (11,4) circle [radius = 0.025] node(12a1){}; \draw[fill] (11,3) circle [radius = 0.025] node(12a2){};

\draw[fill] (12,4) circle [radius = 0.025] node(13a1){};

\draw (4.5, 5) circle [radius = 0.05] node(a){};

\draw[->, red] (0a1)--(1a2);
\draw[->, dotted, magenta] (1a1)--(0a1);

\draw[->, red] (1a1)--(2a2); \draw[->, red] (1a2)--(2a3);
\draw[->, dotted, magenta] (2a1)--(1a1); \draw[->, dotted, magenta] (2a2)--(1a2);

\draw[->, red] (2a1)--(3a2); \draw[->, red] (2a2)--(3a3);
\draw[->, dotted, magenta] (3a2)--(2a2); \draw[->, dotted, magenta] (3a1)--(2a1); \draw[->, dotted, magenta] (3a3)--(2a3);

\draw[->, red] (3a1)--(4a2); \draw[->, red] (3a2)--(4a3); \draw[->, red] (3a3)--(4a4);
\draw[->, dotted, magenta] (4a1)--(3a1); \draw[->, dotted, magenta] (4a2)--(3a2); \draw[->, dotted, magenta] (4a3)--(3a3);

\draw[->, red] (4a1)--(5a2); \draw[->, red] (4a2)--(5a3); \draw[->, red] (4a3)--(5a4);
\draw[->, dotted, magenta] (5a1)--(4a1); \draw[->, dotted, magenta] (5a2)--(4a2); \draw[->, dotted, magenta] (5a3)--(4a3);
\draw[->, dotted, magenta] (5a4)--(4a4);

\draw[->, blue] (6a1)--(5a1); \draw[->, blue] (6a2)--(5a2); \draw[->, blue] (6a3)--(5a3); \draw[->, blue] (6a4)--(5a4);
\draw[->] (5a1)--(5a2); \draw[->] (5a2)--(5a3); \draw[->] (5a3)--(5a4); \draw[->] (6a1)--(6a2); \draw[->] (6a2)--(6a3); \draw[->] (6a3)--(6a4);

\draw[->, red] (6a1)--(7a1); \draw[->, red] (6a2)--(7a2); \draw[->, red] (6a3)--(7a3); \draw[->, red] (6a4)--(7a4);
\draw[->, dotted, magenta] (7a1)--(6a2); \draw[->, dotted, magenta] (7a2)--(6a3); \draw[->, dotted, magenta] (7a3)--(6a4);

\draw[->, red] (7a1)--(8a1); \draw[->, red] (7a2)--(8a2); \draw[->, red] (7a3)--(8a3); \draw[->, red] (7a4)--(8a4);
\draw[->, dotted, magenta] (8a1)--(7a2); \draw[->, dotted, magenta] (8a2)--(7a3); \draw[->, dotted, magenta] (8a3)--(7a4);

\draw[->, blue] (9a1)--(8a1); \draw[->, blue] (9a2)--(8a2); \draw[->, blue] (9a3)--(8a3); \draw[->, blue] (9a4)--(8a4);
\draw[->] (8a1)--(8a2); \draw[->] (8a2)--(8a3); \draw[->] (8a3)--(8a4); \draw[->] (9a1)--(9a2); \draw[->] (9a2)--(9a3); \draw[->] (9a3)--(9a4);

\draw[->, red] (9a1)--(10a1); \draw[->, red] (9a2)--(10a2); \draw[->, red] (9a3)--(10a3);
\draw[->, dotted, magenta] (10a1)--(9a2); \draw[->, dotted, magenta] (10a2)--(9a3); \draw[->, dotted, magenta] (10a3)--(9a4);

\draw[->, red] (10a1)--(11a1); \draw[->, red] (10a2)--(11a2);
\draw[->, dotted, magenta] (11a1)--(10a2); \draw[->, dotted, magenta] (11a2)--(10a3);

\draw[->, red] (11a1)--(12a1); \draw[->, red] (11a2)--(12a2);
\draw[->, dotted, magenta] (12a1)--(11a2);

\draw[->, red] (12a1)--(13a1);
\draw[->, dotted, magenta] (13a1)--(12a2);

\draw[->, green] (a) to [in = 15, out = 245] (5a1);
\end{scope}

\begin{scope}[shift = {(0, 14)}]
\draw[fill] (1,1) circle [radius = 0.025] node(2b1){};

\draw[fill] (2,2) circle [radius = 0.025] node(3b1){};

\draw[fill] (3,3) circle [radius = 0.025] node(4b1){};

\draw[fill] (4,4) circle [radius = 0.025] node(5b1){}; \draw[fill] (4,3) circle [radius = 0.025] node(5b2){};

\draw[fill] (5,4) circle [radius = 0.025] node(6b1){}; \draw[fill] (5,3) circle [radius = 0.025] node(6b2){};

\draw[fill] (6,5) circle [radius = 0.025] node(7b1){}; \draw[fill] (6,4) circle [radius = 0.025] node(7b2){};

\draw[fill] (7,6) circle [radius = 0.025] node(8b1){}; \draw[fill] (7,5) circle [radius = 0.025] node(8b2){};

\draw[fill] (8,5) circle [radius = 0.025] node(9b1){};

\draw[fill] (9,5) circle [radius = 0.025] node(10b1){};

\draw[fill] (10,5) circle [radius = 0.025] node(11b1){};

\draw[fill] (11,5) circle [radius = 0.025] node(12b1){};

\draw[fill] (12,5) circle [radius = 0.025] node(13b1){};

\draw (7.5, 7) circle [radius = 0.05] node(b){};

\draw[->, dotted, magenta] (3b1)--(2b1);

\draw[->, dotted, magenta] (4b1)--(3b1);

\draw[->, red] (4b1)--(5b2);
\draw[->, dotted, magenta] (5b1)--(4b1);

\draw[->, blue] (6b1)--(5b1); \draw[->, blue] (6b2)--(5b2);
\draw[->] (5b1)--(5b2); \draw[->] (6b1)--(6b2);

\draw[->, red] (6b1)--(7b2);
\draw[->, dotted, magenta] (7b1)--(6b1); \draw[->, dotted, magenta] (7b2)--(6b2);

\draw[->, red] (7b1)--(8b2);
\draw[->, dotted, magenta] (8b1)--(7b1); \draw[->, dotted, magenta] (8b2)--(7b2);

\draw[->, blue] (9b1)--(8b2);
\draw[->] (8b1)--(8b2);

\draw[->, red] (9b1)--(10b1);

\draw[->, red] (10b1)--(11b1);

\draw[->, red] (11b1)--(12b1);

\draw[->, red] (12b1)--(13b1);

\draw[->, green] (b) to [in = 15, out = 245] (8b1);
\end{scope}

\draw[gray] (4a4) +(-0.1,0.1)--+(0.1,0.1)--+(0.1,-0.1)--+(-0.1,-0.1)--+(-0.1,0.1);
\draw[gray] (7a4) ++(-0.1,0) ++(-0.1,-0.1)--++(2.2,0)--++(1.2,1.2)--++(-2.2,0)--++(-1.2,-1.2);

\draw[gray] (4a4) ++(0,-5) +(-0.1,0.1)--+(0.1,0.1)--+(0.1,-0.1)--+(-0.1,-0.1)--+(-0.1,0.1);
\draw[gray] (7a4) ++(-0.1,-5) ++(-0.1,-0.1)--++(2.2,0)--++(1.2,1.2)--++(-2.2,0)--++(-1.2,-1.2);

\draw (3a3) node at +(0,-2.5){$z_1$};
\draw[red,dashed,->] (3b1) to[in=180,out=0] (4a4);
\draw[magenta,dashed,->] (5b2) to[in=45,out=225] (4a4);

\draw (6a4) node at +(0,-2.5){$z_2$};
\draw[red,dashed,->] (6b2) to[in=180,out=0] (7a4);
\draw[red,dashed,->] (7b2) to[in=180,out=0] (8a3);
\draw[magenta,dashed,->] (10b1) to[in=45,out=225] (9a3);
\draw[magenta,dashed,->] (11b1) to[in=45,out=225] (10a3);
\draw[->, dashed] (8b2) to[out = 270, in = 290] (8a3);
\draw[->, dashed] (9b1) to[out = 270, in = 255] (9a3);
\end{tikzpicture}

\caption{Explicit 2-dimensional pencil of invariant curves with tangent weight \smash{$\bigl(\frac{\ub_1}{\ub_2}\bigr)^{\pm 1}$} at its fixed points.}
\label{fig:invariant_pencil}
\end{figure}

From general considerations, $\gamma$ is isomorphic to $\CCP^1$. We obtain an explicit parametrization by multiplying each of the new edges by $t\in\CC$. This gives us an isomorphism $\CC \to \gamma \setminus\{p_2\}$. Compactifying $\CC$ with a point $\infty$ and mapping $\infty$ to $p_2$ yields a parametrized curve $\gamma\colon\CCP^1 \to \bowvar(\brane)$ (we are abusing notation and using $\gamma$ to refer both to the parametrization and its image). Note that performing the inverse butterfly surgery yields the same invariant curve with a different parametrization, one with $\gamma(0) = p_2$.

Suppose $\site$ has $k$ connected components $\site_1,\dots ,\site_k$. By multiplying the new edges pointing into $\site_i$ by $z_i$, where $z = (z_1, \dots , z_k)\in(\CC\setminus\{0\})^k$, we obtain a $k$-parameter family of $\torus$-invariant curves containing $p_1$ and $p_2$. In the closure of their union, there are other fixed points formed by stacking some of the connected components but not others, as well as invariant curves between them. For each subset $I\subset \{1,\dots ,k\}$, consider the fixed point $p_I$ to be obtained from the butterfly surgery with site the union of components $\site_i$ for $i\in I$. So $p_1 = p_\varnothing$ and \smash{$p_{\{1,\dots ,k\}} = p_2$}.

Now consider the hypercube \smash{$\bigl(\CCP^1\bigr)^k$} endowed with the diagonal $\torus$-action with every tangent weight equal to \smash{$\hb^{-\Delta_\surgery y} \ub / \ub'$} at $(0,\dots ,0)$. We obtain a fixed point \smash{$x_I\in\bigl(\CCP^1\bigr)^k$} for each $I\subset \{1,\dots ,k\}$ where the $i$-th component of $x_I$ is $\infty$ if $i\in I$ and $0$ otherwise. Next we obtain a~$\torus$-equivariant map \smash{$\pen_\surgery\colon\bigl(\CCP^1\bigr)^k \to \bowvar(\brane)$} that sends each $x_I$ to $p_I$. We define this map first on the interior and on the facets containing $x_\varnothing$ by sending each $z = (z_1, \dots , z_k)\in\CC^k$ to the butterfly diagram of $p_\varnothing$ with new edges of $\surgery$ pointing into $\site_i$ multiplied by $z_i$, providing all curves containing $p_\varnothing$ in the image of $\pen_\surgery$. We similarly define that for each \smash{$z = (z_1, \dots , z_k)\in\bigl(\CCP^1\bigr)^k$} with $\infty$ in each component of index in $I\subset\{1,\dots ,k\}$, $\pen_\surgery$ maps $z$ to the butterfly diagram of~$p_I$ with all new edges of $\surgery$ that still can be drawn in the butterfly diagram of $p_I$ (the ones with site components $\site_i$ where $z_i$ is finite) multiplied by the appropriate finite $z$-values. Despite the fact that only the curves in $\pen_\surgery\bigl((\CC\setminus\{0\})^k\bigr)$ are between $p_1$ and $p_2$, we will still call $\pen_\surgery$ a pencil of invariant curves between $p_1$ and $p_2$.
Consider the $k$ individual butterfly surgeries corresponding to each of the $k$ connected components of $\site$. These butterfly surgeries produce invariant curves $\gamma_1, \dots , \gamma_k$. We will show in Proposition~\ref{prop:pencil_generation} that the tangent spaces of these curves at $p_1$ (and also at $p_2$) are linearly independent. Moreover, we have
\[
\forall z\in\CC^k\setminus\{0\},\qquad \left. \frac{\partial}{\partial t} \right|_{t=0} \pen_\surgery(tz) = \sum_{i = 1}^k z_i\gamma_i'(0).
\]
It follows that each invariant curve in $\Gamma_\surgery$ has a distinct tangent space at $p_1$ and is therefore distinct. Similar considerations can be made for each other fixed point $p_I$. In other words, $\pen_\surgery$~is injective. The results of this section are summarized in the following lemma.

\begin{Lemma} \label{lemma:butterfly_surgery}
Suppose a butterfly surgery $\surgery$ sends a $\torus$-fixed point $p_1$ to the $\torus$-fixed point $p_2$. Let the site $\site$ be moved from the $U$ butterfly to the $U'$ butterfly by $\surgery$, where $k$ is the number of connected components of $\site$. Then, there is a $k$-dimensional pencil \smash{$\pen_\surgery\colon\bigl(\CCP^1\bigr)^k\to\bowvar(\brane)$}, a~$\torus$-equ\-ivariant map where $\pen_\surgery\bigl((\CC\setminus\{0\})^k\bigr)$ consists of compact invariant curves containing $p_1$ and~$p_2$. Moreover, each curve in the pencil has tangent weight
\smash{$
\frac{\ub}{\ub'}\hb^{-\Delta_\surgery y}$},
at $p_1$, where
$
\Delta_\surgery y = y(v') - y(v)$,
for any $v\in\site$ and corresponding $v'\in\site'$. Here, $y$ denotes the height of a vertex $($see Definition~{\rm\ref{def:height})}.
\end{Lemma}

It will be shown that butterfly surgeries capture all of the compact invariant curves in $\bowvar(\brane)$. To do so, our approach also requires us to construct noncompact invariant curves.

\subsection{Botched butterfly surgery} \label{sec:botched_surgery}
For the remainder of Section~\ref{sec:curves}, we will assume that $M$ is separated. Analogous results will follow for the nonseparated case by tracing the relevant constructions through Hanany--Witten transitions. See the proof of \cite[Proposition~8.1]{NT} for the explicit construction of the Hanany--Witten isomorphism. We explicitly trace a butterfly through the Hanany--Witten transition in~\cite[Section~3.2.4]{S}. First, let us define a surgery operation that closely resembles a butterfly surgery, but is not exactly a butterfly surgery.

Let $\brane$ be a separated brane diagram with $n$ NS5 branes and $m$ D5 branes. As in Section~\ref{sec:bct_weights}, index the branes from left to right. The NS5 branes are denoted $V_1, \dots , V_n$, the D5 branes are denoted $U_1, \dots , U_m$, and the segments (D3 branes) are denoted $X_0, X_1, \dots , X_{n+m}$, where $X_0$ and $X_{n+m}$ are the infinite left and right segments. We will divide $\brane$ into a ``left side'' and a~``right side''. The left side comprises $X_0, \dots , X_n$ and $V_1, \dots , V_n$. The right side is the remaining portion of the diagram. The notion of left and right side extends naturally to butterflies and butterfly diagrams. The left side of a butterfly always has a certain inverted staircase-like shape. See Figure~\ref{fig:butterfly_sides} for an example.

\begin{figure}[th!]
\centering
\begin{tikzpicture}[scale=0.3]
\draw[dashed, cyan] (29, 2)--(29, -58);

\draw[thick,red] (-0.5,-1)--(0.5,1) node[pos=0.5](V1){} node[pos=1,inner sep=0](V1t){} node[pos=0,inner sep=0](V1b){};
\draw[thick,red] (2.5,-1)--(3.5,1) node[pos=0.5](V2){} node[pos=1,inner sep=0](V2t){} node[pos=0,inner sep=0](V2b){};
\draw[thick,red] (5.5,-1)--(6.5,1) node[pos=0.5](V3){} node[pos=1,inner sep=0](V3t){} node[pos=0,inner sep=0](V3b){};
\draw[thick,red] (8.5,-1)--(9.5,1) node[pos=0.5](V4){} node[pos=1,inner sep=0](V4t){} node[pos=0,inner sep=0](V4b){};
\draw[thick,red] (11.5,-1)--(12.5,1) node[pos=0.5](V5){} node[pos=1,inner sep=0](V5t){} node[pos=0,inner sep=0](V5b){};
\draw[thick,red] (14.5,-1)--(15.5,1) node[pos=0.5](V6){} node[pos=1,inner sep=0](V6t){} node[pos=0,inner sep=0](V6b){};
\draw[thick,red] (17.5,-1)--(18.5,1) node[pos=0.5](V7){} node[pos=1,inner sep=0](V7t){} node[pos=0,inner sep=0](V7b){};
\draw[thick,red] (20.5,-1)--(21.5,1) node[pos=0.5](V8){} node[pos=1,inner sep=0](V8t){} node[pos=0,inner sep=0](V8b){};
\draw[thick,red] (23.5,-1)--(24.5,1) node[pos=0.5](V9){} node[pos=1,inner sep=0](V9t){} node[pos=0,inner sep=0](V9b){};
\draw[thick,red] (26.5,-1)--(27.5,1) node[pos=0.5](V10){} node[pos=1,inner sep=0](V10t){} node[pos=0,inner sep=0](V10b){};
\draw[thick,blue] (30.5,-1)--(29.5,1) node[pos=0.5](U1){} node[pos=1,inner sep=0](U1t){} node[pos=0,inner sep=0](U1b){};
\draw[thick,blue] (33.5,-1)--(32.5,1) node[pos=0.5](U2){} node[pos=1,inner sep=0](U2t){} node[pos=0,inner sep=0](U2b){};
\draw[thick,blue] (36.5,-1)--(35.5,1) node[pos=0.5](U3){} node[pos=1,inner sep=0](U3t){} node[pos=0,inner sep=0](U3b){};

\draw (V1)--(V2) node[pos=0.5, above](X1){1};
\draw (V2)--(V3) node[pos=0.5, above](X2){2};
\draw (V3)--(V4) node[pos=0.5, above](X3){3};
\draw (V4)--(V5) node[pos=0.5, above](X4){4};
\draw (V5)--(V6) node[pos=0.5, above](X5){5};
\draw (V6)--(V7) node[pos=0.5, above](X6){7};
\draw (V7)--(V8) node[pos=0.5, above](X7){9};
\draw (V8)--(V9) node[pos=0.5, above](X8){9};
\draw (V9)--(V10) node[pos=0.5, above](X9){10};
\draw (V10)--(U1) node[pos=0.5, above](X10){11};
\draw (U1)--(U2) node[pos=0.5, above](X11){6};
\draw (U2)--(U3) node[pos=0.5, above](X12){4};

\node () at (0,-6) {};

\draw (30,-3) circle[radius=0.14] node(u1){};
\draw[fill] (1.5,-33) circle[radius=0.07] node(u1a1a1){};
\draw[fill] (4.5,-30) circle[radius=0.07] node(u1a2a1){};
\draw[fill] (7.5,-27) circle[radius=0.07] node(u1a3a1){};
\draw[fill] (10.5,-24) circle[radius=0.07] node(u1a4a1){};
\draw[fill] (10.5,-27) circle[radius=0.07] node(u1a4a2){};
\draw[fill] (13.5,-21) circle[radius=0.07] node(u1a5a1){};
\draw[fill] (13.5,-24) circle[radius=0.07] node(u1a5a2){};
\draw[fill] (13.5,-27) circle[radius=0.07] node(u1a5a3){};
\draw[fill] (16.5,-18) circle[radius=0.07] node(u1a6a1){};
\draw[fill] (16.5,-21) circle[radius=0.07] node(u1a6a2){};
\draw[fill] (16.5,-24) circle[radius=0.07] node(u1a6a3){};
\draw[fill] (16.5,-27) circle[radius=0.07] node(u1a6a4){};
\draw[fill] (19.5,-15) circle[radius=0.07] node(u1a7a1){};
\draw[fill] (19.5,-18) circle[radius=0.07] node(u1a7a2){};
\draw[fill] (19.5,-21) circle[radius=0.07] node(u1a7a3){};
\draw[fill] (19.5,-24) circle[radius=0.07] node(u1a7a4){};
\draw[fill] (22.5,-12) circle[radius=0.07] node(u1a8a1){};
\draw[fill] (22.5,-15) circle[radius=0.07] node(u1a8a2){};
\draw[fill] (22.5,-18) circle[radius=0.07] node(u1a8a3){};
\draw[fill] (22.5,-21) circle[radius=0.07] node(u1a8a4){};
\draw[fill] (25.5,-9) circle[radius=0.07] node(u1a9a1){};
\draw[fill] (25.5,-12) circle[radius=0.07] node(u1a9a2){};
\draw[fill] (25.5,-15) circle[radius=0.07] node(u1a9a3){};
\draw[fill] (25.5,-18) circle[radius=0.07] node(u1a9a4){};
\draw[fill] (25.5,-21) circle[radius=0.07] node(u1a9a5){};
\draw[fill] (28.5,-6) circle[radius=0.07] node(u1a10a1){};
\draw[fill] (28.5,-9) circle[radius=0.07] node(u1a10a2){};
\draw[fill] (28.5,-12) circle[radius=0.07] node(u1a10a3){};
\draw[fill] (28.5,-15) circle[radius=0.07] node(u1a10a4){};
\draw[fill] (28.5,-18) circle[radius=0.07] node(u1a10a5){};

\draw[->, green] (u1) to[out=-120,in=70] (u1a10a1);
\draw[->,dotted,magenta] (u1a2a1)--(u1a1a1);
\draw[->,dotted,magenta] (u1a3a1)--(u1a2a1);
\draw[->,dotted,magenta] (u1a4a1)--(u1a3a1);
\draw[->,red] (u1a3a1)--(u1a4a2);
\draw[->,dotted,magenta] (u1a5a1)--(u1a4a1);
\draw[->,dotted,magenta] (u1a5a2)--(u1a4a2);
\draw[->,red] (u1a4a1)--(u1a5a2);
\draw[->,red] (u1a4a2)--(u1a5a3);
\draw[->,dotted,magenta] (u1a6a1)--(u1a5a1);
\draw[->,dotted,magenta] (u1a6a2)--(u1a5a2);
\draw[->,dotted,magenta] (u1a6a3)--(u1a5a3);
\draw[->,red] (u1a5a1)--(u1a6a2);
\draw[->,red] (u1a5a2)--(u1a6a3);
\draw[->,red] (u1a5a3)--(u1a6a4);
\draw[->,dotted,magenta] (u1a7a1)--(u1a6a1);
\draw[->,dotted,magenta] (u1a7a2)--(u1a6a2);
\draw[->,dotted,magenta] (u1a7a3)--(u1a6a3);
\draw[->,dotted,magenta] (u1a7a4)--(u1a6a4);
\draw[->,red] (u1a6a1)--(u1a7a2);
\draw[->,red] (u1a6a2)--(u1a7a3);
\draw[->,red] (u1a6a3)--(u1a7a4);
\draw[->,dotted,magenta] (u1a8a1)--(u1a7a1);
\draw[->,dotted,magenta] (u1a8a2)--(u1a7a2);
\draw[->,dotted,magenta] (u1a8a3)--(u1a7a3);
\draw[->,dotted,magenta] (u1a8a4)--(u1a7a4);
\draw[->,red] (u1a7a1)--(u1a8a2);
\draw[->,red] (u1a7a2)--(u1a8a3);
\draw[->,red] (u1a7a3)--(u1a8a4);
\draw[->,dotted,magenta] (u1a9a1)--(u1a8a1);
\draw[->,dotted,magenta] (u1a9a2)--(u1a8a2);
\draw[->,dotted,magenta] (u1a9a3)--(u1a8a3);
\draw[->,dotted,magenta] (u1a9a4)--(u1a8a4);
\draw[->,red] (u1a8a1)--(u1a9a2);
\draw[->,red] (u1a8a2)--(u1a9a3);
\draw[->,red] (u1a8a3)--(u1a9a4);
\draw[->,red] (u1a8a4)--(u1a9a5);
\draw[->,dotted,magenta] (u1a10a1)--(u1a9a1);
\draw[->,dotted,magenta] (u1a10a2)--(u1a9a2);
\draw[->,dotted,magenta] (u1a10a3)--(u1a9a3);
\draw[->,dotted,magenta] (u1a10a4)--(u1a9a4);
\draw[->,dotted,magenta] (u1a10a5)--(u1a9a5);
\draw[->,red] (u1a9a1)--(u1a10a2);
\draw[->,red] (u1a9a2)--(u1a10a3);
\draw[->,red] (u1a9a3)--(u1a10a4);
\draw[->,red] (u1a9a4)--(u1a10a5);
\draw[->] (u1a10a1)--(u1a10a2);
\draw[->] (u1a10a2)--(u1a10a3);
\draw[->] (u1a10a3)--(u1a10a4);
\draw[->] (u1a10a4)--(u1a10a5);

\begin{scope}[yshift=9cm]
\node () at (0,-33) {};

\draw (33,-30) circle[radius=0.14] node(u2){};
\draw[fill] (16.5,-45) circle[radius=0.07] node(u2a6a1){};
\draw[fill] (19.5,-42) circle[radius=0.07] node(u2a7a1){};
\draw[fill] (19.5,-45) circle[radius=0.07] node(u2a7a2){};
\draw[fill] (22.5,-39) circle[radius=0.07] node(u2a8a1){};
\draw[fill] (22.5,-42) circle[radius=0.07] node(u2a8a2){};
\draw[fill] (25.5,-36) circle[radius=0.07] node(u2a9a1){};
\draw[fill] (25.5,-39) circle[radius=0.07] node(u2a9a2){};
\draw[fill] (28.5,-33) circle[radius=0.07] node(u2a10a1){};
\draw[fill] (28.5,-36) circle[radius=0.07] node(u2a10a2){};
\draw[fill] (31.5,-33) circle[radius=0.07] node(u2a11a1){};
\draw[fill] (31.5,-36) circle[radius=0.07] node(u2a11a2){};

\draw[->, green] (u2) to[out=-120,in=70] (u2a11a1);
\draw[->,dotted,magenta] (u2a7a1)--(u2a6a1);
\draw[->,red] (u2a6a1)--(u2a7a2);
\draw[->,dotted,magenta] (u2a8a1)--(u2a7a1);
\draw[->,dotted,magenta] (u2a8a2)--(u2a7a2);
\draw[->,red] (u2a7a1)--(u2a8a2);
\draw[->,dotted,magenta] (u2a9a1)--(u2a8a1);
\draw[->,dotted,magenta] (u2a9a2)--(u2a8a2);
\draw[->,red] (u2a8a1)--(u2a9a2);
\draw[->,dotted,magenta] (u2a10a1)--(u2a9a1);
\draw[->,dotted,magenta] (u2a10a2)--(u2a9a2);
\draw[->,red] (u2a9a1)--(u2a10a2);
\draw[->, blue] (u2a11a2)--(u2a10a2);
\draw[->, blue] (u2a11a1)--(u2a10a1);
\draw[->] (u2a10a1)--(u2a10a2);
\draw[->] (u2a11a1)--(u2a11a2);
\end{scope}
\begin{scope}[yshift=9cm]
\node () at (0,-45) {};

\draw (36,-42) circle[radius=0.14] node(u3){};
\draw[fill] (4.5,-69) circle[radius=0.07] node(u3a2a1){};
\draw[fill] (7.5,-66) circle[radius=0.07] node(u3a3a1){};
\draw[fill] (7.5,-69) circle[radius=0.07] node(u3a3a2){};
\draw[fill] (10.5,-63) circle[radius=0.07] node(u3a4a1){};
\draw[fill] (10.5,-66) circle[radius=0.07] node(u3a4a2){};
\draw[fill] (13.5,-60) circle[radius=0.07] node(u3a5a1){};
\draw[fill] (13.5,-63) circle[radius=0.07] node(u3a5a2){};
\draw[fill] (16.5,-57) circle[radius=0.07] node(u3a6a1){};
\draw[fill] (16.5,-60) circle[radius=0.07] node(u3a6a2){};
\draw[fill] (19.5,-54) circle[radius=0.07] node(u3a7a1){};
\draw[fill] (19.5,-57) circle[radius=0.07] node(u3a7a2){};
\draw[fill] (19.5,-60) circle[radius=0.07] node(u3a7a3){};
\draw[fill] (22.5,-51) circle[radius=0.07] node(u3a8a1){};
\draw[fill] (22.5,-54) circle[radius=0.07] node(u3a8a2){};
\draw[fill] (22.5,-57) circle[radius=0.07] node(u3a8a3){};
\draw[fill] (25.5,-48) circle[radius=0.07] node(u3a9a1){};
\draw[fill] (25.5,-51) circle[radius=0.07] node(u3a9a2){};
\draw[fill] (25.5,-54) circle[radius=0.07] node(u3a9a3){};
\draw[fill] (28.5,-45) circle[radius=0.07] node(u3a10a1){};
\draw[fill] (28.5,-48) circle[radius=0.07] node(u3a10a2){};
\draw[fill] (28.5,-51) circle[radius=0.07] node(u3a10a3){};
\draw[fill] (28.5,-54) circle[radius=0.07] node(u3a10a4){};
\draw[fill] (31.5,-45) circle[radius=0.07] node(u3a11a1){};
\draw[fill] (31.5,-48) circle[radius=0.07] node(u3a11a2){};
\draw[fill] (31.5,-51) circle[radius=0.07] node(u3a11a3){};
\draw[fill] (31.5,-54) circle[radius=0.07] node(u3a11a4){};
\draw[fill] (34.5,-45) circle[radius=0.07] node(u3a12a1){};
\draw[fill] (34.5,-48) circle[radius=0.07] node(u3a12a2){};
\draw[fill] (34.5,-51) circle[radius=0.07] node(u3a12a3){};
\draw[fill] (34.5,-54) circle[radius=0.07] node(u3a12a4){};

\draw[->, green] (u3) to[out=-120,in=70] (u3a12a1);
\draw[->,dotted,magenta] (u3a3a1)--(u3a2a1);
\draw[->,red] (u3a2a1)--(u3a3a2);
\draw[->,dotted,magenta] (u3a4a1)--(u3a3a1);
\draw[->,dotted,magenta] (u3a4a2)--(u3a3a2);
\draw[->,red] (u3a3a1)--(u3a4a2);
\draw[->,dotted,magenta] (u3a5a1)--(u3a4a1);
\draw[->,dotted,magenta] (u3a5a2)--(u3a4a2);
\draw[->,red] (u3a4a1)--(u3a5a2);
\draw[->,dotted,magenta] (u3a6a1)--(u3a5a1);
\draw[->,dotted,magenta] (u3a6a2)--(u3a5a2);
\draw[->,red] (u3a5a1)--(u3a6a2);
\draw[->,dotted,magenta] (u3a7a1)--(u3a6a1);
\draw[->,dotted,magenta] (u3a7a2)--(u3a6a2);
\draw[->,red] (u3a6a1)--(u3a7a2);
\draw[->,red] (u3a6a2)--(u3a7a3);
\draw[->,dotted,magenta] (u3a8a1)--(u3a7a1);
\draw[->,dotted,magenta] (u3a8a2)--(u3a7a2);
\draw[->,dotted,magenta] (u3a8a3)--(u3a7a3);
\draw[->,red] (u3a7a1)--(u3a8a2);
\draw[->,red] (u3a7a2)--(u3a8a3);
\draw[->,dotted,magenta] (u3a9a1)--(u3a8a1);
\draw[->,dotted,magenta] (u3a9a2)--(u3a8a2);
\draw[->,dotted,magenta] (u3a9a3)--(u3a8a3);
\draw[->,red] (u3a8a1)--(u3a9a2);
\draw[->,red] (u3a8a2)--(u3a9a3);
\draw[->,dotted,magenta] (u3a10a1)--(u3a9a1);
\draw[->,dotted,magenta] (u3a10a2)--(u3a9a2);
\draw[->,dotted,magenta] (u3a10a3)--(u3a9a3);
\draw[->,red] (u3a9a1)--(u3a10a2);
\draw[->,red] (u3a9a2)--(u3a10a3);
\draw[->,red] (u3a9a3)--(u3a10a4);
\draw[->, blue] (u3a11a4)--(u3a10a4);
\draw[->, blue] (u3a11a3)--(u3a10a3);
\draw[->, blue] (u3a11a2)--(u3a10a2);
\draw[->, blue] (u3a11a1)--(u3a10a1);
\draw[->, blue] (u3a12a4)--(u3a11a4);
\draw[->, blue] (u3a12a3)--(u3a11a3);
\draw[->, blue] (u3a12a2)--(u3a11a2);
\draw[->, blue] (u3a12a1)--(u3a11a1);
\draw[->] (u3a10a1)--(u3a10a2);
\draw[->] (u3a10a2)--(u3a10a3);
\draw[->] (u3a10a3)--(u3a10a4);
\draw[->] (u3a11a1)--(u3a11a2);
\draw[->] (u3a11a2)--(u3a11a3);
\draw[->] (u3a11a3)--(u3a11a4);
\draw[->] (u3a12a1)--(u3a12a2);
\draw[->] (u3a12a2)--(u3a12a3);
\draw[->] (u3a12a3)--(u3a12a4);
\end{scope}

\draw[gray] (u1a10a4) ++(0.2, 0.2)--++(-3.4, 0)--++(-6, -6)--++(-3, 0)--++(-6, -6)--++(0, -0.4)--++(6.4, 0)--++(6, 6)--++(3, 0)--++(3, 3)--++(0, 3.4);

\draw[gray] (u3a10a4) ++(0, -3) ++(0.2, 0.2)--++(-3.4, 0)--++(-6, -6)--++(-3, 0)--++(-6, -6)--++(0, -0.4)--++(6.4, 0)--++(6, 6)--++(3, 0)--++(3, 3)
--++(0, 3.4);

\draw[gray] (u3a10a3) ++(-0.2, 0.2)--++(-9, -9)--++(-6, 0)--++(-9, -9)--++(0, -0.4)--++(3.4, 0)--++(9, 9)--++(3, 0)--++(6, 6)--++(3, 0)--++(0, 3.4)
--++(-0.4, 0);

\draw[gray] (u2a10a2) ++(0, -3) ++(-0.2, 0.2)--++(-9, -9)--++(-6, 0)--++(-9, -9)--++(0, -0.4)--++(3.4, 0)--++(9, 9)--++(3, 0)--++(6, 6)--++(3, 0)
--++(0, 3.4)--++(-0.4, 0);

\draw[gray, densely dotted, thick] (u3a10a3) ++(0.2, 0.2)--++(6, 0)--++(0, -3.4)--++(-6, 0);

\draw[gray, densely dotted, thick] (u2a10a2) ++(0, -3) ++(0.2, 0.2)--++(6, 0)--++(0, -3.4)--++(-6, 0);

\end{tikzpicture}

\caption{The butterfly diagram of a fixed point of a separated bow variety. The cyan dashed line indicates the boundary between the ``left side'' and ``right side''. Two different botched butterfly surgeries are outlined in solid gray. The surgery with site in the $U_1$ butterfly is subject to and satisfies constraint~\eqref{eqn:typeII_constraint}. The other surgery is not subject to this constraint. The additional edges and vertices added to $\site$ to form $\overline\site$ are outlined in dotted gray.}
\label{fig:butterfly_sides}\vspace{-2mm}
\end{figure}
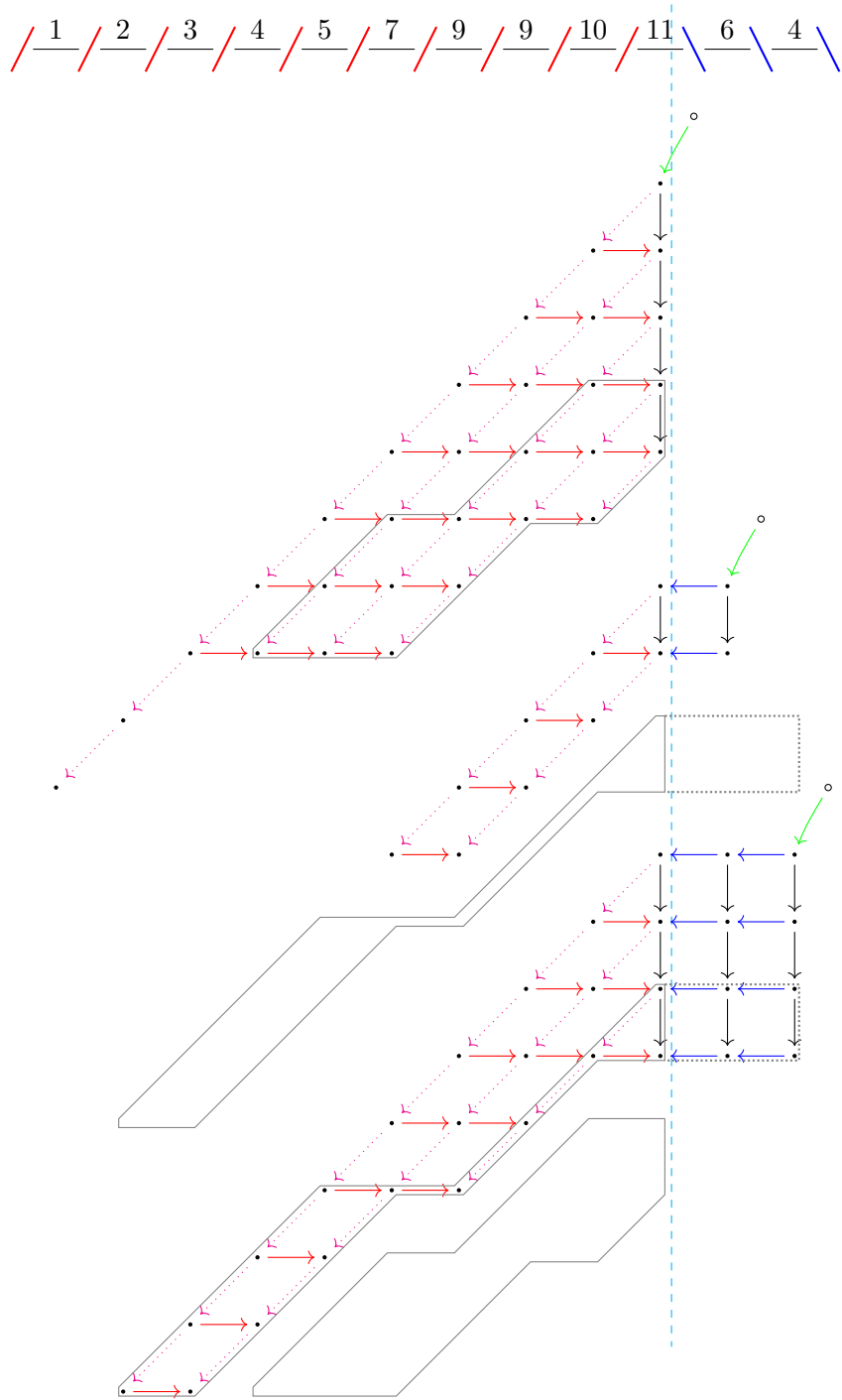

Fix $p\in \bowvar(\brane)^\torus$ and consider its butterfly diagram. Note that the portion of the butterfly diagram below $X_n$ and the right side will be the same for all fixed points. Consider a surgery operation $\surgery$ as in Section~\ref{sec:butterfly_surgery} on the left side of the butterfly diagram. That is, we take a~$B$, $C$, $D$-invariant subgraph $\site$ of the left side of the $U_j$ butterfly and vertically translate it to stack it below the left side of the $U_{j'}$ butterfly, for some $1\leq j\neq j' \leq m$. Assume that the left side of the resulting diagram is the left side of some butterfly diagram (not necessarily associated to the same bow variety). Note that all butterfly surgeries meet these conditions. The key difference is that butterfly surgeries cannot alter the $X_n$ column. If $\site$ includes at least one vertex in the $X_n$ column, then call $\surgery$ a ``botched butterfly surgery''. See Figure~\ref{fig:butterfly_sides} for an example. In Section~\ref{sec:butterfly_surgery}, the $B$, $C$, $D$-invariance of $\site$ emerged from other constraints on $\site$. Here, we must impose this property as an additional condition. We will produce a noncompact invariant curve from a~botched butterfly surgery following the approach of Section~\ref{sec:butterfly_surgery}.

The right side of the butterfly diagram is organized in a grid formed by horizontal leftward edges, which represent $A$ maps, and vertical downward edges, which represent $-B$ maps (see Figure~\ref{fig:butterfly_sides}). Let $\overline\site$ be the minimal $A$, $A^{-1}$, $B$, $C$, $D$-invariant subgraph containing $\site$. Let $\overline\site'$ be a copy of $\overline\site$ translated according to $\surgery$. Take the edges extending from the $U_{j'}$ butterfly to $\overline\site'$ imputed by the rules of Section~\ref{sec:fixed_pts}, and create the corresponding edges between the $U_{j'}$ butterfly and $\overline\site$. As in Section~\ref{sec:butterfly_surgery}, the resulting graph determines an element $\tilde{p}_\gamma \in \MM$.

If $j>j'$, it is easy to verify the 0-momentum condition. If $j<j'$, all momentum conditions will be satisfied except for $B_{U_{j}}A_{U_{j}} - A_{U_{j}}B'_{U_{j}} + a_{U_{j}}b_{U_{j}} = 0$. Due to the structure of separated brane diagrams, the number of vertices in the $X_n, X_{n+1}, \dots , X_{n+k-1}$ column of the $U_k$ butterfly is given by $c_k$. Recall that $c$ is the margin vector of D5 brane charges (see Section~\ref{sec:fixed_pts}). Suppose~$\site$ contains the bottom $r$ vertices of the $X_n$ column, where
\begin{equation} \label{eqn:typeII_constraint}
r \geq c_j - c_{j'} + 1.
\end{equation}
Label the vertices of the $X_{n+j-1}$ column of the $U_j$ butterfly from top to bottom as $v_1, \dots , v_{c_j}$. Label the vertices of the $X_{n+j}$ column of the $U{j'}$ butterfly from top to bottom as $v'_1, \dots , v'_{c_{j'}}$. Add additional edges according to the following rule:
\begin{itemize}\itemsep=0pt
\item Create negative blue edges from $v'_{k + c_{j'} - c_j + r}$ to $v_k$ for $1 \leq k \leq c_j - r$. Recall that the blue edges represent $A$.
\item Create a negative green edge from $v'_{c_{j'} - c_j + r}$ to the $U_j$ framing vertex. Recall that the green edge directed into the framing vertex represents $-b$ (Section~\ref{sec:fixed_pts}).
\end{itemize}
This corrects the momentum, so that the 0-momentum condition is satisfied. This rule is depicted in Figure~\ref{fig:special_botched}. The point $\tilde{p}_\gamma$ is illustrated in Figure~\ref{fig:typeII_curve} for both of the surgeries in Figure~\ref{fig:butterfly_sides}.

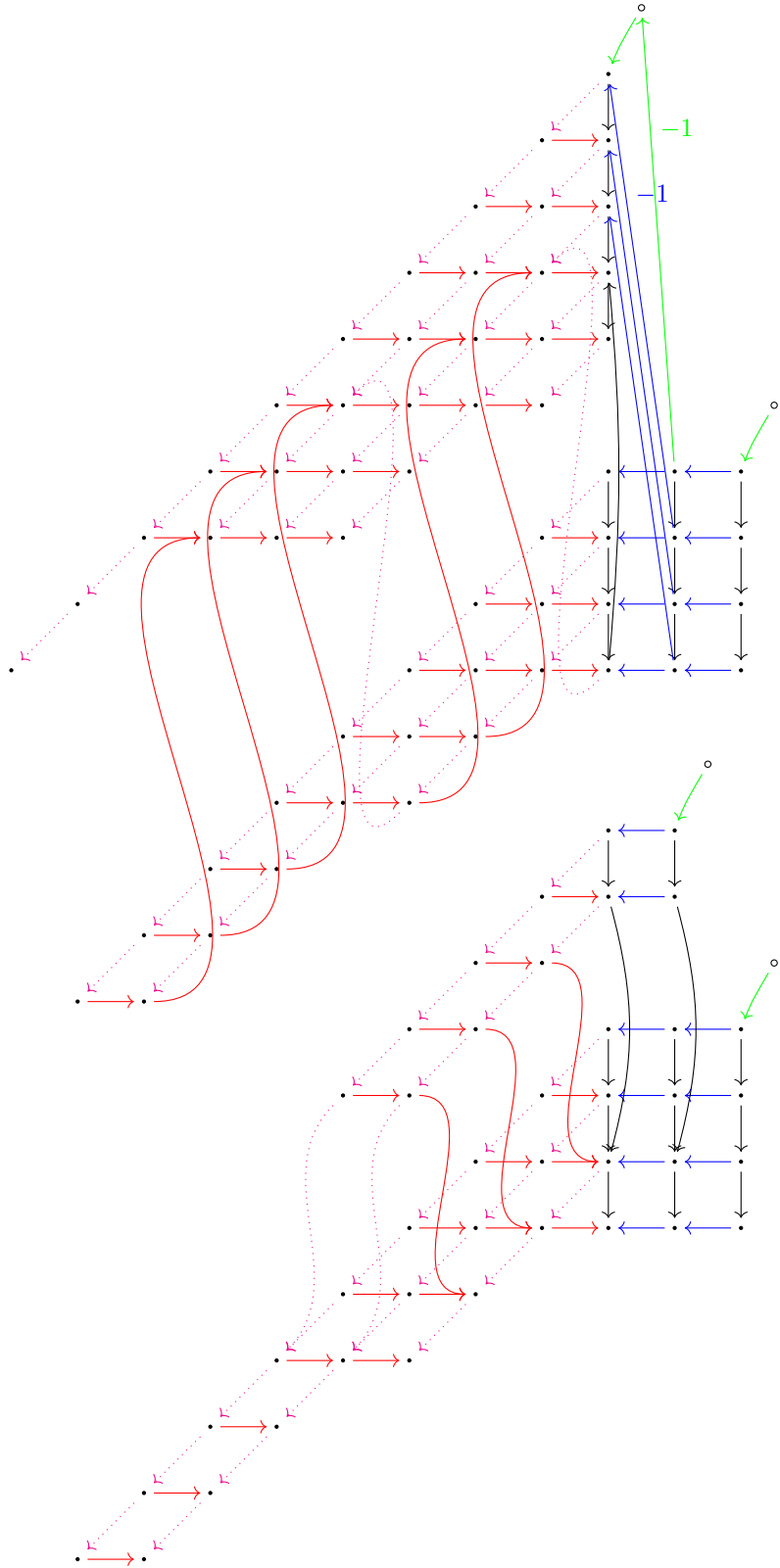
\begin{figure}[th!]
\centering
\begin{tikzpicture}[scale=0.28]
\node () at (0,-6) {};

\draw (30,-3) circle[radius=0.14] node(u1){};
\draw[fill] (1.5,-33) circle[radius=0.07] node(u1a1a1){};
\draw[fill] (4.5,-30) circle[radius=0.07] node(u1a2a1){};
\draw[fill] (7.5,-27) circle[radius=0.07] node(u1a3a1){};
\draw[fill] (10.5,-24) circle[radius=0.07] node(u1a4a1){};
\draw[fill] (10.5,-27) circle[radius=0.07] node(u1a4a2){};
\draw[fill] (13.5,-21) circle[radius=0.07] node(u1a5a1){};
\draw[fill] (13.5,-24) circle[radius=0.07] node(u1a5a2){};
\draw[fill] (13.5,-27) circle[radius=0.07] node(u1a5a3){};
\draw[fill] (16.5,-18) circle[radius=0.07] node(u1a6a1){};
\draw[fill] (16.5,-21) circle[radius=0.07] node(u1a6a2){};
\draw[fill] (16.5,-24) circle[radius=0.07] node(u1a6a3){};
\draw[fill] (16.5,-27) circle[radius=0.07] node(u1a6a4){};
\draw[fill] (19.5,-15) circle[radius=0.07] node(u1a7a1){};
\draw[fill] (19.5,-18) circle[radius=0.07] node(u1a7a2){};
\draw[fill] (19.5,-21) circle[radius=0.07] node(u1a7a3){};
\draw[fill] (19.5,-24) circle[radius=0.07] node(u1a7a4){};
\draw[fill] (22.5,-12) circle[radius=0.07] node(u1a8a1){};
\draw[fill] (22.5,-15) circle[radius=0.07] node(u1a8a2){};
\draw[fill] (22.5,-18) circle[radius=0.07] node(u1a8a3){};
\draw[fill] (22.5,-21) circle[radius=0.07] node(u1a8a4){};
\draw[fill] (25.5,-9) circle[radius=0.07] node(u1a9a1){};
\draw[fill] (25.5,-12) circle[radius=0.07] node(u1a9a2){};
\draw[fill] (25.5,-15) circle[radius=0.07] node(u1a9a3){};
\draw[fill] (25.5,-18) circle[radius=0.07] node(u1a9a4){};
\draw[fill] (25.5,-21) circle[radius=0.07] node(u1a9a5){};
\draw[fill] (28.5,-6) circle[radius=0.07] node(u1a10a1){};
\draw[fill] (28.5,-9) circle[radius=0.07] node(u1a10a2){};
\draw[fill] (28.5,-12) circle[radius=0.07] node(u1a10a3){};
\draw[fill] (28.5,-15) circle[radius=0.07] node(u1a10a4){};
\draw[fill] (28.5,-18) circle[radius=0.07] node(u1a10a5){};

\draw[->, green] (u1) to[out=-120,in=70] (u1a10a1);
\draw[->,dotted,magenta] (u1a2a1)--(u1a1a1);
\draw[->,dotted,magenta] (u1a3a1)--(u1a2a1);
\draw[->,dotted,magenta] (u1a4a1)--(u1a3a1);
\draw[->,red] (u1a3a1)--(u1a4a2);
\draw[->,dotted,magenta] (u1a5a1)--(u1a4a1);
\draw[->,dotted,magenta] (u1a5a2)--(u1a4a2);
\draw[->,red] (u1a4a1)--(u1a5a2);
\draw[->,red] (u1a4a2)--(u1a5a3);
\draw[->,dotted,magenta] (u1a6a1)--(u1a5a1);
\draw[->,dotted,magenta] (u1a6a2)--(u1a5a2);
\draw[->,dotted,magenta] (u1a6a3)--(u1a5a3);
\draw[->,red] (u1a5a1)--(u1a6a2);
\draw[->,red] (u1a5a2)--(u1a6a3);
\draw[->,red] (u1a5a3)--(u1a6a4);
\draw[->,dotted,magenta] (u1a7a1)--(u1a6a1);
\draw[->,dotted,magenta] (u1a7a2)--(u1a6a2);
\draw[->,dotted,magenta] (u1a7a3)--(u1a6a3);
\draw[->,dotted,magenta] (u1a7a4)--(u1a6a4);
\draw[->,red] (u1a6a1)--(u1a7a2);
\draw[->,red] (u1a6a2)--(u1a7a3);
\draw[->,red] (u1a6a3)--(u1a7a4);
\draw[->,dotted,magenta] (u1a8a1)--(u1a7a1);
\draw[->,dotted,magenta] (u1a8a2)--(u1a7a2);
\draw[->,dotted,magenta] (u1a8a3)--(u1a7a3);
\draw[->,dotted,magenta] (u1a8a4)--(u1a7a4);
\draw[->,red] (u1a7a1)--(u1a8a2);
\draw[->,red] (u1a7a2)--(u1a8a3);
\draw[->,red] (u1a7a3)--(u1a8a4);
\draw[->,dotted,magenta] (u1a9a1)--(u1a8a1);
\draw[->,dotted,magenta] (u1a9a2)--(u1a8a2);
\draw[->,dotted,magenta] (u1a9a3)--(u1a8a3);
\draw[->,dotted,magenta] (u1a9a4)--(u1a8a4);
\draw[->,red] (u1a8a1)--(u1a9a2);
\draw[->,red] (u1a8a2)--(u1a9a3);
\draw[->,red] (u1a8a3)--(u1a9a4);
\draw[->,red] (u1a8a4)--(u1a9a5);
\draw[->,dotted,magenta] (u1a10a1)--(u1a9a1);
\draw[->,dotted,magenta] (u1a10a2)--(u1a9a2);
\draw[->,dotted,magenta] (u1a10a3)--(u1a9a3);
\draw[->,dotted,magenta] (u1a10a4)--(u1a9a4);
\draw[->,dotted,magenta] (u1a10a5)--(u1a9a5);
\draw[->,red] (u1a9a1)--(u1a10a2);
\draw[->,red] (u1a9a2)--(u1a10a3);
\draw[->,red] (u1a9a3)--(u1a10a4);
\draw[->,red] (u1a9a4)--(u1a10a5);
\draw[->] (u1a10a1)--(u1a10a2);
\draw[->] (u1a10a2)--(u1a10a3);
\draw[->] (u1a10a3)--(u1a10a4);
\draw[->] (u1a10a4)--(u1a10a5);

\begin{scope}[yshift=21cm]
\node () at (0,-45) {};

\draw (36,-42) circle[radius=0.14] node(u3){};
\draw[fill] (4.5,-69) circle[radius=0.07] node(u3a2a1){};
\draw[fill] (7.5,-66) circle[radius=0.07] node(u3a3a1){};
\draw[fill] (7.5,-69) circle[radius=0.07] node(u3a3a2){};
\draw[fill] (10.5,-63) circle[radius=0.07] node(u3a4a1){};
\draw[fill] (10.5,-66) circle[radius=0.07] node(u3a4a2){};
\draw[fill] (13.5,-60) circle[radius=0.07] node(u3a5a1){};
\draw[fill] (13.5,-63) circle[radius=0.07] node(u3a5a2){};
\draw[fill] (16.5,-57) circle[radius=0.07] node(u3a6a1){};
\draw[fill] (16.5,-60) circle[radius=0.07] node(u3a6a2){};
\draw[fill] (19.5,-54) circle[radius=0.07] node(u3a7a1){};
\draw[fill] (19.5,-57) circle[radius=0.07] node(u3a7a2){};
\draw[fill] (19.5,-60) circle[radius=0.07] node(u3a7a3){};
\draw[fill] (22.5,-51) circle[radius=0.07] node(u3a8a1){};
\draw[fill] (22.5,-54) circle[radius=0.07] node(u3a8a2){};
\draw[fill] (22.5,-57) circle[radius=0.07] node(u3a8a3){};
\draw[fill] (25.5,-48) circle[radius=0.07] node(u3a9a1){};
\draw[fill] (25.5,-51) circle[radius=0.07] node(u3a9a2){};
\draw[fill] (25.5,-54) circle[radius=0.07] node(u3a9a3){};
\draw[fill] (28.5,-45) circle[radius=0.07] node(u3a10a1){};
\draw[fill] (28.5,-48) circle[radius=0.07] node(u3a10a2){};
\draw[fill] (28.5,-51) circle[radius=0.07] node(u3a10a3){};
\draw[fill] (28.5,-54) circle[radius=0.07] node(u3a10a4){};
\draw[fill] (31.5,-45) circle[radius=0.07] node(u3a11a1){};
\draw[fill] (31.5,-48) circle[radius=0.07] node(u3a11a2){};
\draw[fill] (31.5,-51) circle[radius=0.07] node(u3a11a3){};
\draw[fill] (31.5,-54) circle[radius=0.07] node(u3a11a4){};
\draw[fill] (34.5,-45) circle[radius=0.07] node(u3a12a1){};
\draw[fill] (34.5,-48) circle[radius=0.07] node(u3a12a2){};
\draw[fill] (34.5,-51) circle[radius=0.07] node(u3a12a3){};
\draw[fill] (34.5,-54) circle[radius=0.07] node(u3a12a4){};

\draw[->, green] (u3) to[out=-120,in=70] (u3a12a1);
\draw[->,dotted,magenta] (u3a3a1)--(u3a2a1);
\draw[->,red] (u3a2a1)--(u3a3a2);
\draw[->,dotted,magenta] (u3a4a1)--(u3a3a1);
\draw[->,dotted,magenta] (u3a4a2)--(u3a3a2);
\draw[->,red] (u3a3a1)--(u3a4a2);
\draw[->,dotted,magenta] (u3a5a1)--(u3a4a1);
\draw[->,dotted,magenta] (u3a5a2)--(u3a4a2);
\draw[->,red] (u3a4a1)--(u3a5a2);
\draw[->,dotted,magenta] (u3a6a1)--(u3a5a1);
\draw[->,dotted,magenta] (u3a6a2)--(u3a5a2);
\draw[->,red] (u3a5a1)--(u3a6a2);
\draw[->,dotted,magenta] (u3a7a1)--(u3a6a1);
\draw[->,dotted,magenta] (u3a7a2)--(u3a6a2);
\draw[->,red] (u3a6a1)--(u3a7a2);
\draw[->,red] (u3a6a2)--(u3a7a3);
\draw[->,dotted,magenta] (u3a8a1)--(u3a7a1);
\draw[->,dotted,magenta] (u3a8a2)--(u3a7a2);
\draw[->,dotted,magenta] (u3a8a3)--(u3a7a3);
\draw[->,red] (u3a7a1)--(u3a8a2);
\draw[->,red] (u3a7a2)--(u3a8a3);
\draw[->,dotted,magenta] (u3a9a1)--(u3a8a1);
\draw[->,dotted,magenta] (u3a9a2)--(u3a8a2);
\draw[->,dotted,magenta] (u3a9a3)--(u3a8a3);
\draw[->,red] (u3a8a1)--(u3a9a2);
\draw[->,red] (u3a8a2)--(u3a9a3);
\draw[->,dotted,magenta] (u3a10a1)--(u3a9a1);
\draw[->,dotted,magenta] (u3a10a2)--(u3a9a2);
\draw[->,dotted,magenta] (u3a10a3)--(u3a9a3);
\draw[->,red] (u3a9a1)--(u3a10a2);
\draw[->,red] (u3a9a2)--(u3a10a3);
\draw[->,red] (u3a9a3)--(u3a10a4);
\draw[->, blue] (u3a11a4)--(u3a10a4);
\draw[->, blue] (u3a11a3)--(u3a10a3);
\draw[->, blue] (u3a11a2)--(u3a10a2);
\draw[->, blue] (u3a11a1)--(u3a10a1);
\draw[->, blue] (u3a12a4)--(u3a11a4);
\draw[->, blue] (u3a12a3)--(u3a11a3);
\draw[->, blue] (u3a12a2)--(u3a11a2);
\draw[->, blue] (u3a12a1)--(u3a11a1);
\draw[->] (u3a10a1)--(u3a10a2);
\draw[->] (u3a10a2)--(u3a10a3);
\draw[->] (u3a10a3)--(u3a10a4);
\draw[->] (u3a11a1)--(u3a11a2);
\draw[->] (u3a11a2)--(u3a11a3);
\draw[->] (u3a11a3)--(u3a11a4);
\draw[->] (u3a12a1)--(u3a12a2);
\draw[->] (u3a12a2)--(u3a12a3);
\draw[->] (u3a12a3)--(u3a12a4);
\end{scope}

\draw[->] (u3a10a4)to[out=85, in=-85](u1a10a4);
\draw[magenta, dotted, ->] (u3a10a4)to[out=225, in=45](u1a9a3); \draw[magenta, dotted, ->] (u3a7a3)to[out=225, in=45](u1a6a2);
\draw[red, ->] (u3a8a3)to[out=0, in=180](u1a9a3); \draw[red, ->] (u3a7a3)to[out=0, in=180](u1a8a3);
\draw[red, ->] (u3a5a2)to[out=0, in=180](u1a6a2); \draw[red, ->] (u3a4a2)to[out=0, in=180](u1a5a2); \draw[red, ->] (u3a3a2)to[out=0, in=180](u1a4a2);
\draw[blue, ->] (u3a11a4)--(u1a10a3); \draw[blue, ->] (u3a11a3)--(u1a10a2); \draw[blue, ->] (u3a11a2)--(u1a10a1) node[pos=0.75, right]{\small $-1$};
\draw[green, ->] (u3a11a1)--(u1) node[pos=0.75, right]{\small $-1$};
\end{tikzpicture}

\vspace{-100pt}
\begin{tikzpicture}[scale=0.28]
\node () at (0,-33) {};

\draw (33,-30) circle[radius=0.14] node(u2){};
\draw[fill] (16.5,-45) circle[radius=0.07] node(u2a6a1){};
\draw[fill] (19.5,-42) circle[radius=0.07] node(u2a7a1){};
\draw[fill] (19.5,-45) circle[radius=0.07] node(u2a7a2){};
\draw[fill] (22.5,-39) circle[radius=0.07] node(u2a8a1){};
\draw[fill] (22.5,-42) circle[radius=0.07] node(u2a8a2){};
\draw[fill] (25.5,-36) circle[radius=0.07] node(u2a9a1){};
\draw[fill] (25.5,-39) circle[radius=0.07] node(u2a9a2){};
\draw[fill] (28.5,-33) circle[radius=0.07] node(u2a10a1){};
\draw[fill] (28.5,-36) circle[radius=0.07] node(u2a10a2){};
\draw[fill] (31.5,-33) circle[radius=0.07] node(u2a11a1){};
\draw[fill] (31.5,-36) circle[radius=0.07] node(u2a11a2){};

\draw[->, green] (u2) to[out=-120,in=70] (u2a11a1);
\draw[->,dotted,magenta] (u2a7a1)--(u2a6a1);
\draw[->,red] (u2a6a1)--(u2a7a2);
\draw[->,dotted,magenta] (u2a8a1)--(u2a7a1);
\draw[->,dotted,magenta] (u2a8a2)--(u2a7a2);
\draw[->,red] (u2a7a1)--(u2a8a2);
\draw[->,dotted,magenta] (u2a9a1)--(u2a8a1);
\draw[->,dotted,magenta] (u2a9a2)--(u2a8a2);
\draw[->,red] (u2a8a1)--(u2a9a2);
\draw[->,dotted,magenta] (u2a10a1)--(u2a9a1);
\draw[->,dotted,magenta] (u2a10a2)--(u2a9a2);
\draw[->,red] (u2a9a1)--(u2a10a2);
\draw[->, blue] (u2a11a2)--(u2a10a2);
\draw[->, blue] (u2a11a1)--(u2a10a1);
\draw[->] (u2a10a1)--(u2a10a2);
\draw[->] (u2a11a1)--(u2a11a2);

\begin{scope}[yshift=3cm]
\node () at (0,-45) {};

\draw (36,-42) circle[radius=0.14] node(u3){};
\draw[fill] (4.5,-69) circle[radius=0.07] node(u3a2a1){};
\draw[fill] (7.5,-66) circle[radius=0.07] node(u3a3a1){};
\draw[fill] (7.5,-69) circle[radius=0.07] node(u3a3a2){};
\draw[fill] (10.5,-63) circle[radius=0.07] node(u3a4a1){};
\draw[fill] (10.5,-66) circle[radius=0.07] node(u3a4a2){};
\draw[fill] (13.5,-60) circle[radius=0.07] node(u3a5a1){};
\draw[fill] (13.5,-63) circle[radius=0.07] node(u3a5a2){};
\draw[fill] (16.5,-57) circle[radius=0.07] node(u3a6a1){};
\draw[fill] (16.5,-60) circle[radius=0.07] node(u3a6a2){};
\draw[fill] (19.5,-54) circle[radius=0.07] node(u3a7a1){};
\draw[fill] (19.5,-57) circle[radius=0.07] node(u3a7a2){};
\draw[fill] (19.5,-60) circle[radius=0.07] node(u3a7a3){};
\draw[fill] (22.5,-51) circle[radius=0.07] node(u3a8a1){};
\draw[fill] (22.5,-54) circle[radius=0.07] node(u3a8a2){};
\draw[fill] (22.5,-57) circle[radius=0.07] node(u3a8a3){};
\draw[fill] (25.5,-48) circle[radius=0.07] node(u3a9a1){};
\draw[fill] (25.5,-51) circle[radius=0.07] node(u3a9a2){};
\draw[fill] (25.5,-54) circle[radius=0.07] node(u3a9a3){};
\draw[fill] (28.5,-45) circle[radius=0.07] node(u3a10a1){};
\draw[fill] (28.5,-48) circle[radius=0.07] node(u3a10a2){};
\draw[fill] (28.5,-51) circle[radius=0.07] node(u3a10a3){};
\draw[fill] (28.5,-54) circle[radius=0.07] node(u3a10a4){};
\draw[fill] (31.5,-45) circle[radius=0.07] node(u3a11a1){};
\draw[fill] (31.5,-48) circle[radius=0.07] node(u3a11a2){};
\draw[fill] (31.5,-51) circle[radius=0.07] node(u3a11a3){};
\draw[fill] (31.5,-54) circle[radius=0.07] node(u3a11a4){};
\draw[fill] (34.5,-45) circle[radius=0.07] node(u3a12a1){};
\draw[fill] (34.5,-48) circle[radius=0.07] node(u3a12a2){};
\draw[fill] (34.5,-51) circle[radius=0.07] node(u3a12a3){};
\draw[fill] (34.5,-54) circle[radius=0.07] node(u3a12a4){};

\draw[->, green] (u3) to[out=-120,in=70] (u3a12a1);
\draw[->,dotted,magenta] (u3a3a1)--(u3a2a1);
\draw[->,red] (u3a2a1)--(u3a3a2);
\draw[->,dotted,magenta] (u3a4a1)--(u3a3a1);
\draw[->,dotted,magenta] (u3a4a2)--(u3a3a2);
\draw[->,red] (u3a3a1)--(u3a4a2);
\draw[->,dotted,magenta] (u3a5a1)--(u3a4a1);
\draw[->,dotted,magenta] (u3a5a2)--(u3a4a2);
\draw[->,red] (u3a4a1)--(u3a5a2);
\draw[->,dotted,magenta] (u3a6a1)--(u3a5a1);
\draw[->,dotted,magenta] (u3a6a2)--(u3a5a2);
\draw[->,red] (u3a5a1)--(u3a6a2);
\draw[->,dotted,magenta] (u3a7a1)--(u3a6a1);
\draw[->,dotted,magenta] (u3a7a2)--(u3a6a2);
\draw[->,red] (u3a6a1)--(u3a7a2);
\draw[->,red] (u3a6a2)--(u3a7a3);
\draw[->,dotted,magenta] (u3a8a1)--(u3a7a1);
\draw[->,dotted,magenta] (u3a8a2)--(u3a7a2);
\draw[->,dotted,magenta] (u3a8a3)--(u3a7a3);
\draw[->,red] (u3a7a1)--(u3a8a2);
\draw[->,red] (u3a7a2)--(u3a8a3);
\draw[->,dotted,magenta] (u3a9a1)--(u3a8a1);
\draw[->,dotted,magenta] (u3a9a2)--(u3a8a2);
\draw[->,dotted,magenta] (u3a9a3)--(u3a8a3);
\draw[->,red] (u3a8a1)--(u3a9a2);
\draw[->,red] (u3a8a2)--(u3a9a3);
\draw[->,dotted,magenta] (u3a10a1)--(u3a9a1);
\draw[->,dotted,magenta] (u3a10a2)--(u3a9a2);
\draw[->,dotted,magenta] (u3a10a3)--(u3a9a3);
\draw[->,red] (u3a9a1)--(u3a10a2);
\draw[->,red] (u3a9a2)--(u3a10a3);
\draw[->,red] (u3a9a3)--(u3a10a4);
\draw[->, blue] (u3a11a4)--(u3a10a4);
\draw[->, blue] (u3a11a3)--(u3a10a3);
\draw[->, blue] (u3a11a2)--(u3a10a2);
\draw[->, blue] (u3a11a1)--(u3a10a1);
\draw[->, blue] (u3a12a4)--(u3a11a4);
\draw[->, blue] (u3a12a3)--(u3a11a3);
\draw[->, blue] (u3a12a2)--(u3a11a2);
\draw[->, blue] (u3a12a1)--(u3a11a1);
\draw[->] (u3a10a1)--(u3a10a2);
\draw[->] (u3a10a2)--(u3a10a3);
\draw[->] (u3a10a3)--(u3a10a4);
\draw[->] (u3a11a1)--(u3a11a2);
\draw[->] (u3a11a2)--(u3a11a3);
\draw[->] (u3a11a3)--(u3a11a4);
\draw[->] (u3a12a1)--(u3a12a2);
\draw[->] (u3a12a2)--(u3a12a3);
\draw[->] (u3a12a3)--(u3a12a4);
\end{scope}

\draw[magenta, dotted, ->] (u2a6a1)to[out=225, in=45](u3a5a1); \draw[magenta, dotted, ->] (u2a7a2)to[out=225, in=45](u3a6a2);
\draw[red, ->] (u2a7a2)to[out=0, in=180](u3a8a3); \draw[red, ->] (u2a8a2)to[out=0, in=180](u3a9a3); \draw[red, ->] (u2a9a2)to[out=0, in=180](u3a10a3);
\draw[->] (u2a10a2)to[out=-75, in=75](u3a10a3); \draw[->] (u2a11a2)to[out=-75, in=75](u3a11a3);
\end{tikzpicture}

\caption{Construction of the invariant curves corresponding to the two botched butterfly surgeries in Figure~\ref{fig:butterfly_sides}.}
\label{fig:typeII_curve}
\vspace{-2mm}

\end{figure}

Using similar arguments to those in Section~\ref{sec:butterfly_surgery} (see Lemma~\ref{lemma:type23_curve}), one can show that \smash{$\tilde{p}_\gamma \in \prequot$}, the orbit of its image $p_\gamma \in \bowvar(\brane)$ is 1-dimensional, and its closure $\gamma$ is an invariant curve containing~$p$. Additionally, if $\site$ has $k$ connected components, then a $k$-dimensional pencil $\pen_\surgery$ of invariant curves containing $p$ can be constructed. These results are summarized by the following analog to Lemma~\ref{lemma:butterfly_surgery}. The question of compactness is postponed until Section~\ref{sec:classification}.

\begin{Lemma} \label{lemma:botched_surgery}
Fix a separated bow variety $\bowvar(\brane)$ and $p\in \bowvar(\brane)^\torus$. Let $\surgery$ be a botched butterfly surgery moving the site $\site$ from the $U_j$ butterfly to the $U_{j'}$ butterfly. If $j < j'$, impose the additional condition that $\site$ contains at least $c_j - c_{j'} + 1$ vertices of the $X_n$ column, where $c$ is the margin vector of D$5$ brane charges. Then, there exists a $k$-dimensional pencil $\pen_\surgery$ of $\torus$-invariant curves containing $p$. Each curve has tangent weight
\smash{$
\frac{\ub}{\ub'}\hb^{-\Delta_\surgery y}$},
at $p$.
\end{Lemma}

\subsection{Nonsurgery curves} \label{sec:nonsurgery}
We retain the conventions of Section~\ref{sec:botched_surgery} for this section. Fix $1\leq j < j' \leq m$. We will construct $\max\{0, c_{j'} - c_j\}$ many invariant curves. To construct the $i$-th curve,
\begin{itemize}\itemsep=0pt
\item create a green edge from $v'_i$ to the $U_j$ framing vertex,
\item create blue edges from $v'_{i+k}$ to $v_k$ for $1 \leq k \leq c_j$.
\end{itemize}
Figure~\ref{fig:nonsurgery_curve} depicts these added edges. Using similar arguments to those in Section~\ref{sec:butterfly_surgery} (see Lemma~\ref{lemma:type23_curve}), one can show that this results in a point of $\bowvar(\brane)$ with a 1-dimensional orbit having~$p$ in its closure. The tangent weight of this curve at $p$ can easily be determined by examining the added $b$ edge and comparing the heights of the target and source.

\begin{figure}[t]
\centering
\subcaptionbox{botched surgery \label{fig:special_botched}}[0.4\textwidth]
{
\begin{tikzpicture}
\draw (0, 0) node(l){} ++(2, 0) node(u){} ++(2, 0) node(r){};
\draw (l)--(r);
\draw[blue] (u) +(0.5, -0.5) -- +(-0.5, 0.5) +(0, 0.5) node{$U_j$};

\draw (2, -1) circle[radius = .05] node(uj){};

\draw (1, -1.5) node(a1){} ++(0, -.5) node(a2){} ++(0, -.5) node(a3){} ++(0, -.5) node(a4){} ++(0, -.5) node(a5){} ++(0, -.5) node(a6){}
 ++(0, -.5) node(a7){} ++(0, -.5) node(a8){};

\draw (a8) ++(0, -1) node(b1){} ++(0, -.5) node(b2){} ++(0, -.5) node(b3){} ++(0, -.5) node(b4){} ++(0, -.5) node(b5){} ++(0, -.5) node(b6){}
 ++(0, -.5) node(b7){} ++(0, -.5) node(b8){};

\draw (b1) ++(2, 0) node(c1){} ++(0, -.5) node(c2){} ++(0, -.5) node(c3){} ++(0, -.5) node(c4){} ++(0, -.5) node(c5){} ++(0, -.5) node(c6){}
 ++(0, -.5) node(c7){} ++(0, -.5) node(c8){};

\draw[fill] (a1) circle[radius = 0.025]; \draw[fill] (a4) circle [radius = 0.025]; \draw[fill] (a5) circle [radius = 0.025]; \draw[fill] (a8) circle [radius = 0.025];
\draw[fill] (b1) circle[radius = 0.025]; \draw[fill] (b4) circle[radius = 0.025]; \draw[fill] (b5) circle[radius = 0.025]; \draw[fill] (b8) circle[radius = 0.025];
\draw[fill] (c1) circle[radius = 0.025]; \draw[fill] (c4) circle[radius = 0.025]; \draw[fill] (c5) circle[radius = 0.025]; \draw[fill] (c8) circle[radius = 0.025];

\draw[->, green] (uj) to[out = 250, in = 45] (a1);
\draw[->] (a1) -- (a2); \draw[->] (a3) -- (a4); \draw[->] (a4) -- (a5); \draw[->] (a5) -- (a6); \draw[->] (a7) -- (a8);
\draw[->] (b1) -- (b2); \draw[->] (b3) -- (b4); \draw[->] (b4) -- (b5); \draw[->] (b5) -- (b6); \draw[->] (b7) -- (b8);
\draw[->] (c1) -- (c2); \draw[->] (c3) -- (c4); \draw[->] (c4) -- (c5); \draw[->] (c5) -- (c6); \draw[->] (c7) -- (c8);

\draw[->, blue] (c1) -- (b1); \draw[->, blue] (c4) -- (b4); \draw[->, blue] (c5) -- (b5); \draw[->, blue] (c8) -- (b8);

\draw[dotted] (a2) -- (a3); \draw[dotted] (a6) -- (a7);
\draw[dotted] (b2) -- (b3); \draw[dotted] (b6) -- (b7);
\draw[dotted] (c2) -- (c3); \draw[dotted] (c6) -- (c7);

\draw[->, green] (c4) to[out = 75, in = 280] (uj); \draw (2.6, -2.2) node[green]{\small -1};
\draw[->, blue] (c5) -- (a1) node[pos = 0.75, right]{\small -1}; \draw[->, blue] (c8) -- (a4) node[pos = 0.75, right]{\small -1};
\draw[->] (b8) to[out = 105, in = 255] (a5);

\draw[gray] ($(a5)+(-0.1, 0.1)$) -- ($(a8)+(-0.1, -0.1)$) -- ($(a8)+(0.1, -0.1)$) -- ($(a5)+(0.1, 0.1)$) -- ($(a5)+(-0.1, 0.1)$);
\end{tikzpicture}
}
\subcaptionbox{nonsurgery \label{fig:nonsurgery_curve}}[0.4\textwidth]
{
\begin{tikzpicture}
\draw (0, 0) node(l){} ++(2, 0) node(u){} ++(2, 0) node(r){};
\draw (l)--(r);
\draw[blue] (u) +(0.5, -0.5) -- +(-0.5, 0.5) +(0, 0.5) node{$U_j$};

\draw (2, -1) circle[radius = .05] node(uj){};

\draw (1, -1.5) node(a1){} ++(0, -.5) node(a2){} ++(0, -.5) node(a3){} ++(0, -.5) node(a4){} ++(0, -.5);

\draw (a4) ++(0, -1.5) node(b1){} ++(0, -.5) node(b2){} ++(0, -.5) node(b3){} ++(0, -.5) node(b4){} ++(0, -.5) node(b5){} ++(0, -.5) node(b6){}
 ++(0, -.5) node(b7){} ++(0, -.5) node(b8){} ++(0, -.5) node(b9){} ++(0, -.5) node(b10){} ++(0, -.5) node(b11){};

\draw (b1) ++(2, 0) node(c1){} ++(0, -.5) node(c2){} ++(0, -.5) node(c3){} ++(0, -.5) node(c4){} ++(0, -.5) node(c5){} ++(0, -.5) node(c6){}
 ++(0, -.5) node(c7){} ++(0, -.5) node(c8){} ++(0, -.5) node(c9){} ++(0, -.5) node(c10){} ++(0, -.5) node(c11){};

\draw[fill] (a1) circle[radius = 0.025]; \draw[fill] (a4) circle [radius = 0.025];

\draw[fill] (b1) circle[radius = 0.025]; \draw[fill] (b4) circle[radius = 0.025]; \draw[fill] (b5) circle[radius = 0.025]; \draw[fill] (b8) circle[radius = 0.025];
\draw[fill] (b11) circle [radius = 0.025];

\draw[fill] (c1) circle[radius = 0.025]; \draw[fill] (c4) circle[radius = 0.025]; \draw[fill] (c5) circle[radius = 0.025]; \draw[fill] (c8) circle[radius = 0.025];
\draw[fill] (c11) circle [radius = 0.025];

\draw[->, green] (uj) to[out = 250, in = 45] (a1);

\draw[->] (a1) -- (a2); \draw[dotted] (a2) -- (a3); \draw[->] (a3) -- (a4);

\draw[->] (b1) -- (b2); \draw[dotted] (b2) -- (b3); \draw[->] (b3) -- (b4); \draw[->] (b4) -- (b5); \draw[->] (b5) -- (b6); \draw[dotted] (b6) -- (b7);
\draw[->] (b7) -- (b8); \draw[->] (b8) -- (b9); \draw[dotted] (b9) -- (b10); \draw[->] (b10) -- (b11);

\draw[->] (c1) -- (c2); \draw[dotted] (c2) -- (c3); \draw[->] (c3) -- (c4); \draw[->] (c4) -- (c5); \draw[->] (c4) -- (c5); \draw[->] (c5) -- (c6);
\draw[dotted] (c6) -- (c7); \draw[->] (c7) -- (c8); \draw[->] (c8) -- (c9); \draw[dotted] (c9) -- (c10); \draw[->] (c10) -- (c11);

\draw[->, blue] (c1) -- (b1); \draw[->, blue] (c4) -- (b4); \draw[->, blue] (c5) -- (b5); \draw[->, blue] (c8) -- (b8); \draw[->, blue] (c11) -- (b11);

\draw[->, green] (c4) to[out = 75, in = 280] (uj);
\draw[->, blue] (c5) -- (a1); \draw[->, blue] (c8) -- (a4);
\end{tikzpicture}
}

\caption{Local depictions of (a) the curve associated to a botched surgery moving a part of a higher butterfly to a lower butterfly and (b) a nonsurgery invariant curve. In (a), those visible vertices that are contained in $\overline\site$ are outlined in gray.}
\label{fig:special_curves}
\end{figure}
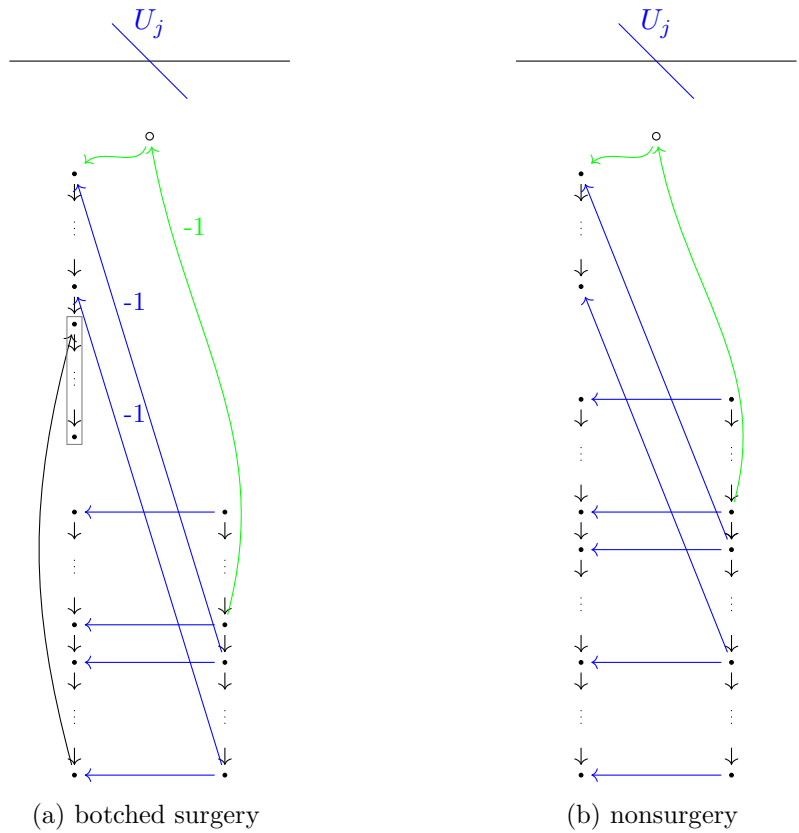

\begin{Lemma} \label{lemma:type23_curve}
Suppose $\tilde{p}_\gamma \in \MM$ is obtained through a botched butterfly surgery or nonsurgery construction with respect to $p \in \bowvar(\brane)^\torus$. Then,
\begin{enumerate}\itemsep=0pt
\item[$(1)$] $\tilde{p}_\gamma$ satisfies the S$1$, S$2$, and $\nu$ stability conditions.
\item[$(2)$] $p_\gamma \neq p$, so that $\dim(\torus.p_\gamma) > 0$.
\end{enumerate}
\end{Lemma}

\begin{proof}
(1) S1 follows from the fact that the $A$ maps are injective. S2 can be verified in the same way as for butterfly surgeries (see Lemma~\ref{lemma:type1_curve}). For the $\nu$ condition, we have to modify the strategy for butterfly surgeries slightly. The novelty is in botched butterfly surgeries moving part of a higher butterfly to a lower butterfly and nonsurgery curves. In these cases, for some vertices $w'$ of the $U_{j'}$ butterfly under $U_j^+$, there are two distinct $A_{U_j}$ edges with source $w'$. Namely, \smash{$A_{U_j}^\gamma(w') = A_{U_j}(w') + w$}, where $w$ is a vertex of the $U_j$ butterfly under \smash{$U_j^-$}, and we use the superscript of $\gamma$ to distinguish between the maps of $\tilde{p}$ and $\tilde{p}_\gamma$. However, $w$ can be obtained by repeatedly applying \smash{$B_{U_j}^-$} to \smash{$a_{U_j}(\pm 1)$}. Any vertex of the \smash{$U_{j'}$} butterfly can be obtained by~applying a sequence of $A$, $B$, $C$, $D$ maps to $a_{U_{j'}}(\pm 1)$ and canceling out the components of $A_{U_j}$ landing in the $U_j$ butterfly, as described above.

(2) Let \smash{$\tilde{p} \in \prequot$} be the point described by the butterfly diagram of $p$. The statement for~bot\-ched butterfly surgeries moving part of a higher butterfly to a lower butterfly and non\-sur\-gery curves~follows immediately from the fact that $b_{U_j}$ vanishes in $\tilde{p}$ but not in $\tilde{p}_\gamma$. For \linebreak butterfly~surgeries moving part of a lower butterfly to a higher butterfly, let
\smash{$\ell = d_{U_{j'}^-}^{U_{j'}}$},
the number of vertices of the $U_{j'}$ butterfly below $U_{j'}^-$. So, \smash{$B^\ell a_{U_{j'}}$} vanishes in $\tilde{p}$ but not in~$\tilde{p}_\gamma$.
\end{proof}

\begin{Lemma} \label{lemma:nonsurgery}
Let $\bowvar(\brane)$ be a separated bow variety and let $c = (c_1, \dots , c_m)$ denote the margin vector of D5 brane charges. Let $p\in \bowvar(\brane)^\torus$, and $1\leq j < j' \leq m$. Then, there are $\max\{0, c_{j'} - c_j\}$ invariant curves containing $p$ with tangent weights
\[
\frac{\ub_j}{\ub_{j'}}\hb, \frac{\ub_j}{\ub_{j'}}\hb^2, \dots , \frac{\ub_j}{\ub_{j'}}\hb^{\max\{0, c_{j'} - c_j\}}
\]
at $p$.
\end{Lemma}

\subsection{Young diagram surgery}
The combinatorics of Sections~\ref{sec:butterfly_surgery} and~\ref{sec:botched_surgery} can be reformulated using the more familiar language of partitions and Young diagrams. Again, we fix a separated bow variety $\bowvar(\brane)$ and a~fixed point~${p\in \bowvar(\brane)^\torus}$. Let $M$ be the associated BCT and $c=(c_1, \dots , c_m)$ be the margin vector of D5 brane charges. For each D5 brane $U_j$, define a partition~\smash{$\lambda^{(j)}$} with $c_j$ distinct parts~\smash{$\lambda^{(j)}_1 > \cdots > \lambda^{(j)}_{c_i} > 0$} by
\begin{enumerate}\itemsep=0pt
\item[(1)] letting $i_1 < i_2 < \cdots < i_{c_j}$ be the indices for which $M(i_k, j) = 1$ and
\item[(2)] setting $\lambda^{(j)}_k = n - i_k + 1$.
\end{enumerate}
In other words, the parts of the partition correspond to the vertical positions of the 1's in the $j$-th column of $M$ measured from the bottom. The Young diagrams associated with these partitions will be drawn in a nonstandard way with the boxes aligned on the right. Each row of boxes represents a part of the partition, and the longer rows will be drawn on top. See Figure~\ref{fig:young} for an example. We will use $\lambda^{(j)}$ to denote the partition as well as its Young diagram.

\begin{figure}
\centering
\subcaptionbox{}[0.3\textwidth]
{
$
\begin{pmatrix}
1&0&0\\
0&0&1\\
0&0&1\\
1&0&0\\
1&0&0\\
1&1&0\\
0&1&1\\
0&0&0\\
1&0&0\\
0&0&1
\end{pmatrix}
$
}
\subcaptionbox{}[0.3\textwidth]
{
\begin{tikzpicture}[scale=0.35, xscale=-1]
\draw (0, 0)--(10, 0); \draw (0, -1)--(10, -1); \draw (0, -2)--(7, -2); \draw (0, -3)--(6, -3); \draw (0, -4)--(5, -4); \draw (0, -5)--(2, -5);
\draw (0, 0)--(0, -5); \draw (1, 0)--(1, -5); \draw (2, 0)--(2, -5); \draw (3, 0)--(3, -4); \draw (4, 0)--(4, -4); \draw (5, 0)--(5, -4); \draw (6, 0)--(6, -3); \draw (7, 0)--(7, -2); \draw (8, 0)--(8, -1); \draw (9, 0)--(9, -1); \draw (10, 0)--(10, -1);

\draw (0, -6)--(5, -6); \draw (0, -7)--(5, -7); \draw (0, -8)--(4, -8);
\draw (0, -6)--(0, -8); \draw (1, -6)--(1, -8); \draw (2, -6)--(2, -8); \draw (3, -6)--(3, -8); \draw (4, -6)--(4, -8); \draw (5, -6)--(5, -7);

\draw (0, -9)--(9, -9); \draw (0, -10)--(9, -10); \draw (0, -11)--(8, -11); \draw (0, -12)--(4, -12); \draw (0, -13)--(1, -13);
\draw (0, -9)--(0, -13); \draw (1, -9)--(1, -13); \draw (2, -9)--(2, -12); \draw (3, -9)--(3, -12); \draw (4, -9)--(4, -12); \draw (5, -9)--(5, -11); \draw (6, -9)--(6, -11); \draw (7, -9)--(7, -11); \draw (8, -9)--(8, -11); \draw (9, -9)--(9, -10);
\end{tikzpicture}
}
\caption{The (a) BCT and (b) Young diagrams associated to the fixed point of Figure~\ref{fig:butterfly_sides}. Our (nonstandard) convention is to draw the Young diagrams aligned on the right with longer rows on top.}
\label{fig:young}
\end{figure}
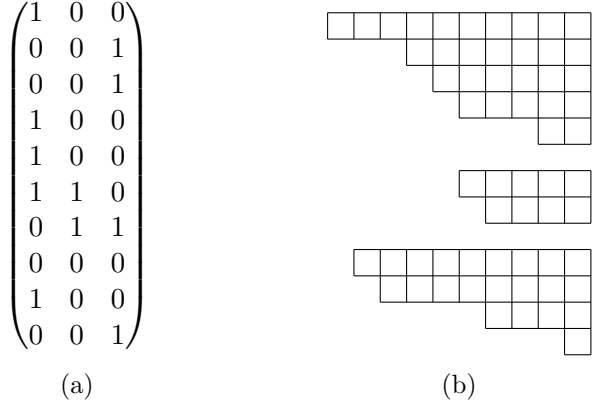

Note that these Young diagrams can also be realized by taking the vertices in the left side of each butterfly, realigning them at the top, and replacing them with boxes. Under this realization, butterfly surgeries and botched butterfly surgeries can be realized as surgery operations on Young diagrams. For $j\neq j'$, let $\site$ be a set of boxes of $\lambda^{(j)}$ with the property that if a box is in $\site$, then any box below it is also in $\site$. Suppose deleting $\site$ from $\lambda^{(j)}$ leaves a Young diagram with distinct parts, and stacking those boxes below $\lambda^{(j')}$ results in a Young diagram with distinct parts. Then, call the operation on the set of Young diagrams $\big\{\lambda^{(j)} \mid 1\leq j \leq m\big\}$ that deletes $\site$ from \smash{$\lambda^{(j)}$} and stacks it below \smash{$\lambda^{(j')}$} a ``Young diagram surgery''. If $\site$ does not contain any boxes in the rightmost column, then the Young diagram surgery corresponds to a butterfly surgery. Otherwise, if $\site$ contains at least one box of the rightmost column, then the Young diagram surgery corresponds to a botched butterfly surgery. The constraint \eqref{eqn:typeII_constraint} is reformulated as
\begin{equation} \label{eqn:box_constraint}
\site\text{ contains at least }c_j - c_{j'} + 1\text{ boxes of the rightmost column.}
\end{equation}
This is simply a reformulation of Sections~\ref{sec:butterfly_surgery} and~\ref{sec:botched_surgery}. By Lemmas~\ref{lemma:butterfly_surgery} and~\ref{lemma:botched_surgery}, Young diagram surgeries satisfying the additional constraint \eqref{eqn:box_constraint} whenever a box in the rightmost column is moved correspond to pencils of invariant curves. Given a box $\bx$ in a Young diagram $\lambda$, define its height $y(\bx)$ by $y(\bx) = 1 - i$, where $\bx$ belongs to the $i$-th row of $\lambda$. The tangent weight of the invariant curve is given by comparing the heights of any box in $\site$ before and after the surgery. Namely, the exponent of $\hb$ is $-\Delta_\surgery y = y(\bx)-y(\bx')$, where $\bx$ is any box in $\site$ and $\bx'$ is the result of applying $\surgery$ to $\bx$. Treating the Young diagram as a subset of $\RR^2$ and each box as a closed square, the number of connected components of $\site$ is precisely the number of connected components of the corresponding (botched) butterfly surgery. In Figure~\ref{fig:yd_surgeries}, we reformulate Figure~\ref{fig:butterfly_sides} in the language of Young diagram surgeries.

\begin{Lemma} \label{lemma:Young_curves}
Let $\brane$ be a separated brane diagram and $p\in\bowvar(\brane)^\torus$. Let $\surgery$ be a surgery on the corresponding Young diagrams that moves boxes from $\lambda^{(j)}$ to $\lambda^{(j')}$. If $j<j'$ and the site $\site$ contains at least one box of the rightmost column, then impose the additional constraint that $\site$ contains at least $c_j-c_{j'}+1$ boxes of the rightmost column, where $c=(c_1,\dots ,c_m)$ is the margin vector of D$5$ brane charges. Then, there is a $k$-dimensional pencil $\pen_\surgery$ of $\torus$-invariant curves containing $p$ with tangent weight
\smash{$
\frac{\ub_j}{\ub_{j'}}\hb^{-\Delta_\surgery y}$},
where $k$ is the number of connected components of~$\site$.
\end{Lemma}

\begin{figure}
\centering
\subcaptionbox{$w = \frac{\ub_1}{\ub_3}\hb$}
{
\begin{tikzpicture}[scale=0.31, xscale=-1]
\draw (0, -2)--(10, -2); \draw (0, -3)--(10, -3); \draw (0, -4)--(4, -4); \draw (0, -5)--(1, -5);
\draw (0, -2)--(0, -5); \draw (1, -2)--(1, -5); \draw (2, -2)--(2, -4); \draw (3, -2)--(3, -4); \draw (4, -2)--(4, -4); \draw (5, -2)--(5, -3); \draw (6, -2)--(6, -3); \draw (7, -2)--(7, -3); \draw (8, -2)--(8, -3); \draw (9, -2)--(9, -3); \draw (10, -2)--(10, -3);

\draw (0, -6)--(5, -6); \draw (0, -7)--(5, -7); \draw (0, -8)--(4, -8);
\draw (0, -6)--(0, -8); \draw (1, -6)--(1, -8); \draw (2, -6)--(2, -8); \draw (3, -6)--(3, -8); \draw (4, -6)--(4, -8); \draw (5, -6)--(5, -7);

\draw (0, -9)--(9, -9); \draw (0, -10)--(9, -10); \draw (0, -11)--(8, -11); \draw (0, -12)--(4, -12); \draw (0, -13)--(1, -13);
\draw (0, -9)--(0, -13); \draw (1, -9)--(1, -13); \draw (2, -9)--(2, -12); \draw (3, -9)--(3, -12); \draw (4, -9)--(4, -12); \draw (5, -9)--(5, -11); \draw (6, -9)--(6, -11); \draw (7, -9)--(7, -11); \draw (8, -9)--(8, -11); \draw (9, -9)--(9, -10);

\begin{scope}[yshift=-10cm]
\draw[fill=gray] (4, -1) rectangle (5, -2); \draw[fill=gray] (5, -1) rectangle (6, -2); \draw[fill=gray] (6, -1) rectangle (7, -2);

\draw[fill=gray] (1, -2) rectangle (2, -3); \draw[fill=gray] (2, -2) rectangle (3, -3); \draw[fill=gray] (3, -2) rectangle (4, -3);
\draw[fill=gray] (4, -2) rectangle (5, -3); \draw[fill=gray] (5, -2) rectangle (6, -3);

\draw[fill=gray] (0, -3) rectangle (1, -4); \draw[fill=gray] (1, -3) rectangle (2, -4); \draw[fill=gray] (2, -3) rectangle (3, -4);
\draw[fill=gray] (3, -3) rectangle (4, -4); \draw[fill=gray] (4, -3) rectangle (5, -4);

\draw[fill=gray] (0, -4) rectangle (1, -5); \draw[fill=gray] (1, -4) rectangle (2, -5);
\end{scope}

\draw[->, thick] (13,-6.5)--(11,-6.5);

\begin{scope}[xshift = 14cm, yshift = -2cm]
\draw (0, 0)--(10, 0); \draw (0, -1)--(10, -1); \draw (0, -2)--(7, -2); \draw (0, -3)--(6, -3); \draw (0, -4)--(5, -4); \draw (0, -5)--(2, -5);
\draw (0, 0)--(0, -5); \draw (1, 0)--(1, -5); \draw (2, 0)--(2, -5); \draw (3, 0)--(3, -4); \draw (4, 0)--(4, -4); \draw (5, 0)--(5, -4); \draw (6, 0)--(6, -3); \draw (7, 0)--(7, -2); \draw (8, 0)--(8, -1); \draw (9, 0)--(9, -1); \draw (10, 0)--(10, -1);

\draw (0, -6)--(5, -6); \draw (0, -7)--(5, -7); \draw (0, -8)--(4, -8);
\draw (0, -6)--(0, -8); \draw (1, -6)--(1, -8); \draw (2, -6)--(2, -8); \draw (3, -6)--(3, -8); \draw (4, -6)--(4, -8); \draw (5, -6)--(5, -7);

\draw (0, -9)--(9, -9); \draw (0, -10)--(9, -10); \draw (0, -11)--(8, -11); \draw (0, -12)--(4, -12); \draw (0, -13)--(1, -13);
\draw (0, -9)--(0, -13); \draw (1, -9)--(1, -13); \draw (2, -9)--(2, -12); \draw (3, -9)--(3, -12); \draw (4, -9)--(4, -12); \draw (5, -9)--(5, -11); \draw (6, -9)--(6, -11); \draw (7, -9)--(7, -11); \draw (8, -9)--(8, -11); \draw (9, -9)--(9, -10);

\draw[fill=gray] (4, -1) rectangle (5, -2); \draw[fill=gray] (5, -1) rectangle (6, -2); \draw[fill=gray] (6, -1) rectangle (7, -2);

\draw[fill=gray] (1, -2) rectangle (2, -3); \draw[fill=gray] (2, -2) rectangle (3, -3); \draw[fill=gray] (3, -2) rectangle (4, -3);
\draw[fill=gray] (4, -2) rectangle (5, -3); \draw[fill=gray] (5, -2) rectangle (6, -3);

\draw[fill=gray] (0, -3) rectangle (1, -4); \draw[fill=gray] (1, -3) rectangle (2, -4); \draw[fill=gray] (2, -3) rectangle (3, -4);
\draw[fill=gray] (3, -3) rectangle (4, -4); \draw[fill=gray] (4, -3) rectangle (5, -4);

\draw[fill=gray] (0, -4) rectangle (1, -5); \draw[fill=gray] (1, -4) rectangle (2, -5);
\end{scope}
\end{tikzpicture}
}
\hfill
\subcaptionbox{$w = \frac{\ub_3}{\ub_2}$}
{
\begin{tikzpicture}[scale=0.31, xscale=-1]
\draw (0, 0)--(10, 0); \draw (0, -1)--(10, -1); \draw (0, -2)--(7, -2); \draw (0, -3)--(6, -3); \draw (0, -4)--(5, -4); \draw (0, -5)--(2, -5);
\draw (0, 0)--(0, -5); \draw (1, 0)--(1, -5); \draw (2, 0)--(2, -5); \draw (3, 0)--(3, -4); \draw (4, 0)--(4, -4); \draw (5, 0)--(5, -4); \draw (6, 0)--(6, -3); \draw (7, 0)--(7, -2); \draw (8, 0)--(8, -1); \draw (9, 0)--(9, -1); \draw (10, 0)--(10, -1);

\draw (0, -6)--(5, -6); \draw (0, -7)--(5, -7); \draw (0, -8)--(4, -8);
\draw (0, -6)--(0, -8); \draw (1, -6)--(1, -8); \draw (2, -6)--(2, -8); \draw (3, -6)--(3, -8); \draw (4, -6)--(4, -8); \draw (5, -6)--(5, -7);

\draw (0, -11)--(5, -11); \draw (0, -12)--(5, -12); \draw (0, -13)--(4, -13);
\draw (0, -11)--(0, -13); \draw (1, -11)--(1, -13); \draw (2, -11)--(2, -13); \draw (3, -11)--(3, -13); \draw (4, -11)--(4, -13); \draw (5, -11)--(5, -12);

\begin{scope}[yshift=3cm]
\draw[fill=gray] (5, -9) rectangle (6, -10); \draw[fill=gray] (6, -9) rectangle (7, -10); \draw[fill=gray] (7, -9) rectangle (8, -10);
\draw[fill=gray] (8, -9) rectangle (9, -10);

\draw[fill=gray] (4, -10) rectangle (5, -11); \draw[fill=gray] (5, -10) rectangle (6, -11); \draw[fill=gray] (6, -10) rectangle (7, -11);
\draw[fill=gray] (7, -10) rectangle (8, -11);

\draw[fill=gray] (0, -11) rectangle (1, -12); \draw[fill=gray] (1, -11) rectangle (2, -12); \draw[fill=gray] (2, -11) rectangle (3, -12);
\draw[fill=gray] (3, -11) rectangle (4, -12);

\draw[fill=gray] (0, -12) rectangle (1, -13);
\end{scope}

\draw[->, thick] (13,-6.5)--(11,-6.5);

\begin{scope}[xshift = 14cm, yshift = 0cm]
\draw (0, 0)--(10, 0); \draw (0, -1)--(10, -1); \draw (0, -2)--(7, -2); \draw (0, -3)--(6, -3); \draw (0, -4)--(5, -4); \draw (0, -5)--(2, -5);
\draw (0, 0)--(0, -5); \draw (1, 0)--(1, -5); \draw (2, 0)--(2, -5); \draw (3, 0)--(3, -4); \draw (4, 0)--(4, -4); \draw (5, 0)--(5, -4); \draw (6, 0)--(6, -3); \draw (7, 0)--(7, -2); \draw (8, 0)--(8, -1); \draw (9, 0)--(9, -1); \draw (10, 0)--(10, -1);

\draw (0, -6)--(5, -6); \draw (0, -7)--(5, -7); \draw (0, -8)--(4, -8);
\draw (0, -6)--(0, -8); \draw (1, -6)--(1, -8); \draw (2, -6)--(2, -8); \draw (3, -6)--(3, -8); \draw (4, -6)--(4, -8); \draw (5, -6)--(5, -7);

\draw (0, -9)--(9, -9); \draw (0, -10)--(9, -10); \draw (0, -11)--(8, -11); \draw (0, -12)--(4, -12); \draw (0, -13)--(1, -13);
\draw (0, -9)--(0, -13); \draw (1, -9)--(1, -13); \draw (2, -9)--(2, -12); \draw (3, -9)--(3, -12); \draw (4, -9)--(4, -12); \draw (5, -9)--(5, -11); \draw (6, -9)--(6, -11); \draw (7, -9)--(7, -11); \draw (8, -9)--(8, -11); \draw (9, -9)--(9, -10);

\draw[fill=gray] (5, -9) rectangle (6, -10); \draw[fill=gray] (6, -9) rectangle (7, -10); \draw[fill=gray] (7, -9) rectangle (8, -10);
\draw[fill=gray] (8, -9) rectangle (9, -10);

\draw[fill=gray] (4, -10) rectangle (5, -11); \draw[fill=gray] (5, -10) rectangle (6, -11); \draw[fill=gray] (6, -10) rectangle (7, -11);
\draw[fill=gray] (7, -10) rectangle (8, -11);

\draw[fill=gray] (0, -11) rectangle (1, -12); \draw[fill=gray] (1, -11) rectangle (2, -12); \draw[fill=gray] (2, -11) rectangle (3, -12);
\draw[fill=gray] (3, -11) rectangle (4, -12);

\draw[fill=gray] (0, -12) rectangle (1, -13);
\end{scope}
\end{tikzpicture}
}

\caption{The two botched butterfly surgeries of Figure~\ref{fig:butterfly_sides} reformulated as Young diagram surgeries. The sites are shaded gray. Both sites are connected, so exactly two curves are obtained (no infinite pencils). The tangent weights $w$ of the corresponding invariant curves are also shown.}
\label{fig:yd_surgeries}
\end{figure}
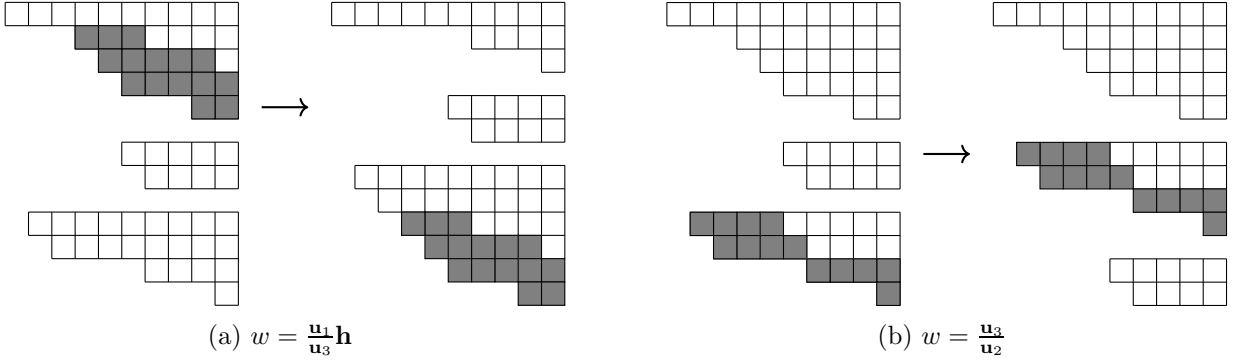

\begin{Lemma} \label{lemma:surgery_uniqueness}
Let $\surgery$ be a Young diagram surgery with connected site $\site$. Then, $\surgery$ is uniquely determined by its action on any box in $\site$.
\end{Lemma}
\begin{proof}
Since Young diagram surgeries are rigid translations, the action of $\surgery$ on a box in $\site$ completely determines the action of $\surgery$ on any other box in $\site$. It is trivial to show that the action~$\surgery$ on any box in $\site$ completely determines whether or not each adjacent box is in $\site$. The result follows from these two statements and the connectedness of $\site$.
\end{proof}

\begin{Corollary} \label{cor:disjoint_sites}
Let $\surgery_1$ and $\surgery_2$ be two connected Young diagram surgeries with $\Delta_{\surgery_1}y=\Delta_{\surgery_2}y$. Then, their sites are either disjoint or equal.
\end{Corollary}

\section{BCT block swap moves} \label{sec:block_swap}
An elementary result on BCTs is that any two BCTs with the same margin vectors are related by a sequence of ``swap moves''. A swap move acts on a BCT $M$ by taking a $\left(\begin{smallmatrix} 1 & 0 \\ 0 & 1 \end{smallmatrix}\right)$ minor and replacing it with $\left(\begin{smallmatrix} 0 & 1 \\ 1 & 0 \end{smallmatrix}\right)$ or vice-versa. A similar type of combinatorial move will provide an alternative combinatorial code for compact $\torus$-invariant curves.

\begin{Remark}
Swap moves are also used in {\rm\cite{W}}, where they are called ``simple moves'' and appear in a generalization of the Chevalley--Monk formula.
\end{Remark}

Let us first introduce some preliminary notions.

\begin{Definition} \label{def:blocks}
Let $M$ be the BCT of some $p\in\bowvar(\brane)^\torus$. Call a minor
\[
\overline M = \overline{M}(i_1, i_2, j_1, j_2) = \left(M_{ij}\right)^{i = i_1, i_1+1, \dots , i_2}_{j = j_1, j_2}
\]
consisting of two columns and any number of consecutive rows of $M$ a ``block''. Furthermore, define
\[
\delta\bigl(\overline M\bigr) = \begin{cases}
\displaystyle\sum_{i = i_1}^{i_2} (M_{i j_1} - M_{i j_2}) & \text{if }j_1 < j_2, \\
\displaystyle\sum_{i = i_1}^{i_2} (M_{i j_2} - M_{i j_1}) & \text{if }j_1 > j_2.
\end{cases}
\]
Call a block $\overline M$ ``matched'' if $\delta\bigl(\overline M\bigr) = 0$.
\end{Definition}

The following result provides an algorithmic way of constructing Young diagram surgeries. It will be a key ingredient in the proof of our main result, Theorem~\ref{thm:curves}.

\begin{Proposition} \label{prop:BCT_curves}
Let $M$ be the BCT of some $p\in\bowvar(\brane)^\torus$ where $\brane$ is separated with $m$ D$5$ branes. Then, using Definition~{\rm\ref{def:pairs}},
\begin{enumerate}\itemsep=0pt
\item[$(1)$] Every $01$-pair $(i, j_0, j_1)$ or 10-pair $(i, j_1, j_0)$ in $M$ corresponds to a Young diagram surgery $\surgery$ moving boxes from $\lambda^{(j_1)}$ to $\lambda^{(j_0)}$.

\item[$(2)$] In the context above, if $\overline{M}(i, i', j_0, j_1)$ is matched for some $i' > i$, then the site $\site$ of $\surgery$ does not contain any boxes of the rightmost column of $\lambda^{(j_1)}$. If no such matched block exists, then $\site$ contains $\big|\delta\bigl(\overline{M}(i, n, j_0, j_1)\bigr)\big|$ boxes of the rightmost column of $\lambda^{(j_1)}$.

\item[$(3)$] The relative displacement of the boxes during $\surgery$ is given by
\[
\Delta_\surgery(y) = \sum_{k = 1}^{i} (M_{k j_1} - M_{k j_0}) - 1 = s_{i j_1} - s_{i j_0} - 1.
\]
\end{enumerate}
\end{Proposition}

\begin{proof}
We will provide the proof for 10-pairs. The result for 01-pairs follows from a symmetric argument. Make the simplifying assumption that $M$ is $n\times 2$. The general case follows from applying this argument to every pair of columns in $M$. To simplify the notation, let $\overline{M}(i_1, i_2) = \overline{M}(i_1, i_2, 1, 2)$.
Suppose the $i_1$-th row of $M$ is $(1, 0)$. Let $\overline M$ be the smallest matched block of the form~$\overline{M}(i_1, i_2)$ where $i_1 < i_2$, if such a block exists. Otherwise, if no such matched block exists, let $\overline M = \overline{M}(i_1, n)$. Recall that each 1 in the $j$-th column of $\overline M$ represents a part of the Young diagram~$\lambda^{(j)}$. Let~\smash{$\overline{\lambda}^{(j)}$} be the partition consisting of the parts of $\lambda^{(j)}$ associated with the 1's in the $j$-th column of~$\overline M$. Denote the number of parts of \smash{$\overline{\lambda}^{(j)}$} by $\overline{c}_j$. Clearly, $\big|\delta\bigl(\overline{M}(i_1, i + 1)\bigr) - \delta\bigl(\overline{M}(i_1, i)\bigr)\big| \leq 1$ for all~${i_1 \leq i < n}$, and $\delta\bigl(\overline{M}(i_1, i_1)\bigr) = 1$. If it were the case that $\delta\bigl(\overline M\bigr) < 0$, then there would exist~${i > i_1}$ for which $\delta\bigl(\overline{M}(i_1, i)\bigr) = 0$. It follows that $\delta\bigl(\overline M\bigr) \geq 0$, and therefore $\overline{c}_1 \geq \overline{c}_2$.

Since the first row of $\overline M$ is $(1, 0)$, we have
\[
\overline{\lambda}^{(1)}_1 > \overline{\lambda}^{(2)}_1.
\]
 We claim that \smash{$\overline{\lambda}^{(1)}_{\overline{i}+1} \geq \overline{\lambda}^{(2)}_{\overline{i}}$} for all $\overline{i}=1, \dots , \overline{c}_2 - 1$. Suppose to the contrary that for some $\overline{i}$, we have
\[
\overline{\lambda}^{(1)}_{\overline{i}+1} < \overline{\lambda}^{(2)}_{\overline{i}}.
\] Let the $(1, 0)$ row corresponding to \smash{$\overline{\lambda}^{(1)}_{\overline{i}+1}$} be row $i'+1$ of $M$. Then, we have
$\delta\bigl(\overline{M}(i_1, i' + 1)\bigr) \leq 1$. It follows that $\delta\bigl(\overline{M}(i_1, i')\bigr) \leq 0$. By the intermediate value argument above, there is $i_1 < i < i_2$ such that~${\overline{M}(i_1, i)}$ is matched. This contradicts the definition of $\overline M$.

These inequalities imply that there is a connected Young diagram surgery which moves \smash{$\overline{\lambda}^{(1)}_{\overline{i}} - \overline{\lambda}^{(2)}_{\overline{i}} > 0$} boxes from \smash{$\overline{\lambda}^{(1)}_{\overline{i}}$} to \smash{$\overline{\lambda}^{(2)}_{\overline{i}}$}. If $\overline M$ is matched, then this surgery constitutes a surgery on $\lambda^{(1)}$ and~$\lambda^{(2)}$. Otherwise, this surgery can be extended to a surgery on $\lambda^{(1)}$ and $\lambda^{(2)}$ by moving the bottom~${\delta\bigl(\overline{M}\bigr) >0} $ rows of boxes in \smash{$\overline{\lambda}^{(1)}$} completely below $\lambda^{(2)}$. The displacement~$\Delta_\surgery y$ is easily seen to be \smash{$\sum_{k = 1}^{i_1} (M_{k 1} - M_{k 2}) - 1$}.
\end{proof}

Let $\overline M$ be a minimal matched block whose top row is a 01- or 10- pair. The proof of Proposition~\ref{prop:BCT_curves} shows that there is a connected butterfly surgery transforming $M$ by exchanging the columns of $\overline M$. The displacement of the surgery is given by applying the formula in Proposition~\ref{prop:BCT_curves}\,(3) to the top row of $\overline M$. Also note that among any set of consecutive minimal matched blocks, those whose top row is a 10- or 01-pair are associated with surgeries with disjoint sites and the same displacement. Conversely, any matched block can be decomposed uniquely into consecutive minimal matched blocks. It follows that every matched block with top row $(1, 0)$ or~$(0, 1)$ is associated with a pencil of compact invariant curves.

\begin{Definition}
Let $M$ be a BCT. A ``block swap move'' $\psi$ on $M$ consists of swapping the columns of a matched block $\overline M$ within $M$, where the first and last rows of $\overline M$ are $(1, 0)$ or $(0, 1)$. Call $\psi$ indecomposable if $\overline M$ is a minimal matched block.
\end{Definition}

Note that if the top row of a minimal matched block is $(1, 0)$, then the bottom row must be $(0, 1)$ and vice-versa. Moreover, any block swap move can be uniquely decomposed into \emph{simultaneous} indecomposable block swap moves whose associated minimal matched blocks are separated by $(1, 1)$ and $(0, 0)$ rows.

\begin{Proposition}
Let $p\in\bowvar(\brane)^\torus$, with $\brane$ separated. Let $M$ be the BCT of $p$.
\begin{enumerate}\itemsep=0pt
\item[$(1)$] Every block swap move $\psi$ on $M$ corresponds to a pencil $\Gamma$ of compact invariant curves containing $p$.
\item[$(2)$] Suppose $\psi$ decomposes into $k$ simultaneous indecomposable block swap moves $\psi_1, \dots , \psi_k$. Then, we have $\dim(\Gamma) = k$.
\item[$(3)$] The BCT for any fixed point in $\Gamma$ can be obtained by performing some subset of the indecomposable block swap moves $\psi_i$.
\item[$(4)$] Suppose row $i$ is the topmost row affected by $\psi$. Let $(i, j_0)$ be the index of the $0$ in row $i$ that $\psi$ swaps to $1$. Similarly, let $(i, j_1)$ be the index of the $1$ in row $i$ that $\psi$ swaps to $0$. Then, the tangent weight of $\Gamma$ at $p$ is $\frac{\ub_{j_1}}{\ub_{j_0}}\hb^{1+s_{i j_0}-s_{i j_1}}$.
\end{enumerate}
\end{Proposition}

It follows from Theorem~\ref{thm:curves} that all pencils of compact invariant curves in $\bowvar(\brane)$ are associated with block swap moves.

\section{Classification of invariant curves} \label{sec:classification}
\begin{Definition}
Let $\bowvar(\brane)$ be a separated bow variety and $p\in\bowvar(\brane)^\torus$. Define three types of invariant curves containing $p$:
\begin{itemize}\itemsep=0pt
\item[$\rm I)$] curves arising from connected butterfly surgeries $($see Section~{\rm\ref{sec:butterfly_surgery})},
\item[$\rm II)$] curves arising from connected botched butterfly surgeries $($see Section~{\rm\ref{sec:botched_surgery})},
\item[$\rm III)$] nonsurgery curves $($see Section~{\rm\ref{sec:nonsurgery})}.
\end{itemize}
\end{Definition}
A priori, the type of a compact curve may depend on which of the two fixed points you apply this definition to. We will show in Corollary~\ref{cor:compactness} that, in fact, type II and III curves are always noncompact. Because butterfly surgeries are reversible, a type I curve containing $p_1$ is also type~I when viewed with respect to its other fixed point $p_2$.

\begin{Proposition} \label{prop:disjoint_types}
The type $I$, type $II$, and type $III$ curves containing a fixed point $p$ form three disjoint sets. Moreover, no type $II$ curve has the same tangent weight as a type $III$ curve.
\end{Proposition}

\begin{proof}
It is easy to see from Lemma~\ref{lemma:nonsurgery} and the constraint \eqref{eqn:typeII_constraint} that a type III curve cannot have the same tangent weight at $p$ as a type II curve. Thus, no curve is both type II and type~III.

A type I, type II, or type III curve $\gamma$ is constructed as the closure of the orbit of some point~$p_\gamma$ for which we give an explicit representation in \smash{$\tilde{p}_\gamma \in \prequot$}. Suppose $\gamma$ is type II and has tangent weight $\hb^r \ub_j/\ub_j'$ at $p$. If $j > j'$, then
\[
B^\ell a_{U_{j'}} \neq 0 \qquad \text{for} \quad \ell = d_{U_{j'}^-}^{U_{j'}}.
\]
On the other hand, for any type~I curve, \smash{$B^\ell a_{U_{j'}} = 0$} in $\tilde{p}_\gamma$. If $j < j'$, then $b_{U_{j'}} \neq 0$. Any type~I curve, however, has $b_{U_{j'}} = 0$. It follows that no curve is both type I and type II.

Finally, if $\gamma$ is type I, then $\tilde{p}_\gamma$ has $b=0$, and if $\gamma$ is type III, $\tilde{p}_\gamma$ has $b\neq 0$. Thus, no curve can be both type I and type III.
\end{proof}

Curves of type I, II, and III will form a basic set from which all pencils of invariant curves can be generated.

\begin{Proposition} \label{prop:pencil_generation}
Let $\gamma_1, \dots , \gamma_k$ be distinct invariant curves containing $p\in\bowvar(\brane)^\torus$, each of type $I$, type $II$, or type $III$, and all having the same tangent weight at $p$. Then, there is a $k$-dimensional pencil of invariant curves $\Gamma$ with $\gamma_1, \dots , \gamma_k$ in its boundary. Moreover, every vector in $T_p\gamma_1 \oplus\cdots\oplus T_p\gamma_k$ is in $T_p\gamma$ for some $\gamma\in\Gamma$.
\end{Proposition}

\begin{proof}
Since the tangent weights of the $\gamma_i$ are all equal, Corollary~\ref{cor:disjoint_sites} implies that the sites of the surgeries associated with any type I and II $\gamma_i$ are disjoint. Since any botched butterfly surgery must move the bottom vertex of $X_n$, there is at most 1 type II $\gamma_i$.
We can also see from Lemma~\ref{lemma:nonsurgery} that the tangent weights of the type III curves are distinct. Therefore, there is at most 1 type III $\gamma_i$.
From Proposition~\ref{prop:disjoint_types}, type II and type III curves never share the same tangent weight. Hence, if one of the $\gamma_i$ is type II, then none are type III, and vice-versa. In summary, there is at most 1 $i$ for which $\gamma_i$ is not type I.

It follows from the above that each $\gamma_i$ is constructed by adding a set $E_i$ of new edges to disjoint portions of the butterfly diagram of $p$. Attaching a factor of $t \in \CC$ to each edge in $E_i$ results in a lift $\tilde{\gamma}_i\colon\CC \to \prequot$ of $\gamma_i$ in a neighborhood of $p$. Section~\ref{sec:butterfly_surgery} describes this construction for type I curves. According to \cite[Section~2.5]{NT}, there is a complex
\[
\begin{tikzcd}
\bigoplus_X \End(W_X) \arrow[r, "\alpha"] & \MM \arrow[r, "\beta"] & \NN,
\end{tikzcd}
\]
such that $T_p\bowvar(\brane) \cong \ker\beta/\im\alpha$. In particular, $\alpha$ is the differential of the $\G$-action and $\ker(\beta) \cong \smash{T_{\tilde{p}}\prequot}$, where \smash{$\tilde{p} \in \prequot$} is the representative of $p$ given by the butterfly diagram. We will not describe this complex in detail, as it is not required for the proof. Note that $\tilde{\gamma}_i'(0) + \im\alpha$ spans~$T_p\gamma_i$.

Fix $(z_1, \dots , z_k) \in \CC^k - \{0\}$, and let $s = \sum_{i=1}^k z_i (\tilde{\gamma}_i'(0) + \im\alpha)$. Let $\tilde{p}_\gamma \in \MM$ be obtained from $\tilde{p}$ by adding $E_i$ with a factor of $z_i$ for all $i$.
The arguments of Sections~\ref{sec:butterfly_surgery}, \ref{sec:botched_surgery}, and \ref{sec:nonsurgery} can be applied to show that $\tilde{p}_\gamma \in \prequot$ and $\dim(\torus.p_\gamma) = 1$, where $p_\gamma$ is the image of $\tilde{p}_\gamma$ in $\bowvar(\brane)$. Therefore, $\gamma = \overline{\torus.p_\gamma}$ is a $\torus$-invariant curve. As before, we can construct a lift \smash{$\tilde{\gamma} = \sum_{i=1}^k z_i \tilde{\gamma}_i$} of~$\gamma$. Putting everything together, we have
\[
0 \neq T_p\gamma = \CC\big\{ \tilde{\gamma}'(0) + \im\alpha \big\} = \CC \left\{ \sum_{i=1}^k z_i \bigl( \tilde{\gamma}_i'(0) + \im\alpha \bigr) \right\} = \CC \{s\}.
\]
It follows that $T_p\gamma_1, \dots , T_p\gamma_k$ are linearly independent, and every vector in $T_p\gamma_1 \oplus \cdots \oplus T_p\gamma_k$ is in $T_p\gamma$ for some invariant curve $\gamma$. Allowing $z_1, \dots , z_k$ to vary yields the desired pencil $\Gamma$ of invariant curves.
See Section~\ref{sec:butterfly_surgery} for the construction of $\Gamma$ in the case where all $\gamma_i$ are type I.
\end{proof}

\begin{Definition} \label{def:spanned}
We say that the pencil $\Gamma$ in Proposition~{\rm\ref{prop:pencil_generation}} is ``spanned'' by the curves $\gamma_1,\allowbreak \dots , \gamma_k$.
\end{Definition}

 This leads us to our main result.

\begin{Theorem} \label{thm:curves}
Let $\brane$ be separated and $p\in\bowvar(\brane)^\torus$. All $\torus$-invariant pencils of curves containing~$p$ are spanned by curves of type $I$, type $II$, and type $III$.
\end{Theorem}

\begin{proof}
Due to Proposition~\ref{prop:pencil_generation}, it suffices to construct $k$ distinct type I, II, or III $\torus$-invariant curves with tangent weight $w$ at $p$ for each tangent weight $w$ of multiplicity $k$ in $T_p\bowvar(\brane)$. For simplicity, let us reduce to the case where $m=2$, and the BCT $M$ of $p$ is an $n\times 2$ matrix. The general case is obtained by applying the same argument to each pair of columns in $M$.
From Theorem~\ref{thm:weights}, the tangent weights at $p$ come in two varieties: those of the form $\frac{\ub_1}{\ub_2}\hb^r$ and those of the form $\frac{\ub_2}{\ub_1}\hb^r$. It is easy to see from Proposition~\ref{prop:BCT_curves} and Lemma~\ref{lemma:Young_curves} that the weights of the latter form are accounted for by type I and II curves. Hence, we turn our attention to the former.\looseness=-1

We restrict our attention to weights of the form $\frac{\ub_1}{\ub_2}\hb^r$. There is one such weight associated with each 01-pair. First, we consider the case where $c_1 \leq c_2$. In this case, the extra constraint~\eqref{eqn:box_constraint} is automatically satisfied. If there are no 10-pairs in $M$, then by Lemma~\ref{lemma:nonsurgery}, the tangent weights under consideration are accounted for by type III curves. Induct on the number of 10-pairs. Assume that $M$ has both 10- and 01-pairs. We see that whenever a 10-pair and a 01-pair are only separated by 00- or 11-pairs, the tangent weight of the curve associated with the 10-pair via Proposition \ref{prop:BCT_curves} is equal to the $\frac{\ub_1}{\ub_2}\hb^r$ tangent weight associated with the~01-pair. Note also that the minimal matched block between such a 01-pair and 10-pair makes no net contribution to the tangent weight formula of Theorem~\ref{thm:weights} or the displacement formula of Proposition~\ref{prop:BCT_curves}\,(3). A 01-pair with a 10-pair separated by only 00- or 11-pairs always exists as long as there are both 01- and 10-pairs, so deleting the block between them and applying induction completes the argument.

Next, we consider the case $c_1 > c_2$. Let the row indices of the 10-pairs be $i_1, \dots , i_{c_1}$. Define $\overline{M}(i_1, i_2) = \overline{M}(i_1, i_2, 1, 2)$ and $\delta_i = \delta\bigl(\overline{M}(i, n)\bigr)$. Clearly, $|\delta_n| \leq 1$, $\delta_1 = c_1 - c_2$, and $|\delta_i - \delta_{i + 1}| \leq 1$. Since $\delta_i - \delta_{i + 1} = 1$ if and only if $M$ has a 10-pair in the $i$-th row, we have
\[
\{1, \dots , c_1 - c_2\} \subset \{\delta_{i_k}\mid k = 1, \dots , c_1\}.
\]
Moreover, defining
\begin{gather*}
k_1 = \max\{k\mid \delta_{i_k} = c_1 - c_2\},\qquad k_2 = \max\{k\ |\ \delta_{i_k} = c_1 - c_2 - 1\},\qquad \dots ,\\
 k_{c_1 - c_2} = \max\{k\mid \delta_{i_k} = 1\},
\end{gather*}
it follows from an intermediate value argument that $k_1 < k_2 < \cdots < k_{c_1 - c_2}$. We collect two basic facts:
\begin{enumerate}\itemsep=0pt
\item[(1)] If $\delta_{i} - \delta_{i'} = 0$ for $i < i'$, then $\delta\bigl(\overline{M}(i, i' - 1)\bigr) = 0$.
\item[(2)] If $\delta_{i_k} - \delta_{i_{k'}} = 1$ for $k < k'$, then $\delta\bigl(\overline{M}(i_k + 1, i_{k'} -1)\bigr) = 0$.
 \end{enumerate}
If $i_{k_1} > 1$, applying (1) with $i = 1, i' = i_{k_1}$ tells us that $\overline{M}(1, i_{k_1} - 1)$ is matched. Applying (2) to consecutive indices in the ordered sequence $k_1, \dots , k_{c_1 - c_2}$ tells us that the blocks lying between the 10-pairs in rows $i_{k_1}, \dots , i_{k_{c_1- c_2}}$ are matched. If $i_{k_{c_1 - c_2}} < n$, the previous two sentences imply that \smash{$\overline{M}(i_{k_{c_1 - c_2}} + 1, n)$} is matched. Hence, rows $i_{k_1}, \dots , i_{k_{c_2 - c_1}}$ separate $M$ into matched blocks.

Our next goal is to show that every 10-pair except the those in rows $i_{k_1}, \dots , i_{k_{c_1-c_2}}$ are associated with an invariant curve. By Proposition~\ref{prop:BCT_curves}, a 10-pair in row $i$ is associated with a~curve if it is the top row of some matched block or if it satisfies the constraint $\delta_i \geq c_1 - c_2 + 1$ (constraint~\eqref{eqn:box_constraint}). Let $(i, 1, 2)$ be a 10-pair where $i\neq i_k$ for $k = 1, \dots , c_1 - c_2$. Suppose~${\delta_i \leq c_1 - c_2}$. If $\delta_i < 1$, then there exists $i' > i$ such that $\delta\bigl(\overline{M}(i, i')\bigr) = 0$ by an intermediate value argument. In other words, $(i, 1, 2)$ is the top of a matched block. If $1 \leq \delta_i \leq c_1 - c_2$, then $\delta_i = \delta_{i_k}$ for some~$k$. We have $i < i_k$ by definition of $i_k$, so $\overline{M}(i, i_k - 1)$ is matched by (1).
Applying the argument above for the $c_1 \leq c_2$ case to each matched block lying between rows $i_{k_1}, \dots , i_{k_{c_1 - c_2}}$ completes the proof.
\end{proof}

\begin{Remark}
While not necessary for the proof, it is easy to show that the $10$-pairs in rows $i_{k_1}, \dots , i_{k_{c_1 - c_2}}$ correspond to Young diagram surgeries that move at least one block of the right column but fail constraint \eqref{eqn:box_constraint}.
\end{Remark}

At first glance, the existence of nonsurgery (type III) curves and the constraint \eqref{eqn:box_constraint} might seem mysterious. The proof of Theorem~\ref{thm:curves} elucidates their role. By Theorem~\ref{thm:weights}, each 01-pair corresponds to two tangent weights, $w$ and $\hb w^{-1}$. Each 01-pair is also associated with a Young diagram surgery, by Proposition~\ref{prop:BCT_curves}. Hence, it is natural to expect half of the tangent weights to be accounted for by these surgeries, and indeed this is the case. On the other hand, we might expect the remaining half of the tangent weights to be accounted for by surgeries associated with 10-pairs. If there are fewer 10- than 01-pairs, we get a deficit of weights, which is made up by nonsurgery curves. Otherwise, if there are more 10- than 01-pairs, constraint \eqref{eqn:box_constraint} kicks in to correct the surplus.

\begin{Corollary} \label{cor:compactness}
Let $\brane$ be separated and $p\in\bowvar(\brane)$. Invariant curves of type I containing $p$ are compact, and those of type $II$ or $III$ are noncompact.
\end{Corollary}
\begin{proof}
We have already shown in Section~\ref{sec:butterfly_surgery} that type I curves are compact. Let $\gamma$ be a~compact invariant curve containing $p_1$. Then, $\gamma$ contains another fixed point $p_2$. Recall that the tangent weights of $\gamma$ at $p_1$ and $p_2$ are reciprocols. First, assume that \smash{$T_{p_1}\gamma = \frac{\ub_j}{\ub_{j'}}\hb^r$}, where $j<j'$. Suppose to the contrary that $\gamma$ is type II or III.
In Sections~\ref{sec:botched_surgery} and~\ref{sec:nonsurgery}, we constructed a point~${p_{\gamma, 1}\in\bowvar(\brane)}$ such that $\gamma = \overline{\torus.p_{\gamma, 1}}.$ From the form of the tangent weight, we see that~$p_{\gamma, 1}$ is represented by a point \smash{$\tilde{p}_{\gamma, 1} \in \prequot$} with $b\neq 0$.
Viewing $\gamma$ from the perspective of $p_2$, Theorem~\ref{thm:curves} implies that $\gamma$ belongs to a pencil of curves spanned by type I, II, and III curves containing~$p_2$. Thus, the surgery constructions with respect to $p_2$ give us a point $p_{\gamma, 2}\in\bowvar(\brane)$ such that~${\gamma = \overline{\torus.p_{\gamma, 2}}}$. In particular, $p_{\gamma, 1}$ and $p_{\gamma, 2}$ are in the same $\torus$-orbit. It follows that $p_{\gamma, 2}$ must be represented by an element \smash{$\tilde{p}_{\gamma, 2} \in \prequot$} with $b\neq 0$.
However, examining the constructions of Sections~\ref{sec:butterfly_surgery}, \ref{sec:botched_surgery} and~\ref{sec:nonsurgery}, pencils of curves with tangent weight \smash{$\frac{\ub_{j'}}{\ub_j}\hb^{-r}$}, where $j<j'$, are given by orbit closures of points represented by elements of $\prequot$ with $b=0$. Contradiction!

The case where \smash{$T_{p_1}\gamma = \frac{\ub_{j}}{\ub_{j'}}\hb^r$}, where $j > j'$ is similar. For this case, we do not need to consider type III curves. If the curve is type II, then $b=0$ and \smash{$B^\ell a_{U_{j'}} \neq 0$} for \smash{$\ell = d_{U_{j'}^-}^{U_{j'}}$} in $\tilde{p}_{\gamma, 1}$. Hence, we also have $b=0$ and $B^\ell a_{U_{j'}} \neq 0$ in $\tilde{p}_{\gamma, 2}$. The invariant curves containing $p_2$ with tangent weight \smash{$\frac{\ub_{j'}}{\ub_j}\hb^{-r}$} all violate one of these two properties. Contradiction!
\end{proof}

\section{Examples of invariant curves} \label{sec:examples}
We will apply the classification of Section~\ref{sec:classification} to describe the invariant curves of the example bow varieties from \cite{RS} in terms of Young diagram surgeries.

\subsection[Invariant curves ...]{Invariant curves for $\boldsymbol{\bowvar(\ttt{\fs 1\fs 2\fs 3\fs 4\fs 5\bs 2\bs})}$} \label{sec:dualofGr_2(C^5)}
Letting $\brane = \ttt{\fs 1\fs 2\fs 3\fs 4\fs 5\bs 2\bs}$, the bow variety $\bowvar(\brane)$ is Hanany--Witten isomorphic to the 3d mirror dual of $T^*\mathrm{Gr}(2,5)$. Its $\torus$-fixed points are in bijection with the $2$-element subsets of the set $\{1,2,3,4,5\}$, where the subset $\{k,l\}$ corresponds to the tie diagram with $U_2$ connected to both $V_k$ and $V_l$, and $U_1$ connected to the three remaining NS5 branes. This fixed point will be denoted $kl$. The invariant curves from Figure \ref{fig:GKM1} were computed via Young diagram surgeries.

\begin{figure}[t]
\centering
 \begin{tikzpicture}[scale=1.52]

\draw[] (3.5,-.5) -- (-0.5,3.5);
\draw[] (3.5,1.5) -- (.5,4.5);
\draw[] (3,4) -- (1.5,5.5);

\draw [](3,0) to [out=40,in=-40] (3,4);
\draw [](3,2) to [out=40,in=-40] (3,6);

\draw[] (1.5,.5) -- (3,2);
\draw[] (.5,1.5) -- (3.5,4.5);
\draw[] (-.5,2.5) -- (3.5,6.5);
\draw[] (3,0) -- (2.5,-.5);
\draw[] (3,6) -- (2.5,6.5);

 \node[] at (2.5,.2) {\scalebox{.6}[.6]{$\frac{\ub_1}{\ub_2}\hb^{-2}$}};
 \node[] at (2.6,-.2) {\scalebox{.6}[.6]{$\frac{\ub_2}{\ub_1}\hb^{2}$}};
 \node[] at (3.5,.2) {\scalebox{.6}[.6]{$\frac{\ub_1}{\ub_2}\hb^{-1}$}};
 \node[] at (3.44,-.2) {\scalebox{.6}[.6]{$\frac{\ub_2}{\ub_1}\hb^3$}};

 \node[] at (1.5,1.2) {\scalebox{.6}[.6]{$\frac{\ub_1}{\ub_2}\hb^{-1}$}};
 \node[] at (1.6,.8) {\scalebox{.6}[.6]{$\frac{\ub_2}{\ub_1}\hb^{2}$}};
 \node[] at (2.55,1.2) {\scalebox{.6}[.6]{$\frac{\ub_1}{\ub_2}\hb^{-1}$}};
 \node[] at (2.45,.8) {\scalebox{.6}[.6]{$\frac{\ub_2}{\ub_1}\hb^2$}};

 \node[] at (0.6,2.2) {\scalebox{.6}[.6]{$\frac{\ub_1}{\ub_2}$}};
 \node[] at (0.6,1.8) {\scalebox{.6}[.6]{$\frac{\ub_2}{\ub_1}\hb^2$}};
 \node[] at (1.14,2.33) {\scalebox{.6}[.6]{$\frac{\ub_1}{\ub_2}\hb^{-1}$}};
 \node[] at (1.2,1.6) {\scalebox{.6}[.6]{$\frac{\ub_2}{\ub_1}\hb$}};

 \node[] at (2.9,2.4) {\scalebox{.6}[.6]{$\frac{\ub_1}{\ub_2}\hb^{-1}$}};
 \node[] at (2.8,1.55) {\scalebox{.6}[.6]{$\frac{\ub_2}{\ub_1}\hb$}};
 \node[] at (3.4,2.2) {\scalebox{.6}[.6]{$\frac{\ub_1}{\ub_2}$}};
 \node[] at (3.45,1.8) {\scalebox{.6}[.6]{$\frac{\ub_2}{\ub_1}\hb^2$}};

 \node[] at (1.6,3.2) {\scalebox{.6}[.6]{$\frac{\ub_1}{\ub_2}$}};
 \node[] at (1.6,2.8) {\scalebox{.6}[.6]{$\frac{\ub_2}{\ub_1}\hb$}};
 \node[] at (2.4,3.2) {\scalebox{.6}[.6]{$\frac{\ub_1}{\ub_2}$}};
 \node[] at (2.4,2.8) {\scalebox{.6}[.6]{$\frac{\ub_2}{\ub_1}\hb$}};

 \node[] at (-.4,3.2) {\scalebox{.6}[.6]{$\frac{\ub_1}{\ub_2}\hb$}};
 \node[] at (-.4,2.8) {\scalebox{.6}[.6]{$\frac{\ub_2}{\ub_1}\hb^2$}};
 \node[] at (.5,3.2) {\scalebox{.6}[.6]{$\frac{\ub_1}{\ub_2}\hb^{-1}$}};
 \node[] at (.4,2.8) {\scalebox{.6}[.6]{$\frac{\ub_2}{\ub_1}$}};

 \node[] at (.6,4.2) {\scalebox{.6}[.6]{$\frac{\ub_1}{\ub_2}\hb$}};
 \node[] at (.6,3.8) {\scalebox{.6}[.6]{$\frac{\ub_2}{\ub_1}\hb$}};
 \node[] at (1.2,4.35) {\scalebox{.6}[.6]{$\frac{\ub_1}{\ub_2}$}};
 \node[] at (1.25,3.6) {\scalebox{.6}[.6]{$\frac{\ub_2}{\ub_1}$}};

 \node[] at (2.85,4.3) {\scalebox{.6}[.6]{$\frac{\ub_1}{\ub_2}$}};
 \node[] at (2.75,3.55) {\scalebox{.6}[.6]{$\frac{\ub_2}{\ub_1}$}};
 \node[] at (3.45,4.2) {\scalebox{.6}[.6]{$\frac{\ub_1}{\ub_2}\hb$}};
 \node[] at (3.4,3.8) {\scalebox{.6}[.6]{$\frac{\ub_2}{\ub_1}\hb$}};

 \node[] at (1.6,5.2) {\scalebox{.6}[.6]{$\frac{\ub_1}{\ub_2}\hb$}};
 \node[] at (1.6,4.8) {\scalebox{.6}[.6]{$\frac{\ub_2}{\ub_1}$}};
 \node[] at (2.45,5.2) {\scalebox{.6}[.6]{$\frac{\ub_1}{\ub_2}\hb$}};
 \node[] at (2.4,4.8) {\scalebox{.6}[.6]{$\frac{\ub_2}{\ub_1}$}};

 \node[] at (2.6,6.2) {\scalebox{.6}[.6]{$\frac{\ub_1}{\ub_2}\hb$}};
 \node[] at (2.6,5.8) {\scalebox{.6}[.6]{$\frac{\ub_2}{\ub_1}\hb^{-1}$}};
 \node[] at (3.45,6.2) {\scalebox{.6}[.6]{$\frac{\ub_1}{\ub_2}\hb^2$}};
 \node[] at (3.4,5.8) {\scalebox{.6}[.6]{$\frac{\ub_2}{\ub_1}$}};

 \draw [](2.2,1.2) to [out=53,in=-53] (2.2,2.8);
 \draw [](2.1,1.12) to [out=76,in=-76] (2.1,2.88);
 \draw [](2,1.15) to [out=90,in=-90] (2,2.85);
 \draw [](1.9,1.12) to [out=104,in=-104] (1.9,2.88);
 \draw [](1.8,1.2) to [out=127,in=-127] (1.8,2.8);

 \draw [](2.2,3.2) to [out=53,in=-53] (2.2,4.8);
 \draw [](2.1,3.12) to [out=76,in=-76] (2.1,4.88);
 \draw [](2,3.15) to [out=90,in=-90] (2,4.85);
 \draw [](1.9,3.12) to [out=104,in=-104] (1.9,4.88);
 \draw [](1.8,3.2) to [out=127,in=-127] (1.8,4.8);

 \draw [](2,1) to (2.35,0.5);
 \draw [](2,1) to (2.2,0.5);
 \draw [](2,1) to (2 ,0.5);
 \draw [](2,1) to (1.8,0.5);
 \draw [](2,1) to (1.65,0.5);

 \draw [](2,5) to (2.35,5.5);
 \draw [](2,5) to (2.2,5.5);
 \draw [](2,5) to (2 ,5.5);
 \draw [](2,5) to (1.8,5.5);
 \draw [](2,5) to (1.65,5.5);

 \node[fill,white,draw,circle,minimum size=.5cm,inner sep=0pt] at (3,0) {\footnotesize$45$};
 \node[fill,white,draw,circle,minimum size=.5cm,inner sep=0pt] at (2,1) {\footnotesize$35$};
 \node[fill,white,draw,circle,minimum size=.5cm,inner sep=0pt] at (3,2) {\footnotesize$34$};
 \node[fill,white,draw,circle,minimum size=.5cm,inner sep=0pt] at (1,2) {\footnotesize$25$};
 \node[fill,white,draw,circle,minimum size=.5cm,inner sep=0pt] at (2,3) {\footnotesize$24$};
 \node[fill,white,draw,circle,minimum size=.5cm,inner sep=0pt] at (3,4) {\footnotesize$23$};
 \node[fill,white,draw,circle,minimum size=.5cm,inner sep=0pt] at (0,3) {\footnotesize$15$};
 \node[fill,white,draw,circle,minimum size=.5cm,inner sep=0pt] at (1,4) {\footnotesize$14$};
 \node[fill,white,draw,circle,minimum size=.5cm,inner sep=0pt] at (2,5) {\footnotesize$13$};
 \node[fill,white,draw,circle,minimum size=.5cm,inner sep=0pt] at (3,6) {\footnotesize$12$};
 \node[draw,circle,minimum size=.5cm,inner sep=0pt] at (3,0) {\footnotesize$45$};
 \node[draw,circle,minimum size=.5cm,inner sep=0pt] at (2,1) {\footnotesize$35$};
 \node[draw,circle,minimum size=.5cm,inner sep=0pt] at (3,2) {\footnotesize$34$};
 \node[draw,circle,minimum size=.5cm,inner sep=0pt] at (1,2) {\footnotesize$25$};
 \node[draw,circle,minimum size=.5cm,inner sep=0pt] at (2,3) {\footnotesize$24$};
 \node[draw,circle,minimum size=.5cm,inner sep=0pt] at (3,4) {\footnotesize$23$};
 \node[draw,circle,minimum size=.5cm,inner sep=0pt] at (0,3) {\footnotesize$15$};
 \node[draw,circle,minimum size=.5cm,inner sep=0pt] at (1,4) {\footnotesize$14$};
 \node[draw,circle,minimum size=.5cm,inner sep=0pt] at (2,5) {\footnotesize$13$};
 \node[draw,circle,minimum size=.5cm,inner sep=0pt] at (3,6) {\footnotesize$12$};

 \end{tikzpicture}
\vspace{-1mm}

\caption{Illustration of $\torus$-fixed points and invariant curves of $\bowvar(\ttt{\fs 1\fs 2\fs 3\fs 4\fs 5\bs 2\bs})$ (with their $\torus$-weights), which is the 3d mirror dual of $T^*\mathrm{Gr}(2,5)$.} \label{fig:GKM1}
\vspace{-1mm}
\end{figure}

As an example, consider the fixed point $13$, whose BCT and Young diagrams appear in Figure~\ref{fig:fixedpoint13}. By Theorem~\ref{thm:weights}, $T_{13}\bowvar(\brane) = \frac{\ub_1}{\ub_2}\hb + \frac{\ub_2}{\ub_1} + \frac{\ub_1}{\ub_2}\hb + \frac{\ub_2}{\ub_1}$. Using Young diagram surgeries, four invariant curves are constructed, depicted in Figure~\ref{fig:13_yd_surgeries}, one for each tangent weight.

Each pair of surgeries that share a tangent weight (columns of Figure~\ref{fig:13_yd_surgeries}) have disjoint sites, and thus span pencils. The pencil of weight $\frac{\ub_2}{\ub_1}$ is spanned by two type I curves, so all the curves of this pencil are compact. The other pencil has weight $\frac{\ub_1}{\ub_2} \hb$ and is spanned by a type I curve and a type II curve, and thus only one curve of the pencil is compact.

Note that $13$ also has one Young diagram surgery that is not depicted, given by moving a~single box in the rightmost column, but it is an example of a surgery that does not satisfy constraint \eqref{eqn:box_constraint}, as $c_1-c_2+1=2$.

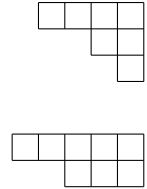
\begin{figure}[th]
\centering
\subcaptionbox{}[0.3\textwidth]
{
$
\begin{pmatrix}
0&1\\
1&0\\
0&1\\
1&0\\
1&0
\end{pmatrix}
$
}
\subcaptionbox{}[0.3\textwidth]
{
\begin{tikzpicture}[scale=0.35]
\draw (1, 0)--(5, 0); \draw (1, -1)--(5, -1); \draw (3, -2)--(5, -2); \draw (4, -3)--(5, -3);
\draw (1, 0)--(1, -1); \draw (2, 0)--(2, -1); \draw (3, 0)--(3, -2); \draw (4, 0)--(4, -3); \draw (5, 0)--(5, -3);

\draw (0, -5)--(5, -5); \draw (0, -6)--(5, -6); \draw (2, -7)--(5, -7); \draw (0, -5)--(0, -6); \draw (1, -5)--(1, -6); \draw (2, -5)--(2, -7); \draw (3, -5)--(3, -7); \draw (4, -5)--(4, -7); \draw (5, -5)--(5, -7);
\end{tikzpicture}
}
\caption{The (a) BCT and (b) Young diagrams associated to the fixed point $13$ of Figure~\ref{fig:GKM1}.}
\label{fig:fixedpoint13}
\end{figure}

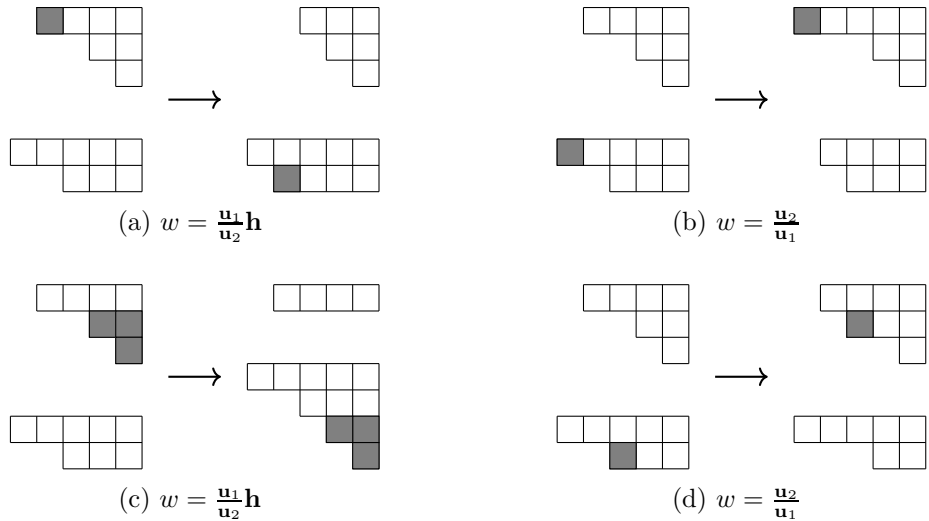
\begin{figure}
\centering
\hfill
\subcaptionbox*{(a) $w = \frac{\ub_1}{\ub_2}\hb$}
{
\begin{tikzpicture}[scale=0.35]
\draw (1, 0)--(5, 0); \draw (1, -1)--(5, -1); \draw (3, -2)--(5, -2); \draw (4, -3)--(5, -3);
\draw (1, 0)--(1, -1); \draw (2, 0)--(2, -1); \draw (3, 0)--(3, -2); \draw (4, 0)--(4, -3); \draw (5, 0)--(5, -3);

\draw (0, -5)--(5, -5); \draw (0, -6)--(5, -6); \draw (2, -7)--(5, -7); \draw (0, -5)--(0, -6); \draw (1, -5)--(1, -6); \draw (2, -5)--(2, -7); \draw (3, -5)--(3, -7); \draw (4, -5)--(4, -7); \draw (5, -5)--(5, -7);

\draw[fill=gray] (1, -1) rectangle (2, 0);

\draw[->, thick] (6, -3.5)--(8, -3.5);

\begin{scope}[xshift = 9cm]
\draw (2, 0)--(5, 0); \draw (2, -1)--(5, -1); \draw (3, -2)--(5, -2); \draw (4, -3)--(5, -3);
\draw (2, 0)--(2, -1); \draw (3, 0)--(3, -2); \draw (4, 0)--(4, -3); \draw (5, 0)--(5, -3);

\draw (0, -5)--(5, -5); \draw (0, -6)--(5, -6); \draw (1, -7)--(5, -7); \draw (0, -5)--(0, -6); \draw (1, -5)--(1, -7); \draw (2, -5)--(2, -7); \draw (3, -5)--(3, -7); \draw (4, -5)--(4, -7); \draw (5, -5)--(5, -7);

\draw[fill=gray] (1, -7) rectangle (2, -6);
\end{scope}
\end{tikzpicture}
}
\hfill
\subcaptionbox*{(b) $w = \frac{\ub_2}{\ub_1}$}
{
\begin{tikzpicture}[scale=0.35]
\draw (1, 0)--(5, 0); \draw (1, -1)--(5, -1); \draw (3, -2)--(5, -2); \draw (4, -3)--(5, -3);
\draw (1, 0)--(1, -1); \draw (2, 0)--(2, -1); \draw (3, 0)--(3, -2); \draw (4, 0)--(4, -3); \draw (5, 0)--(5, -3);

\draw (0, -5)--(5, -5); \draw (0, -6)--(5, -6); \draw (2, -7)--(5, -7); \draw (0, -5)--(0, -6); \draw (1, -5)--(1, -6); \draw (2, -5)--(2, -7); \draw (3, -5)--(3, -7); \draw (4, -5)--(4, -7); \draw (5, -5)--(5, -7);

\draw[fill=gray] (0, -5) rectangle (1, -6);

\draw[->, thick] (6, -3.5)--(8, -3.5);

\begin{scope}[xshift = 9cm]
\draw (0, 0)--(5, 0); \draw (0, -1)--(5, -1); \draw (3, -2)--(5, -2); \draw (4, -3)--(5, -3);
\draw (0, 0)--(0, -1); \draw (1, 0)--(1, -1); \draw (2, 0)--(2, -1); \draw (3, 0)--(3, -2); \draw (4, 0)--(4, -3); \draw (5, 0)--(5, -3);

\draw (1, -5)--(5, -5); \draw (1, -6)--(5, -6); \draw (2, -7)--(5, -7); \draw (1, -5)--(1, -6); \draw (2, -5)--(2, -7); \draw (3, -5)--(3, -7); \draw (4, -5)--(4, -7); \draw (5, -5)--(5, -7);

\draw[fill=gray] (0, 0) rectangle (1, -1);
\end{scope}

\end{tikzpicture}
}
\hspace*{\fill}

\vspace{15pt}
\hfill
\subcaptionbox*{(c) $w = \frac{\ub_1}{\ub_2}\hb$}
{
\begin{tikzpicture}[scale=0.35]
\draw (1, 0)--(5, 0); \draw (1, -1)--(5, -1); \draw (3, -2)--(5, -2); \draw (4, -3)--(5, -3);
\draw (1, 0)--(1, -1); \draw (2, 0)--(2, -1); \draw (3, 0)--(3, -2); \draw (4, 0)--(4, -3); \draw (5, 0)--(5, -3);

\draw (0, -5)--(5, -5); \draw (0, -6)--(5, -6); \draw (2, -7)--(5, -7); \draw (0, -5)--(0, -6); \draw (1, -5)--(1, -6); \draw (2, -5)--(2, -7); \draw (3, -5)--(3, -7); \draw (4, -5)--(4, -7); \draw (5, -5)--(5, -7);

\draw[fill=gray] (3, -1) rectangle (4, -2); \draw[fill=gray] (4, -1) rectangle (5, -2); \draw[fill=gray] (4, -2) rectangle (5, -3);

\draw[->, thick] (6, -3.5)--(8, -3.5);

\begin{scope}[xshift = 9cm]
\draw (1, 0)--(5, 0); \draw (1, -1)--(5, -1);
\draw (1, 0)--(1, -1); \draw (2, 0)--(2, -1); \draw (3, 0)--(3, -1); \draw (4, 0)--(4, -1); \draw (5, 0)--(5, -1);

\draw (0, -3)--(5, -3); \draw (0, -4)--(5, -4); \draw (2, -5)--(5, -5); \draw (3, -6)--(5, -6); \draw (4, -7)--(5, -7); \draw (0, -3)--(0, -4); \draw (1, -3)--(1, -4); \draw (2, -3)--(2, -5); \draw (3, -3)--(3, -6); \draw (4, -3)--(4, -7); \draw (5, -3)--(5, -7);

\draw[fill=gray] (4, -7) rectangle (5, -6); \draw[fill=gray] (4, -6) rectangle (5, -5); \draw[fill=gray] (3, -6) rectangle (4, -5);
\end{scope}

\end{tikzpicture}
}
\hfill
\subcaptionbox*{(d) $w = \frac{\ub_2}{\ub_1}$}
{
\begin{tikzpicture}[scale=0.35]
\draw (1, 0)--(5, 0); \draw (1, -1)--(5, -1); \draw (3, -2)--(5, -2); \draw (4, -3)--(5, -3);
\draw (1, 0)--(1, -1); \draw (2, 0)--(2, -1); \draw (3, 0)--(3, -2); \draw (4, 0)--(4, -3); \draw (5, 0)--(5, -3);

\draw (0, -5)--(5, -5); \draw (0, -6)--(5, -6); \draw (2, -7)--(5, -7); \draw (0, -5)--(0, -6); \draw (1, -5)--(1, -6); \draw (2, -5)--(2, -7); \draw (3, -5)--(3, -7); \draw (4, -5)--(4, -7); \draw (5, -5)--(5, -7);

\draw[fill=gray] (2, -6) rectangle (3, -7);

\draw[->, thick] (6, -3.5)--(8, -3.5);

\begin{scope}[xshift = 9cm]
\draw (1, 0)--(5, 0); \draw (1, -1)--(5, -1); \draw (2, -2)--(5, -2); \draw (4, -3)--(5, -3);
\draw (1, 0)--(1, -1); \draw (2, 0)--(2, -2); \draw (3, 0)--(3, -2); \draw (4, 0)--(4, -3); \draw (5, 0)--(5, -3);

\draw (0, -5)--(5, -5); \draw (0, -6)--(5, -6); \draw (3, -7)--(5, -7); \draw (0, -5)--(0, -6); \draw (1, -5)--(1, -6); \draw (2, -5)--(2, -6); \draw (3, -5)--(3, -7); \draw (4, -5)--(4, -7); \draw (5, -5)--(5, -7);

\draw[fill=gray] (2, -1) rectangle (3, -2);
\end{scope}
\end{tikzpicture}
}
\hspace*{\fill}

\caption{Four Young diagram surgeries representing four invariant curves of Figure~\ref{fig:GKM1} containing fixed point $13$. The sites are shaded gray. The tangent weights $w$ of the corresponding invariant curves are also shown. Each pair of curves with the same weight span one of the two 2-dimensional pencils seen in Figure~\ref{fig:GKM1}.}
\label{fig:13_yd_surgeries}
\end{figure}

\subsection[Invariant curves ...]{Invariant curves for $\boldsymbol{\bowvar(\ttt{\fs 2\fs 3\fs 5\bs 3\bs 2\bs})}$}
In \cite{RS}, the fixed curves are depicted for the bow variety $\bowvar(\ttt{\bs 1\fs 2\fs 2\bs 2\bs 1\fs})$. This is Hanany--Witten isomorphic to the separated bow variety $\bowvar(\ttt{\fs 2\fs 3\fs 5\bs 3\bs 2\bs})$. This isomorphism is $\torus$-equivariant up to the reparametrization of $\torus$ described in Section~\ref{sec:HW}. The following computation will be done using tangent weights of fixed points in $\bowvar(\ttt{\fs 2\fs 3\fs 5\bs 3\bs 2\bs})$, which differs from \cite{RS} in the exponents of $\hb$. The invariant curves of this bow variety can be seen in Figure~\ref{fig:GKM3}.

\begin{figure}[th!]
\centering
\begin{tikzpicture}[scale=2.5]

\draw[] (.07,.07) -- (.93,.93);
\draw[] (.07,1.93) -- (.93,1.07);
\draw[] (1.07,0.93) -- (1.93,0.07);
\draw[] (1.07,1.07) -- (1.93,1.93);
\draw[] (0,.1) -- (0,1.9);
\draw[] (2,.1) -- (2,1.9);

\draw[] (-.07,.07) -- (-.5,.5);
\draw[] (-.07,-.07) -- (-.5,-.5);
\draw[] (0,-.1) -- (0,-.5);
\draw[] (.07,-.07) -- (.5,-.5);

\draw[] (2.07,.07) -- (2.5,.5);
\draw[] (2.07,-.07) -- (2.5,-.5);
\draw[] (1.93,-.07) -- (1.5,-.5);
\draw[] (2,-.1) -- (2,-.5);

\draw[] (-.07,2.07) -- (-.5,2.5);
\draw[] (-.07,1.93) -- (-.5,1.5);
\draw[] (0,2.1) -- (0,2.5);
\draw[] (.07,2.07) -- (.5,2.5);

\draw[] (2.07,2.07) -- (2.5,2.5);
\draw[] (2.07,1.93) -- (2.5,1.5);
\draw[] (1.93,2.07) -- (1.5,2.5);
\draw[] (2,2.1) -- (2,2.5);

\draw[] (1,1) -- (1,1.6); \draw[] (1,1) -- (1.1,1.6); \draw[] (1,1) -- (1.2,1.6); \draw[] (1,1) -- (1.3,1.6); \draw[] (1,1) -- (1.4,1.6); \draw[] (1,1) -- (1.5,1.6);
\draw[] (1,1) -- (1,.4);\draw[] (1,1) -- (.9,.4);\draw[] (1,1) -- (.8,.4);\draw[] (1,1) -- (.7,.4);\draw[] (1,1) -- (.6,.4);\draw[] (1,1) -- (.5,.4);

\node[fill,white,draw,circle,minimum size=.5cm,inner sep=0pt] at (0,0) {\footnotesize $1$};
\node[fill,white,draw,circle,minimum size=.5cm,inner sep=0pt] at (1,1) {$3$};
\node[fill,white,draw,circle,minimum size=.5cm,inner sep=0pt] at (2,0) {$2$};
\node[fill,white,draw,circle,minimum size=.5cm,inner sep=0pt] at (0,2) {$4$};
\node[fill,white,draw,circle,minimum size=.5cm,inner sep=0pt] at (2,2) {$5$};
\node[draw,circle,minimum size=.5cm,inner sep=0pt] at (0,0) {\footnotesize $1$};
\node[draw,circle,minimum size=.5cm,inner sep=0pt] at (1,1) {\footnotesize$3$};
\node[draw,circle,minimum size=.5cm,inner sep=0pt] at (2,0) {\footnotesize$2$};
\node[draw,circle,minimum size=.5cm,inner sep=0pt] at (0,2) {\footnotesize$4$};
\node[draw,circle,minimum size=.5cm,inner sep=0pt] at (2,2) {\footnotesize$5$};

\node at (1.38,1.22) {\scalebox{.6}[.6]{$\frac{{\ub}_2}{{\ub}_3}\hb$}};
\node at (.9,1.4) {\scalebox{.6}[.6]{$\frac{{\ub}_2}{{\ub}_3}\hb$}};
\node at (.67,1.22) {\scalebox{.6}[.6]{$\frac{{\ub}_1}{{\ub}_2}$}};
\node at (.65,.76) {\scalebox{.6}[.6]{$\frac{{\ub}_3}{{\ub}_2}$}};
\node at (1.07,.6) {\scalebox{.6}[.6]{$\frac{{\ub}_3}{{\ub}_2}$}};
\node at (1.4,.75) {\scalebox{.6}[.6]{$\frac{{\ub}_2}{{\ub}_1}\hb$}};

\node at (2.38,2.22) {\scalebox{.6}[.6]{$\frac{{\ub}_2}{{\ub}_3}\hb^2$}};
\node at (1.89,2.33) {\scalebox{.6}[.6]{$\frac{{\ub}_2}{{\ub}_3}\hb$}};
\node at (1.58,2.27) {\scalebox{.6}[.6]{$\frac{{\ub}_1}{{\ub}_3}\hb$}};
\node at (1.65,1.77) {\scalebox{.6}[.6]{$\frac{{\ub}_3}{{\ub}_2}\hb$}};
\node at (2.075,1.64) {\scalebox{.6}[.6]{$\frac{{\ub}_3}{{\ub}_1}$}};
\node at (2.39,1.73) {\scalebox{.6}[.6]{$\frac{{\ub}_3}{{\ub}_2}$}};

\node at (2.4,.23) {\scalebox{.6}[.6]{$\frac{{\ub}_2}{{\ub}_3}\hb$}};
\node at (1.93,.4) {\scalebox{.6}[.6]{$\frac{{\ub}_1}{{\ub}_3}$}};
\node at (1.59,.19) {\scalebox{.6}[.6]{$\frac{{\ub}_1}{{\ub}_2}\hb^{-1}$}};
\node at (1.63,-.24) {\scalebox{.6}[.6]{$\frac{{\ub}_3}{{\ub}_1}\hb$}};
\node at (2.12,-.4) {\scalebox{.6}[.6]{$\frac{{\ub}_2}{{\ub}_1}\hb^2 $}};
\node at (2.38,-0.27) {\scalebox{.6}[.6]{$\frac{{\ub}_3}{{\ub}_2}$}};

\node at (.32,.2) {\scalebox{.6}[.6]{$\frac{{\ub}_2}{{\ub}_3}$}};
\node at (-.07,.4) {\scalebox{.6}[.6]{$\frac{{\ub}_1}{{\ub}_3}$}};
\node at (-.36,.22) {\scalebox{.6}[.6]{$\frac{{\ub}_2}{{\ub}_3}\hb$}};
\node at (-0.37,-.25) {\scalebox{.6}[.6]{$\frac{{\ub}_3}{{\ub}_2}\hb$}};
\node at (.06,-.4) {\scalebox{.6}[.6]{$\frac{{\ub}_3}{{\ub}_2} $}};
\node at (.41,-0.25) {\scalebox{.6}[.6]{$\frac{{\ub}_3}{{\ub}_1}\hb$}};

\node at (.35,2.2) {\scalebox{.6}[.6]{$\frac{{\ub}_1}{{\ub}_3}\hb$}};
\node at (-.1,2.4) {\scalebox{.6}[.6]{$\frac{{\ub}_1}{{\ub}_2}\hb$}};
\node at (-.4,2.24) {\scalebox{.6}[.6]{$\frac{{\ub}_2}{{\ub}_3}\hb$}};
\node at (-0.35,1.78) {\scalebox{.6}[.6]{$\frac{{\ub}_3}{{\ub}_2}$}};
\node at (-.07,1.5) {\scalebox{.6}[.6]{$\frac{{\ub}_3}{{\ub}_1}$}};
\node at (.31,1.815) {\scalebox{.6}[.6]{$\frac{{\ub}_2}{{\ub}_1}$}};
\end{tikzpicture}
\caption{Illustration of $\torus$-fixed points and invariant curves (with their $\torus$-weights) of $\bowvar(\ttt{\fs 2\fs 3\fs 5\bs 3\bs 2\bs})$.} \label{fig:GKM3}
\end{figure}

\begin{figure}[th!]
\centering
\subcaptionbox{}[0.3\textwidth]
{
$
\begin{pmatrix}
1&0&1\\
0&1&0\\
1&0&1
\end{pmatrix}
$
}
\subcaptionbox{}[0.3\textwidth]
{
\begin{tikzpicture}[scale=0.35]
\draw (1, 0)--(4, 0); \draw (1, -1)--(4, -1); \draw (3, -2)--(4, -2);
\draw (1,0)--(1,-1); \draw (2,0)--(2,-1); \draw (3,0)--(3,-2); \draw (4,0)--(4,-2);

\draw (2,-3)--(4,-3); \draw (2,-4)--(4,-4);
\draw (2,-3)--(2,-4); \draw (3,-3)--(3,-4); \draw (4,-3)--(4,-4);

\draw (1, -5)--(4, -5); \draw (1, -6)--(4, -6); \draw (3, -7)--(4, -7);
\draw (1,-5)--(1,-6); \draw (2,-5)--(2,-6); \draw (3,-5)--(3,-7); \draw (4,-5)--(4,-7);
\end{tikzpicture}
}
\caption{The (a) BCT and (b) Young diagrams associated to the fixed point $3$ of Figure~\ref{fig:GKM3}.}
\label{fig:fixedpoint3}
\end{figure}

\begin{figure}[th!]
\centering
\hfill
\subcaptionbox*{(a) $w = \frac{\ub_1}{\ub_2}$}
{
\begin{tikzpicture}[scale=0.35]
\draw (1, 0)--(4, 0); \draw (1, -1)--(4, -1); \draw (3, -2)--(4, -2);
\draw (1,0)--(1,-1); \draw (2,0)--(2,-1); \draw (3,0)--(3,-2); \draw (4,0)--(4,-2);

\draw (2,-3)--(4,-3); \draw (2,-4)--(4,-4);
\draw (2,-3)--(2,-4); \draw (3,-3)--(3,-4); \draw (4,-3)--(4,-4);

\draw[fill=gray] (1, -1) rectangle (2, 0);

\draw[->, thick] (5, -2)--(7, -2);

\begin{scope}[xshift = 7cm]
\draw (2, 0)--(4, 0); \draw (2, -1)--(4, -1); \draw (3, -2)--(4, -2);
\draw (2,0)--(2,-1); \draw (3,0)--(3,-2); \draw (4,0)--(4,-2);

\draw (1,-3)--(4,-3); \draw (1,-4)--(4,-4);
\draw (1,-3)--(1,-4); \draw (2,-3)--(2,-4); \draw (3,-3)--(3,-4); \draw (4,-3)--(4,-4);

\draw[fill=gray] (1, -3) rectangle (2, -4);
\end{scope}

\end{tikzpicture}
}
\hfill
\subcaptionbox*{(b) $w = \frac{\ub_2}{\ub_1}\hb$}
{
\begin{tikzpicture}[scale=0.35]
\draw (1, 0)--(4, 0); \draw (1, -1)--(4, -1); \draw (3, -2)--(4, -2);
\draw (1,0)--(1,-1); \draw (2,0)--(2,-1); \draw (3,0)--(3,-2); \draw (4,0)--(4,-2);

\draw (2,-3)--(4,-3); \draw (2,-4)--(4,-4);
\draw (2,-3)--(2,-4); \draw (3,-3)--(3,-4); \draw (4,-3)--(4,-4);

\draw[fill=gray] (2, -3) rectangle (3, -4);

\draw[->, thick] (5, -2)--(7, -2);

\begin{scope}[xshift = 7cm]
\draw (1, 0)--(4, 0); \draw (1, -1)--(4, -1); \draw (2, -2)--(4, -2);
\draw (1,0)--(1,-1); \draw (2,0)--(2,-2); \draw (3,0)--(3,-2); \draw (4,0)--(4,-2);

\draw (3,-3)--(4,-3); \draw (3,-4)--(4,-4);
\draw (3,-3)--(3,-4); \draw (4,-3)--(4,-4);

\draw[fill=gray] (3, -2) rectangle (2, -1);
\end{scope}

\end{tikzpicture}
}
\hspace*{\fill}

\caption{The two Young diagram surgeries possible between the top two Young diagrams of Figure~\ref{fig:fixedpoint3}. The sites are shaded gray and $w$ is the tangent weight of the corresponding invariant curve.}
\label{fig:3_top_yd_surgeries}
\end{figure}
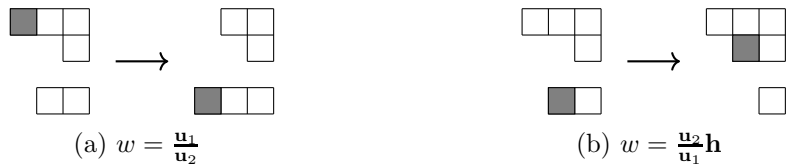

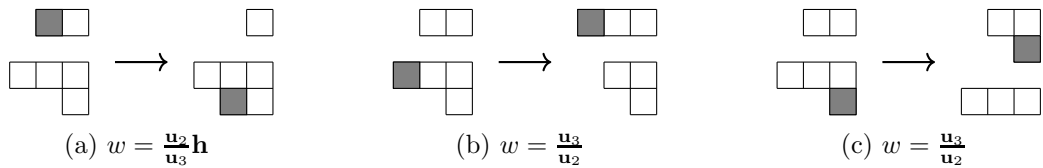
\begin{figure}[th!]
\centering
\hfill
\subcaptionbox*{(a) $w = \frac{\ub_2}{\ub_3}\hb$}
{
\begin{tikzpicture}[scale=0.35]
\draw (2,-3)--(4,-3); \draw (2,-4)--(4,-4);
\draw (2,-3)--(2,-4); \draw (3,-3)--(3,-4); \draw (4,-3)--(4,-4);

\draw (1, -5)--(4, -5); \draw (1, -6)--(4, -6); \draw (3, -7)--(4, -7);
\draw (1,-5)--(1,-6); \draw (2,-5)--(2,-6); \draw (3,-5)--(3,-7); \draw (4,-5)--(4,-7);

\draw[fill=gray] (2, -4) rectangle (3, -3);

\draw[->, thick] (5, -5)--(7, -5);

\begin{scope}[xshift = 7cm]
\draw (3,-3)--(4,-3); \draw (3,-4)--(4,-4);
\draw (3,-3)--(3,-4); \draw (4,-3)--(4,-4);

\draw (1, -5)--(4, -5); \draw (1, -6)--(4, -6); \draw (2, -7)--(4, -7);
\draw (1,-5)--(1,-6); \draw (2,-5)--(2,-7); \draw (3,-5)--(3,-7); \draw (4,-5)--(4,-7);

\draw[fill=gray] (2, -6) rectangle (3, -7);
\end{scope}

\end{tikzpicture}
}
\hfill
\subcaptionbox*{(b) $w = \frac{\ub_3}{\ub_2}$}
{
\begin{tikzpicture}[scale=0.35]
\draw (2,-3)--(4,-3); \draw (2,-4)--(4,-4);
\draw (2,-3)--(2,-4); \draw (3,-3)--(3,-4); \draw (4,-3)--(4,-4);

\draw (1, -5)--(4, -5); \draw (1, -6)--(4, -6); \draw (3, -7)--(4, -7);
\draw (1,-5)--(1,-6); \draw (2,-5)--(2,-6); \draw (3,-5)--(3,-7); \draw (4,-5)--(4,-7);

\draw[fill=gray] (1, -5) rectangle (2, -6);

\draw[->, thick] (5, -5)--(7, -5);

\begin{scope}[xshift = 7cm]
\draw (1,-3)--(4,-3); \draw (1,-4)--(4,-4);
\draw (1,-3)--(1,-4); \draw (2,-3)--(2,-4); \draw (3,-3)--(3,-4); \draw (4,-3)--(4,-4);

\draw (2, -5)--(4, -5); \draw (2, -6)--(4, -6); \draw (3, -7)--(4, -7);
\draw (2,-5)--(2,-6); \draw (3,-5)--(3,-7); \draw (4,-5)--(4,-7);

\draw[fill=gray] (2, -4) rectangle (1, -3);
\end{scope}

\end{tikzpicture}
}
\hfill
\subcaptionbox*{(c) $w = \frac{\ub_3}{\ub_2}$}
{
\begin{tikzpicture}[scale=0.35]
\draw (2,-3)--(4,-3); \draw (2,-4)--(4,-4);
\draw (2,-3)--(2,-4); \draw (3,-3)--(3,-4); \draw (4,-3)--(4,-4);

\draw (1, -5)--(4, -5); \draw (1, -6)--(4, -6); \draw (3, -7)--(4, -7);
\draw (1,-5)--(1,-6); \draw (2,-5)--(2,-6); \draw (3,-5)--(3,-7); \draw (4,-5)--(4,-7);

\draw[fill=gray] (3, -7) rectangle (4, -6);

\draw[->, thick] (5, -5)--(7, -5);

\begin{scope}[xshift = 7cm]
\draw (2,-3)--(4,-3); \draw (2,-4)--(4,-4); \draw (3,-5)--(4,-5);
\draw (2,-3)--(2,-4); \draw (3,-3)--(3,-5); \draw (4,-3)--(4,-5);

\draw (1,-6)--(4,-6); \draw(1,-7)--(4,-7);
\draw (1,-6)--(1,-7); \draw(2,-6)--(2,-7); \draw (3,-6)--(3,-7); \draw (4,-6)--(4,-7);

\draw[fill=gray] (4, -5) rectangle (3, -4);
\end{scope}

\end{tikzpicture}
}
\hspace*{\fill}

\caption{The three surgeries possible between the bottom two Young diagrams of Figure \ref{fig:fixedpoint3}. The sites are shaded gray and the tangent weights are included. Note that surgeries (b) and (c) share a tangent weight and have disjoint sites, and thus span a pencil of invariant curves.}
\label{fig:3_bottom_yd_surgeries}
\end{figure}

Let us consider the fixed point labeled $3$, which has BCT and Young diagrams in Figure \ref{fig:fixedpoint3}. By Theorem \ref{thm:weights}, $T_{3}\bowvar(\brane) = \frac{\ub_2}{\ub_3}\hb + \frac{\ub_3}{\ub_2} + \frac{\ub_1}{\ub_2} + \frac{\ub_2}{\ub_1}\hb + \frac{\ub_2}{\ub_3}\hb + \frac{\ub_3}{\ub_2}$. We can produce an invariant curve for each of these tangent weights.

First we check for type III curves. By Lemma \ref{lemma:nonsurgery}, there is one type III curve, since $c_3-c_2=1$. This curve has tangent weight $\frac{\ub_2}{\ub_3}\hb$. If there are other curves with this weight, they must all be type I and, together with the type III curve, will span a pencil.

We now seek to construct five more invariant curves that come from Young diagram surgeries. To do this, we will consider surgeries between each pair of Young diagrams separately. There are two surgeries between the top two Young diagrams, three surgeries between the bottom two Young diagrams, and zero surgeries between the top and bottom Young diagrams. These five surgeries are depicted in Figures \ref{fig:3_top_yd_surgeries} and~\ref{fig:3_bottom_yd_surgeries}.

With these six curves described, we now build pencils of curves. There is a pencil of curves with weight $\frac{\ub_3}{\ub_2}$, which is spanned by a type I curve and type II curve, as seen in Figure \ref{fig:3_bottom_yd_surgeries}. There is also a pencil of curves with weight $\frac{\ub_2}{\ub_3}\hb$, spanned by a type I curve from Figure \ref{fig:3_bottom_yd_surgeries} and a type III curve. No other curves share weights, so these are the only multidimensional pencils of invariant curves.

\subsection*{Acknowledgements}

We would like to thank R.~Rim\'anyi for helpful discussions on the topic.
We would also like to thank the referees, whose comments have resulted in numerous corrections and improvements.

\pdfbookmark[1]{References}{ref}
\LastPageEnding

\end{document}